\documentclass[10pt]{amsart}
\usepackage{a4wide,color}
\usepackage{amsmath,amssymb}
\usepackage[active]{srcltx}
\usepackage{ulem}

\allowdisplaybreaks

\newcommand{\R}{{\mathbb R}}

\newtheorem{theorem}{Theorem}
\newtheorem{lemma}[theorem]{Lemma}
\newtheorem{proposition}[theorem]{Proposition}

\newtheorem{corollary}[theorem]{Corollary}

\let \ka=\kappa

\let \Om=\Omega

\begin{document}

\title[Boltzmann equation]{H\"older regularity of solutions of the steady Boltzmann equation with soft potentials}

\author{Kung-Chien Wu}
\address{Kung-Chien Wu, Department of Mathematics, National Cheng Kung
University, Tainan, Taiwan and National Center for Theoretical Sciences,
National Taiwan University, Taipei, Taiwan}
\email{kungchienwu@gmail.com}

\author{Kuan-Hsiang Wang}
\address{Kuan-Hsiang Wang, Department of Mathematics, National Kaohsiung Normal
University, Kaohsiung, Taiwan}
\email{khwang0511@gmail.com}

\begin{abstract}
We consider the H\"older regularity of solutions to the steady Boltzmann equation with in-flow boundary condition
in bounded and strictly convex domains $\Om\subset\R^{3}$ for gases with cutoff soft potential $(-3<\gamma<0)$. We prove that there is a unique solution with a bounded $L^{\infty}$ norm in space and velocity. This solution is H\"older continuous, and it's order depends not only on the regularity of the incoming boundary data, but also on the potential power $\gamma$. The result for modulated soft potential case $-2<\gamma<0$ is similar to hard potential case $(0\leq\gamma<1)$ since we have $C^{1}$ velocity regularity from collision part. However, we observe that for very soft potential case $(-3<\gamma\leq -2)$, the regularity in velocity obtained by the collision part is lower (H\"older only), but the boundary regularity still can transfer to solution (in both space and velocity) by transport and collision part under the restriction of $\gamma$.
\end{abstract}

\maketitle

\section{Introduction}

\subsection{The model}
In this paper, we study the regularity of solutions to the following steady Boltzmann equation with cutoff soft potential:
\begin{equation}\label{SB}
v\cdot\nabla_x F=\frac{1}{\kappa_n}Q(F, F), \quad (x, v) \in \Omega \times \mathbb{R}^3,
\end{equation}
where $\kappa_n$ is the Knudsen number.
In \eqref{SB},  $F(x,v)\geq 0$ denotes the density distribution function for gas particles with position $x=(x_{1},x_{2},x_{3})\in \Omega \subset\mathbb{R}^3$ and
microscopic velocity $v=(v_{1},v_{2},v_{3})\in {\mathbb{R}}^{3}$. The left-hand side of \eqref{SB} models the transport of particles and
the operator on the right-hand side models the effect of collisions during
the transport,
\begin{equation*}\label{Co}
Q(F, G)(v)=\int_{\mathbb{R}^3\times \mathbb{S}^2}|v-u|^{\gamma}b(\vartheta)\left[F(v')G(u')-F(v)G(u)\right]d\omega du,
\end{equation*}
where $\gamma=1$ denotes hard sphere case, $0\leq \gamma<1$ denotes hard potential case, $-3<\gamma<0$ denotes soft potential case, the post-collisional velocities of particles satisfy
\begin{equation*}
v'=v-\left[(v-u)\cdot\omega\right]\omega,\quad u'=u+\left[(v-u)\cdot\omega\right]\omega, \quad\omega\in\mathbb{S}^2,
\end{equation*}
and $\vartheta$ is defined by
\begin{equation*}
\cos\vartheta=\frac{\left|(v-u)\cdot\omega \right|}{|v-u|}.
\end{equation*}
Throughout this paper, we consider the soft potential ($-3<\gamma<0$),  $b(\vartheta )$
satisfies the Grad cutoff assumption
\begin{equation*}
0<b(\vartheta )\leq C \cos \vartheta
\end{equation*}%
for some constant $C>0$, and $\kappa_n=1$ in our problem.

We assume that $\Omega=\{x\in \mathbb{R}^3 : \xi(x)<0\}$ is connected and bounded with $\xi(x)$ be a smooth function and $\nabla \xi(x)\neq 0$ at the boundary $\xi(x)=0$.
The outward normal vector at $\partial\Omega$ is given by $n(x)= \frac{\nabla \xi(x)}{|\nabla \xi(x)|}$. We say that $\Omega$ is strictly convex if there is $c_\xi>0$ such that $\sum_{ij}\partial_{ij}\xi(x)\zeta^i\zeta^j\geq c_\xi|\zeta|^2$ for all $x\in \overline{\Omega}$ and all $\zeta\in\mathbb{R}^3$.
We denote the phase boundary of the phase space $\Omega\times\mathbb{R}^3$ as $\Gamma=\partial\Omega\times\mathbb{R}^3$, and split it into three parts, an outgoing boundary $\Gamma_+$, an incoming boundary $\Gamma_-$, and a singular boundary $\Gamma_0$ for grazing velocities:
\begin{align*}
\Gamma_{+}&=\{(x,v)\in\Gamma : n(x)\cdot v>0\};\\
\Gamma_-&=\{(x,v)\in\Gamma : n(x)\cdot v<0\};\\
\Gamma_0&=\{(x,v)\in\Gamma : n(x)\cdot v=0\}.
\end{align*}

Let $M_{T, w, \rho}$ be a local Maxwellian with density $\rho$, bulk velocity $w$, and temperature $T$ and be defined as
\begin{equation*}
M_{T, w, \rho}(x, v):=\rho(2\pi T)^{-\frac32}e^{-\frac{|v-w|^2}{2T}}.
\end{equation*}
Moreover, $M(v):=M_{1, 0, 1}(v)=(2\pi)^{-\frac32}e^{-\frac{|v|^2}{2}}$ is a global Maxwellian.
In this paper, the steady Boltzmann equation \eqref{SB} is given with the in-flow boundary condition and $\kappa_n=1$ as follows:
\begin{equation}\label{SBE}
\begin{cases}
v\cdot\nabla_x F =Q(F, F),& \text{ in  } \Omega \times \mathbb{R}^3,\\
F= \widetilde M,&  \text{ on } \Gamma_-,
\end{cases}
\end{equation}
where $\widetilde M(x, v)\geq 0$ is a given function and is taken as a small perturbation of the global Maxwellian $M(v)$. One example is the local Maxwellian $M_{T(x), 0, 1}$, which describes a non-isothermal wall temperature for the motion of rarefied gases in a bounded domain.

It is well known that the global Maxwellian $M(v)$ is a (steady-state) solution to \eqref{SB}.
Therefore, it is natural to consider \eqref{SBE} around the global Maxwellian $M(v)$ with the standard perturbation $f=f(x,v)$ to $M(v)$ as
\begin{equation*}
F=M+M^{\frac12}f.
\end{equation*}
After substituting $F$ into \eqref{SBE}, the problem for $f$ reads as
\begin{equation}\label{SBEf}
\begin{cases}
v\cdot\nabla_x f +Lf=N(f,f), &  \text{ in  } \Omega \times \mathbb{R}^3,\\
f = r, & \text{ on } \Gamma_-,
\end{cases}
\end{equation}
where $r=(\widetilde M-M)M^{-\frac12}$, $L$ is the linearized collision operator defined as
\begin{align*}
Lf&=-M^{-\frac12}\left[Q(M, M^{\frac12}f)+Q(M^{\frac12}f, M)\right]\\
\ &=\nu(v)f-Kf,
\end{align*}
and $N$ is the nonlinear operator defined as
\begin{align*}
N(f, g)(v)& =M^{-\frac12}(v)Q(M^{\frac12}f, M^{\frac12}g)(v) \\
\ & =\int_{\mathbb{R}^3\times \mathbb{S}^2}|v-u|^\gamma b(\vartheta) M^{\frac12}(u)\left(f(v')g(u')-f(v)g(u)\right) d\omega du \\
\ &=:N_{+}(f, g)(v)-N_{-}(f,g)(v).
\end{align*}
Here $\nu$ is a multiplicative operator, $K$ is an integral operator, and
$N_+$ and $N_-$ correspond to the gain part and loss part in $Q$ respectively. Moreover, the loss part $N_{-}$ can be rewritten as $N_-(f, g)(v)=f(v)\varphi(g)(v)$, where $\varphi(g)$ is defined as
\begin{equation}\label{Psi}
\varphi(g)(v)=\int_{\mathbb{R}^3\times\mathbb{S}^2}|v-u|^\gamma b(\vartheta) M^{\frac12}(u)g(u) d\omega du.
\end{equation}

\subsection{Notations}
Let us define some notations used in this paper.
Denote the $L^p(\Omega\times\mathbb{R}^3)$-norm by $\|\cdot\|_p$ for $p\in [1, \infty]$. For the phase boundary integration, we denote $d\Gamma =|n(x)\cdot v|dS_x dv$ with the surface measure $dS_x$ and denote the $L^p(\Gamma)$-norm by
$$|f|_\infty=\sup_{(x, v)\in \Gamma} |f(x,v)|$$
and
$$|f|_p=\left(\int_{\Gamma} |f(x, v)|^p d\Gamma\right)^{\frac1p}$$
 for $1\leq p<\infty$. Furthermore, we denote the $L^p(\Gamma_\pm)$-norm by $|f|_{p, \pm}=|f\mathbf{1}_{\Gamma_\pm}|_{p}$.
We use $C_a, C_b, \ldots$ to denote the positive constants depending on $a, b, \ldots$, respectively.
For simplicity of notations, hereafter, we abbreviate \textquotedblleft {\ $%
\leq C$} \textquotedblright \ to \textquotedblleft {\ $\lesssim $ }%
\textquotedblright , where $C$ is a positive constant depending only on
fixed numbers, and the notation $A \sim B$ denotes the statement that $A\lesssim B\lesssim A$.
$a^\pm$ denotes $a\pm\epsilon$ with $\epsilon>0$ arbitrary small.
Let $v_1, v_2\in\mathbb{R}^3$, the notation $v_{1,2}$ is defined as $ v_{1,2}=v_1$ if $|v_1|\geq |v_2|$ and $ v_{1,2}=v_2$ if $|v_1|< |v_2|$.
Denote $\langle v\rangle^s=(1+|v|^2)^{\frac s2}$ for $s\in\mathbb{R}$.
We define a function presenting the distance between $x\in\Omega$ and $\partial\Omega$ as $d_x=\inf_{y\in\partial\Omega}{|x-y|}$, and denote $d_{(x, y)}=\min\{d_x, d_y\}$.
We define a velocity weight function
\begin{equation*}\label{vw}
w(v)=w_{\tau, \kappa}(v)=\langle v\rangle^\tau e^{\kappa |v|^2},
\end{equation*}
where $\tau>0$ and $\ka >0$ are given constants. 

\subsection{Review of previous work and main result}

The regularity and singularity of solutions of the Boltzmann equation in bounded domains for various boundary conditions have been an interesting and challenging question in theoretical studies. The  Boltzmann equation reads as
\begin{equation*}
\partial_t F(t, x, v)+v\cdot \nabla_xF(t, x, v)=Q(F, F)(t, x, v).
\end{equation*}
For a general setting of the initial boundary value problems, we refer the reader to \cite{Ce, CIP, Gl}.
In the existing literature, there are a few papers concerning the regularity and continuity of solutions of the Boltzmann equation in the bounded domain.
In \cite{G10}, Guo studied the time decay and continuity of solutions of the Boltzmann equation with hard sphere and hard potential in bounded domains for four basic types of boundary conditions: in-flow, bounce-back reflection, specular reflection, and diffuse reflection. In convex domains, Guo showed that the continuity of these solutions away from the grazing set $\Gamma_0$.
After that, Kim \cite{K11} demonstrated that singularity (discontinuity) of the Boltzmann solutions does occur  in non-convex domains for in-flow, diffuse reflection, and bounce-back boundary conditions. Furthermore, the author showed that discontinuity propagates only along the forward characteristics emanating from the grazing set.

In \cite{GKTT17}, Guo-Kim-Tonon-Trescases studied the regularity problem for solutions of the Boltzmann equation in convex domains.
The authors showed the non-existence of second order spatial normal derivative of the Boltzmann solutions at the boundary even in convex domains or symmetric domains.
Furthermore, they constructed weighted $C^1$ solutions away from the grazing set for the specular, bounce-back, and diffuse reflection boundary conditions.
However, the weighted $C^1$ norm grows without bound as the time increases; therefore such estimates do not give any information on the stationary Boltzmann solution through large time behavior.
Based on the observation above, Chen \cite{Ch} investigated the regularity of stationary solutions to the linearized Boltzmann equation in convex domains.
Chen proved the stationary linearized Boltzmann solution is H\"older continuous with order $\frac12^-$ away from the boundary provided the incoming data have the same regularity.
Very recently, Chen and Kim \cite{CK22} considered the steady Boltzmann equation \eqref{SB} with the non-isothermal diffuse reflection boundary condition in strictly convex domains for fixed $\kappa_n \sim 1$. The authors constructed a locally $C^{1, \beta}$ solutions away from the grazing boundary, for any $0<\beta<1$, when the non-isothermal wall temperature $T(x)\in C^2(\partial\Omega)$ is a small fluctuation around any constant temperature.
For more related results, we refer the reader to \cite{Cao, CKL, CHK, DHWZ, GKTT16, KL, LY} and references therein.

We notice that there seems to be rarely concerned with the regularity of the Boltzmann solution for the soft potential case in the existing literature.
Based on the fact, the main goal of this paper is to study the H\"older regularity of the stationary Boltzmann solutions on the effects of the modulated soft potential and the boundary regularity.
Therefore, we assume that $\kappa_n=1$ in this paper for simplicity of presentation.
For more related problem involving the  Knudsen number, we refer the reader to \cite{CLT, EGM, EGKM13, EGKM18, JZ, WZW} and references therein.

Our main result is as follows.
\begin{theorem}\label{MT}
Let $-3<\gamma<0$ and let $\tau>3+|\gamma|$ and $0<\ka <\frac18$.
There exist $\delta_0>0$ and $C_0>0$ such that if
\begin{equation*}
| w_{\tau, \ka} r|_{\infty, -}\leq \delta_0,
\end{equation*}
then there exists a unique nonnegative solution $F_*=M+M^{\frac12}f_*$ to the steady problem \eqref{SBE} satisfying the estimate
\begin{equation*}
\| w_{\tau, \ka} f_*\|_\infty+| w_{\tau, \ka} f_*|_\infty \leq C_0| w_{\tau, \ka} r|_{\infty, -}.
\end{equation*}
Moreover, if $\Omega$ is strictly convex and $r$ is continuous on $\Gamma_-$, then $F_*$ is continuous on $\overline{\Omega}\times \mathbb{R}^3\setminus \Gamma_0$.
In addition, when $\Omega$ is strictly convex, assume further that there exist $m>0$ and $0<\beta\leq 1$ such that
\begin{equation}\label{rb}
|r(y_1, v_1)-r(y_2, v_2)|\leq m(|y_1-y_2|+|v_1-v_2|)^\beta
\end{equation}
for every $(y_1, v_1), (y_2, v_2)\in \Gamma_-$.
Let $0<\alpha< \min\{1, 3+\gamma\}$. 
Then the solution $F_*$ is H\"older continuous locally in $\Omega$ with order $\min\{\alpha, \beta\}$. More precisely, for every $x_1, x_2\in\Omega$ and $v_1, v_2\in\mathbb{R}^3$,
\begin{align*}\label{fH}
|f_*(x_1, v_1)-f_*(x_2, v_2)|\leq \widetilde C_0m \left(d_{(x_1, x_2)}^{-2}(|x_1-x_2|+|v_1-v_2|)\right)^{\min\{\alpha, \beta\}}
\end{align*}
for some constant $\widetilde C_0>0$ depending on $\alpha, \beta, \gamma$, and $\Omega$.
\end{theorem}

In order to prove our main theorem, we rewrite the equation (\ref{SBEf}) as
\begin{equation*}
\begin{cases}
v\cdot\nabla_x f+Lf+\varphi(f) f=N_{+}(f,f), & \text{ in } \Omega\times\mathbb{R}^3,\\
f=r, & \text{ on } \Gamma_-.
\end{cases}
\end{equation*}

Since we consider the problem in perturbative region, one may expect the regularity of solutions to \eqref{SBEf} could be determined by the regularity estimates of the stationary linearized Boltzmann equation:
 \begin{equation}\label{LSBE-0}
\begin{cases}
v\cdot\nabla_x f+Lf=0, & \text{ in } \Omega\times\mathbb{R}^3, \\
f=r, & \text{ on } \Gamma_-.
\end{cases}
\end{equation}
Therefore, the approach for our problem is by constructing the H\"older regularity of the linearized problem, then design a suitable iterative scheme to solve the nonlinear problem. Precisely, we design the iteration as follows:
\begin{equation}\label{is}
\begin{cases}
v\cdot\nabla_x f^{j+1}+Lf^{j+1}+\varphi(f^j) f^{j+1}=N_+(f^j, f^j), & \text{ in } \Omega\times\mathbb{R}^3,\\
f^{j+1}=r, & \text{ on } \Gamma_-.
\end{cases}
\end{equation}
In order to proceed, we need to study the existence and continuity of the following approximate steady Boltzmann equation:
\begin{equation}\label{approx-1}
\begin{cases}
v\cdot\nabla_x f+Lf+\varphi(g_1) f=g_2, & \text{ in } \Omega\times\mathbb{R}^3,\\
f=r, & \text{ on } \Gamma_-,
\end{cases}
\end{equation}
where $g_1, g_2$ are given. To achieve this goal, modify the approach in \cite{DHWZ}, we turn to consider the solvability of the following approximate problem:
\begin{equation}\label{approx-2}
\begin{cases}
\epsilon f^{\epsilon}+v\cdot\nabla_x f^\epsilon+\nu(v)f^\epsilon-\lambda Kf^\epsilon+\varphi(g_1) f^\epsilon=g_2, & \text{ in } \Omega\times\mathbb{R}^3,\\
f^\epsilon=r, & \text{ on } \Gamma_-.
\end{cases}
\end{equation}
Following the procedures of \cite[Lemma 3.5]{DHWZ}, we only obtain the existence of solutions of \eqref{approx-2} for $0\leq \lambda \leq \lambda^*$ for some $\lambda^*<1$ due to the extra term $\varphi(g_1)f^\epsilon$.
To obtain the existence for $\lambda=1$, we construct a velocity-weighted $L^2$ estimate independent of $\epsilon$ and $\lambda$, which is also used to obtain a unique solution constructed from the $L^\infty$ convergence of $f^\epsilon$ as $\epsilon \to 0$.

Next, we turn to study the H\"older regularity of the linearized Boltzmann equation (\ref{LSBE-0}). For each $x\in\overline\Omega$ and $v\neq 0$, we define the backward exit time:
\begin{equation*}
t_b(x,v)=\sup\{t>0 : x-sv \in\overline\Omega \text{ for } 0\leq s\leq t\}
\end{equation*}
and hence $x-t_b(x,v)v\in \partial\Omega$. We also define the backward exit position in $\partial\Omega$:
\begin{equation*}
x_b(x,v)=x-t_b(x,v)v\in\partial\Omega.
\end{equation*}
Then we always have $v\cdot n(x_b(x,v))< 0$ whenever $x\in \Omega$.
Note that for any $(x, v)$, we use $t_b(x, v)$ whenever it is well-defined.
To analyze the linearized problem, from the method of characteristics, we get the integral form of the stationary linearized Boltzmann solution:
\begin{equation*}\label{IE}
f(x,v)=f(x_b(x, v), v)e^{-\nu(v)t_b(x,v)}+\int_0^{t_b(x,v)}e^{-\nu(v)s}Kf(x-vs, v)ds
\end{equation*}
for every $x\in\Omega$ and $v\in\mathbb{R}^3\setminus\{0\}$. To obtain the H\"older regularity, we follow the pioneering work of Vidav \cite{V} to iterate $Kf$ once more to get
\begin{align*}
\ &f(x,v)\\
=&\ f\left(x_b(x,v), v\right)e^{-\nu(v)t_b(x,v)}\\
\ &+\int_0^{t_b(x,v)}e^{-\nu(v)s} \int_{\mathbb{R}^3}k(v,\eta)f\left(x_b(x-vs, \eta), \eta\right)e^{-\nu(\eta)t_b(x-vs, \eta)} d\eta ds\\
\ &+\int_0^{t_b(x,v)}e^{-\nu(v)s}\int_{\mathbb{R}^3} k(v,\eta)\int_0^{t_b(x-vs, \eta)}e^{-\nu(\eta)\tau}Kf(x-vs-\eta\tau, \eta)d\tau d\eta ds\\
=:&\ F_1(f)(x,v)+F_2(f)(x,v)+F_3(f)(x,v).
\end{align*}

Our estimate is based on \cite{Ch} for the linearized Boltzmann equation with hard potential. The velocity regularity of $F_i(f)$, $1\leq i\leq 3$, come directly from the velocity regularity of $r$ on $\Gamma_-$ and the H\"older regularity of the integral operator $K$. The space regularity of $F_{1}(f)$ and $F_{2}(f)$ come from the space regularity of $r$ on $\Gamma_-$. For the space regularity of $F_3(f)$, the key observation is that the regularity in velocity of $K(f)$ can be partially transferred to space due to transport and collision.

Based on the approach in \cite{Ch}, one need the estimate of $\nabla_{v}K$ to obtain the velocity regularity results of $F_{2}(f)$ and $F_{3}(f)$, and the space regularity result of $F_{3}(f)$. Therefore, to generalize the result \cite{Ch} from hard potential to soft potential in the linearized level, the main difficulty is that we do not have gradient estimate of the integral operator $K$ for $-3<\gamma\leq-2$ due to the singularity of the kernel function $k(v,\eta)$ . Indeed, one can show that $\nabla_{v}Kf$ is uniformly bounded in velocity for $f\in L^\infty(\Omega\times\mathbb{R}^3)$ if $\gamma>-2$.
For $-3<\gamma\leq-2$, instead of the gradient estimate, one can only show that the H\"older regularity of $Kf$ in velocity (see Lemma \ref{kH}).
Our main contribution here is that, for the soft potential case $-2<\gamma<0$, we get similar result for hard potential case. However, for the very soft potential case $-3<\gamma\leq -2$, based on the H\"older regularity of $Kf$, one can prove that the boundary regularity can transfer to the solution but with the restriction of the velocity regularity of $K$.

In our linearized estimate (see Theorem \ref{T1}), we provide two different kind of singularity behaviors (see (\ref{T1v}) and (\ref{T1x})) in our result. The estimate (\ref{T1x}) presents that the singularity occurs when the position is close to the boundary, while the estimate (\ref{T1v}) presents that part of the space singularity can be replaced as the velocity singularity when velocity goes to zero. The later one (see (\ref{T1x})) is the standard presentation for the boundary value problem. However, the former one (see (\ref{T1v})) is necessary for our nonlinear estimate and we will explain it below.

In order to achieve the H\"older regularity estimate of the nonlinear problem, one need to consider the approximate steady Boltzmann equation (\ref{approx-1}) with $g_{1}=g_{2}$. If we apply the estimate (\ref{T1x}) to the approximate problem, the space singularity will increase and therefore it is unable to solve the nonlinear problem. However, if one apply estimate (\ref{T1v}), then the integral operator $K$ will absorb the velocity singularity and there one can control the singularity.
With above preparation, one can design the iteration \eqref{is} to solve the nonlinear problem.

\subsection{Plan of the paper}
In Section \ref{pre}, we present some basic estimates for the operators $L, \nu, K$, and $N$, and also some basic properties of the backward exit time $t_b$ and the backward exit position $x_b$. The existence of solutions to the approximate problem \eqref{approx-1} will be proved in Section \ref{app-sol}. In Section \ref{linear}, we study the H\"older regularity of the linearized problem, and the H\"older regularity of the approximate problem will be discussed in Section \ref{app-holder}. Finally, the proof of Theorem \ref{MT} will be presented in section \ref{Pf-Main}.


\section{Preliminaries}\label{pre}

\subsection{Basic estimates for $L$ and $N$}

It is well known that the null space of the linearized collision operator $L$ is a five-dimensional vector space with the orthonormal basis $\{\phi_i\}_{i=0}^4$, where
\begin{equation*}
\text{Ker$(L)$} = \{\phi_0, \phi_i, \phi_4\}=\left\{M^{\frac12}, v_iM^{\frac12}, \frac{|v|^2-3}{\sqrt{6}}M^{\frac12}\right\}, \ i=1, 2,3.
\end{equation*}
Based on this property, we define the macro-micro decomposition for each $f\in L_v^2$: let $Pf$ (the macroscopic part) be the orthogonal projection of $f$ onto Ker$(L)$, i.e., $Pf=\sum_{i=0}^4\langle f, \phi_i\rangle\phi_i$, and $(I-P)f=f-Pf$ be the microscopic part of $f$. Furthermore, from \cite[Lemma 3]{G03}, there is a constant $C_0>0$ such that
\begin{equation}\label{LP1}
\langle Lf, f\rangle\geq C_0\int_{\mathbb{R}^3} \nu(v)|(I-P)f|^2 dv.
\end{equation}

The linearized collision operator $L$ consists of a multiplicative operator $\nu$ and an integral operator $K$:
\begin{equation*}
Lf=\nu(v)f - Kf,
\end{equation*}
where
\begin{equation*}
\nu(v)=\int_{\mathbb{R}^3\times\mathbb{S}^2}|v-u|^\gamma b(\vartheta) M(u) d\omega du
\end{equation*}
and
\begin{equation*}
Kf=-K_1f+K_2f
\end{equation*}
is defined as \cite{Gr}:
\begin{align*}
K_1f(v)&=\int_{\mathbb{R}^3\times \mathbb{S}^2}|v-u|^\gamma b(\vartheta) M^{\frac12}(v)M^{\frac12}(u)f(u) d\omega du;\\
K_2f(v)&=\int_{\mathbb{R}^3\times \mathbb{S}^2}|v-u|^\gamma b(\vartheta) M^{\frac12}(u)\left( M^{\frac12}(v')f(u')+ M^{\frac12}(u')f(v') \right) d\omega du.
\end{align*}

We next present a number of properties and estimates of the multiplicative operator $\nu(v)$ and the integral operator $K$, which can be found in  \cite{Ca, CH, GP, Gr, LLWW}.
For the multiplicative operator $\nu(v)$, there exist $0<\nu_1<\nu_2$ such that
\begin{equation*}
\nu_1\langle v\rangle^\gamma\leq \nu(v)\leq \nu_2\langle v\rangle^\gamma,
\end{equation*}
and for each multi-index $\alpha$,
\begin{equation*}
\left|\partial_v^\alpha \nu(v)\right|\lesssim \langle v\rangle^{\gamma-|\alpha|}.
\end{equation*}
For the integral operator $K$,
\begin{align*}
Kf &=-K_1f+K_2f =-\int_{\mathbb{R}^3}k_1(v, \eta)f(\eta)d\eta+\int_{\mathbb{R}^3}k_2(v, \eta)f(\eta)d\eta \\
\ &=:\int_{\mathbb{R}^3} k(v ,\eta) f(\eta) d\eta,
\end{align*}
the kernels $k_1$ and $k_2$ are symmetric, i.e., $k_i(v, \eta)=k_i(\eta, v)$, and satisfy
\begin{align*}
k_1(v, \eta)&\leq C |v-\eta|^\gamma e^{-\frac{|v|^2+|\eta|^2}{4}};\\
k_2(v, \eta)&\leq C_{\varepsilon}(1+|v|+|\eta|)^{\gamma-1} A(|v-\eta|) e^{-\frac{1-\varepsilon}{8}\left(|v-\eta|^2+\frac{(|v|^2-|\eta|^2)^2}{|v-\eta|^2}\right)}
\end{align*}
for any $0<\varepsilon<1$, together with
\begin{equation}\label{A}
A(|v-\eta|)=
\begin{cases}
|v-\eta|^{-1}, &\text{if } -1<\gamma<0,\\
|v-\eta|^{-1}\left|\ln|v-\eta|\right|, &\text{if } \gamma= -1,\\
|v-\eta|^{\gamma}, &\text{if } -3<\gamma<-1,
\end{cases}
\end{equation}
and their derivatives have the following estimates
\begin{align*}
\ &\left|\nabla_vk_1(v, \eta)\right|, \left|\nabla_\eta k_1(v, \eta)\right| \lesssim |v-\eta|^{\gamma-1}e^{-\frac{|v|^2+|\eta|^2}{4}};\\
\ &\left|\nabla_vk_2(v, \eta)\right|, \left|\nabla_\eta k_2(v, \eta)\right| \\
\lesssim &\ \frac{1+|v|}{(1+|v|+|\eta|)^{1-\gamma}}\frac{ A(|v-\eta|)}{|v-\eta|} e^{-\frac{1-\varepsilon}{8}\left(|v-\eta|^2+\frac{(|v|^2-|\eta|^2)^2}{|v-\eta|^2}\right)}.
\end{align*}

We need the following lemma.
\begin{lemma}
Let $-3<\widetilde\gamma<0$, $\tau\in\mathbb{R}$, and $\alpha\geq 0$. If $\alpha<3+\widetilde\gamma$, then
\begin{align}\label{A-1}
\int_{\mathbb{R}^3}  |v-\eta|^{\widetilde\gamma}e^{-\frac{1}{16}\left(|v-\eta|^2+\frac{(|v|^2-|\eta|^2)^2}{|v-\eta|^2}\right)} \langle \eta\rangle^\tau |\eta|^{-\alpha} d\eta \lesssim \langle v\rangle^{\tau-\alpha-1}
\end{align}
and
\begin{align}\label{A-2}
\int_{\mathbb{R}^3} |u|^{\widetilde\gamma}e^{-\frac{\kappa}{4}\left(|u|^2+\frac{(|u+v_1|^2-|v_1|^2)^2}{|u|^2} \right)} |v_2+u|^{-\alpha} du \lesssim \langle v_1\rangle^{-1}\langle v_2\rangle^{-\alpha}
\end{align}
for any $0<\kappa< \frac14$ and for any $v, v_1, v_2\in\mathbb{R}^3$. 
\end{lemma}

\begin{proof}
For \eqref{A-1}, we divide the proof into three cases.\\ 
$(i)$ $|\eta|\geq 2|v|$. In this case, one has $|v-\eta|\sim |\eta|$. Thus,  
\begin{align*}
\ &\int_{|\eta|\geq 2|v|} |v-\eta|^{\widetilde\gamma}e^{-\frac{1}{16}\left(|v-\eta|^2+\frac{(|v|^2-|\eta|^2)^2}{|v-\eta|^2}\right)} \langle \eta\rangle^\tau |\eta|^{-\alpha} d\eta\\
\lesssim &\int_{|\eta|\geq 2|v|} |\eta|^{\widetilde\gamma-\alpha}\langle \eta \rangle^{\tau} e^{-C|\eta|^2}  d\eta\\
\lesssim &\int_{0}^\pi\int_0^{2\pi}\int_{2|v|}^\infty \rho^{\widetilde\gamma-\alpha+2}(1+\rho)^\tau  e^{-C\rho^2} \sin\phi \ d\rho d\theta d\phi \\
\lesssim &\ e^{-\frac C2 |v|^2}.
\end{align*}
$(ii)$ $\frac{|v|}{2}<|\eta|<2|v|$. By using the spherical coordinates around $v$, we have
\begin{align*}
\ &\int_{\frac{|v|}{2}<|\eta|< 2|v|} |v-\eta|^{\widetilde\gamma}e^{-\frac{1}{16}\left(|v-\eta|^2+\frac{(|v|^2-|\eta|^2)^2}{|v-\eta|^2}\right)} \langle \eta\rangle^\tau |\eta|^{-\alpha} d\eta   \\
\lesssim &\ \langle v \rangle^{\tau} |v|^{-\alpha} \int_{\frac{|v|}{2}<|\eta|< 2|v|}|v-\eta|^{\widetilde\gamma} e^{-\frac{1}{16}\left(|v-\eta|^2+\frac{(|v|^2-|\eta|^2)^2}{|v-\eta|^2}\right)} d\eta\\
\lesssim &\ \langle v \rangle^{\tau} |v|^{-\alpha}\int_{0<|\eta|<3|v|} |\eta|^{\widetilde\gamma}e^{-\frac{1}{16}\left(|\eta|^2+\frac{(2v\cdot \eta+|\eta|^2)^2}{|\eta|^2}\right)} d\eta  \\
\lesssim &\ \langle v \rangle^{\tau} |v|^{-\alpha} \int_0^{3|v|} \rho^{\widetilde\gamma+2}e^{-\frac{1}{16}\rho^2}\left(\int_0^\pi e^{-\frac{1}{16}(2|v|\cos\phi+\rho)^2}\sin\phi d\phi \right) d\rho  \\
\lesssim &\ \langle v \rangle^{\tau-\alpha-1}.
\end{align*}
$(iii)$ $|\eta|\leq\frac{|v|}{2}$. One has  $|v-\eta|\sim |v|$. Hence 
\begin{align*}
\ &\int_{|\eta|\leq\frac{|v|}{2}}|v-\eta|^{\widetilde\gamma}e^{-\frac{1}{16}\left(|v-\eta|^2+\frac{(|v|^2-|\eta|^2)^2}{|v-\eta|^2}\right)} \langle \eta\rangle^\tau |\eta|^{-\alpha} d\eta \\
\lesssim &\ |v|^{\widetilde\gamma} e^{-C|v|^2}\int_{|\eta|\leq\frac{|v|}{2}} \langle \eta\rangle^\tau |\eta|^{-\alpha} d\eta\\
\lesssim &\ |v|^{\widetilde\gamma} e^{-C|v|^2}\int_0^{\frac{|v|}{2}} (1+\rho)^\tau \rho^{-\alpha+2}d\rho   \\
\lesssim &\ e^{-C|v|^2}.
\end{align*}

The estimates \eqref{A-2} is essential same as that of \eqref{A-1}, so we omit it here. This completes the proof.
\end{proof}

The following lemma shows the H\"older regularity for the kernel of the integral operator $K$, which plays a crucial rule in this paper.

\begin{lemma}\label{kH}
Let $-3<\gamma<0$. For every $v_1, v_2\in \mathbf{R}^3$, it holds that
\begin{equation*}
\int_{\mathbb{R}^3}|k(v_1,\eta)-k(v_2, \eta)| d\eta \lesssim |v_1-v_2|^{\widetilde\alpha_\gamma}\max\left\{\langle v_1\rangle^{\gamma-1}, \langle v_2\rangle^{\gamma-1}\right\},
\end{equation*}
where $\widetilde\alpha_\gamma$ is definded as
\begin{equation*}
\widetilde\alpha_\gamma=
\begin{cases}
1, & \text{ if } -2<\gamma<0,\\
1^-, & \text{ if } \gamma=-2,\\
3+\gamma, &\text{ if } -3<\gamma<-2.
\end{cases}
\end{equation*}
Furthermore, if $f\in L^{\infty}(\Omega\times\mathbb{R}^3)$, then we have
\begin{equation*}
\left|Kf(x, v_1)- Kf(x, v_2)\right|\lesssim \|f\|_{\infty}|v_1-v_2|^{\widetilde\alpha_\gamma}\max\left\{\langle v_1\rangle^{\gamma-1}, \langle v_2\rangle^{\gamma-1}\right\}.
\end{equation*}

\end{lemma}

\begin{proof}
For $|v_1-v_2|\geq 1$, by \eqref{A-1}, we have
\begin{align*}
\int_{\mathbb{R}^3}|k(v_1,\eta)-k(v_2, \eta)| d\eta &\leq \int_{\mathbb{R}^3}|k(v_1,\eta)|+|k(v_2, \eta)| d\eta \\
\ &\lesssim |v_1-v_2|^{\widetilde\alpha_\gamma}\max\left\{\langle v_1\rangle^{\gamma-2}, \langle v_2\rangle^{\gamma-2}\right\}.
\end{align*}
For $|v_1-v_2|<1$, we have
\begin{align*}
\ & \int_{\mathbb{R}^3}|k(v_1, \eta)-k(v_2, \eta)|d\eta\\
\leq &\int_{|\eta-v_1|\leq 2|v_1-v_2|}|k(v_1, \eta)|d\eta+\int_{|\eta-v_1|\leq 2|v_1-v_2|}|k(v_2, \eta)|d\eta\\
\ &\quad+\int_{|\eta-v_1|\geq 2|v_1-v_2|}|k(v_1, \eta)-k(v_2, \eta)|d\eta\\
=: &\ I_1+I_2+I_3\,.
\end{align*}
By using the spherical coordinates around $v_1$, we have
\begin{align*}
I_1&\lesssim \left(1+|v_1|\right)^{\gamma-1}\int_{|\eta-v_1|\leq 2|v_1-v_2|}A(|\eta-v_1|) e^{-\frac{1}{16}\left(|\eta-v_1|^2+\frac{(|\eta|^2-|v_1|^2)^2}{|\eta-v_1|^2}\right)}d\eta\\
\ &\lesssim \left(1+|v_1|\right)^{\gamma-1}\int_0^{2\pi}\int_0^{2|v_1-v_2|}A(\rho)\rho^{2}e^{-\frac{1}{16}\rho^2}\left(\int_0^\pi e^{-\frac{1}{16}(2|v_1|\cos\phi+\rho)^2}\sin\phi d\phi \right)d\rho d\theta\\
\ &\lesssim \left(1+|v_1|\right)^{\gamma-2}\int_0^{2|v_1-v_2|}\rho^{\bar\gamma+2}d\rho\\
\ &\lesssim \left(1+|v_1|\right)^{\gamma-2}|v_1-v_2|^{3+\bar\gamma},
\end{align*}
where $\bar\gamma$ is defind as
\begin{equation}\label{bg}
\bar\gamma=
\begin{cases}
-1, & \text{ if } -1<\gamma<0,\\
(-1)^-, & \text{ if }\gamma=-1,\\
\gamma, & \text{ if }-3<\gamma<-1.
\end{cases}
\end{equation}
For $I_2$, note that $|\eta-v_2|\leq |\eta-v_1|+|v_1-v_2|\leq 3|v_1-v_2|$, and therefore we conclude
\begin{equation*}
I_2\lesssim \left(1+|v_2|\right)^{\gamma-2}|v_1-v_2|^{3+\bar\gamma}.
\end{equation*}

For $I_3$, let $v(s)=v_2+s(v_1-v_2)$ for $s\in[0, 1]$. Note that in this case, we have
\begin{equation*}
|\eta-v(s)|\geq |\eta-v_1|-(1-s)|v_1-v_2|\geq |v_1-v_2|.
\end{equation*}
Thus one has
\begin{align*}
I_3&=\int_{|\eta-v_1|\geq 2|v_1-v_2|}\left|\int_0^1\left[ \frac{d}{ds}k(v(s), \eta)\right]ds\right|d\eta\\
\ &\lesssim |v_1-v_2|\int_0^1 \left(1+|v(s)|\right)^{\gamma} \\
\ &\qquad\qquad \times \int_{|\eta-v(s)|\geq |v_1-v_2|}|\eta-v(s)|^{\bar\gamma-1}e^{-\frac{1}{16}\left(|\eta-v(s)|^2+\frac{(|\eta|^2-|v(s)|^2)^2}{|\eta-v(s)|^2}\right)}d\eta ds\\
\ &\lesssim |v_1-v_2| \left(\int_0^1 \left(1+|v(s)|\right)^{\gamma-1} ds\right)\int_{|v_1-v_2|}^\infty \rho^{\bar\gamma+1}e^{-\frac{1}{16}\rho^2}d\rho \\
\ &\lesssim |v_1-v_2|\max\{\langle v_1\rangle^{\gamma-1}, \langle v_2\rangle^{\gamma-1}\}\left(\int_{|v_1-v_2|}^1\rho^{\bar\gamma+1}d\rho+\int_{1}^\infty\rho^{\bar\gamma+1}e^{-\frac{1}{16}\rho^2}d\rho \right) \\
\ &\lesssim
\begin{cases}
|v_1-v_2|\max\{\langle v_1\rangle^{\gamma-1}, \langle v_2\rangle^{\gamma-1}\}, & \text{ if } -2<\gamma< 0,\\
|v_1-v_2| \left| \ln|v_1-v_2| \right|  \max\{\langle v_1\rangle^{\gamma-1}, \langle v_2\rangle^{\gamma-1}\}, & \text{ if } \gamma=-2, \\
|v_1-v_2|^{3+\gamma}\max\{\langle v_1\rangle^{\gamma-1}, \langle v_2\rangle^{\gamma-1}\}, & \text{ if } -3<\gamma<-2.
\end{cases}
\end{align*}

In conclusion, we have
\begin{equation*}
I_3\lesssim |v_1-v_2|^{\widetilde\alpha_\gamma}\max\left\{\langle v_1\rangle^{\gamma-1}, \langle v_2\rangle^{\gamma-1}\right\}.
\end{equation*}
This complete the proof of the lemma.
\end{proof}

Next, we present some useful estimates for the nonlinear operator $N$.

\begin{lemma}(\cite[Lemma 2.4]{DHWZ})\label{Ne1}
Let $\tau>0$ and $0<\ka<\frac 18$. Then it holds that
\begin{align*}
\left| w(v)N_{+}(f, g)(v)\right| \lesssim \langle v\rangle^\gamma \|wf\|_\infty \|wg\|_\infty;\\
\left| w(v)N_{-}(f, g)(v)\right| \lesssim \langle v\rangle^\gamma \|wf\|_\infty \|wg\|_\infty.
\end{align*}
\end{lemma}

\begin{lemma}(\cite[Lemma 3.1]{KL}) 
It holds that
\begin{align}
\ &N_{+}(f, g)(v)\notag\\
=&\ 2\int_{\mathbb{R}^3}\int_{\eta\cdot u=0}\frac{|u+\eta|^\gamma}{|u|^2} b\left(\frac{|u|}{|u+\eta|}\right) M^{\frac12}(v+u+\eta)f(v+u)g(v+\eta) d\eta du \label{Ca1}\\
=&\ 2\int_{\mathbb{R}^3}\int_{u\cdot\eta=0}\frac{|u+\eta|^\gamma}{|\eta||u|} b\left(\frac{|u|}{|u+\eta|}\right) M^{\frac12}(v+u+\eta)f(v+u)g(v+\eta) du d\eta. \label{Ca2}
\end{align}
\end{lemma}

With \eqref{Ca1} and \eqref{Ca2}, we have the following lemma. 

\begin{lemma}\label{NH-1}
Let $\tau>0$ and $0<\kappa <\frac 18$. Then for every $x, x_1, x_2\in \Omega$ and for every $v, v_1, v_2\in \mathbb{R}^3\setminus \{0\}$,
\begin{align*}
\ &\ \left| N_+(g, g)(x_1, v)-N_+(g, g)(x_2, v)\right|   \\
\lesssim&\ \| w g\|_\infty \int_{\mathbb{R}^3}A(|u|)e^{-\frac{\kappa}{4}(|u|^2+\frac{(|u+v|^2-|v|^2)^2}{|u|^2})}\left| g(x_1, v+u)-g(x_2, v+u)\right| du; \notag\\
\ &\ \left| N_+(g, g)(x, v_1)-N_+(g, g)(x, v_2)\right|   \\
\lesssim &\ \| w g\|_\infty \int_{\mathbb{R}^3}A(|u|)\left(e^{-\frac{\kappa}{4}(|u|^2+\frac{(|u+v_1|^2-|v_1|^2)^2}{|u|^2})}+e^{-\frac{\kappa}{4}(|u|^2+\frac{(|u+v_2|^2-|v_2|^2)^2}{|u|^2})}\right) \notag\\
\ &\qquad\qquad \times \left| g(x, v_1+u)-g(x, v_2+u)\right| du \notag\\
\ &+\| w g\|_\infty^2 |v_1-v_2| ; \notag \\
\ &\ \left|\varphi(g)(x_1, v)-\varphi(g)(x_2, v)\right|  \\
\lesssim &\ \int_{\mathbb{R}^3} |v-u|^\gamma M^{\frac12}(u) \left| g(x_1, u)-g(x_2, u)\right| du ;\notag     \\
\ &\ \left|\varphi(g)(x, v_1)-\varphi(g)(x, v_2)\right| \\
\lesssim &\ \| w g\|_\infty|v_1-v_2|+ \int_{\mathbb{R}^3}|u|^\gamma M^{\frac12}(v_2+u)\left|g(x, v_1+u)- g(x, v_2+u)\right| du, \notag
\end{align*}
where $A(|u|)$ is defined as \eqref{A}.
\end{lemma}

\begin{proof}
This proof is most same as that of \cite[Lemma 3.2]{KL} and is omitted here. 
\end{proof}

\subsection{Some estimates for $t_{b}$ and $x_{b}$} In this subsection, we will present some results about the backward exit time $t_{b}$ and it's associate position $x_{b}$, which will be used in our analysis.
\begin{proposition}(\cite[Propositions 3.1 and 3.3]{Ch})
Suppose that $\Omega$ is a bounded and strictly convex domain in $\mathbb{R}^3$. Let $x, x_1, x_2,\in\Omega$, $v, v_1, v_2\in\mathbb{R}^3\setminus\{0\}$, and let $\theta_{(v_1, v_2)}$ be the angle between $v_1$ and $v_2$. Then it holds that
\begin{align}
\left|x_b(x_1,v)-x_b(x_2,v)\right|&\lesssim d_{(x_1, x_2)}^{-1}|x_1-x_2|;\label{xbx}\\
\left|t_b(x_1,v)-t_b(x_2,v)\right|&\lesssim d_{(x_1, x_2)}^{-1} |x_1-x_2| |v|^{-1};\label{tbx}\\
\left|x_b(x,v_1)-x_b(x,v_2)\right|&\lesssim d_x^{-1}\theta_{(v_1, v_2)}. \label{xbv}
\end{align}
\end{proposition}

\begin{proposition}(\cite[Proposition 3.2]{Ch})\label{P3.2}
Suppose that $\Omega$ is a bounded and strictly convex domain in $\mathbb{R}^3$. Let $x\in\Omega$, $v\in\mathbb{R}^3\setminus\{0\}$, and $y\in \overline{x x_b(x,v)}$, then
\begin{equation*}
\left|\overline{y x_b(x,v)}\right|\leq\frac{R}{d_x}d_y,
\end{equation*}
where $R$ is the diameter of $\Omega$.
\end{proposition}

The following lemma shows that the singularity in velocity can be transferred to the singularity in space. 

\begin{lemma}
Let $-3<\gamma<0$, $s\geq 0$, and $0< d\leq R$ for some $R>0$. For any $v_1, v_2\in\mathbb{R}^3\setminus\{0\}$ with $|v_1|\geq |v_2|$, we have
\begin{equation}\label{e2}
\left(\frac{d}{|v_1|}\right)^se^{-\frac{\nu(v_2)}{|v_1|} d}\leq C_{\gamma, s, R}\langle v_1\rangle^{-s}.
\end{equation}
\end{lemma}

\begin{proof}
For $|v_1|\leq 1$, one has $\nu(v_2) \sim 1$. Thus,
$$\left(\frac{d}{|v_1|}\right)^se^{-\frac{\nu(v_2)}{|v_1|} d}\leq \left(\frac{d}{|v_1|}\right)^se^{-C_\gamma\frac{d}{|v_1|} }\leq C_{\gamma, s}.$$
For $|v_1|>1$, we have
$$\left(\frac{d}{|v_1|}\right)^se^{-\frac{\nu(v_2)}{|v_1|} d}\leq R^s |v_1|^{-s}.$$
This completes the proof.
\end{proof}

\begin{lemma}
Let $-3<\gamma<0$ and $a>-1$. For any $x, x_1, x_2 \in \Omega$ and $v\in\mathbb{R}^3\setminus\{0\}$, we have
\begin{align}
\int_0^{t_b(x,v)}s^a e^{-\nu(v)s}ds &\lesssim \langle v\rangle^{-1-a};\label{tb1}\\
\left| \int_{t_b(x_2, v)}^{t_b(x_1, v)} e^{-\nu(v)s}ds\right| &\lesssim d_{(x_1, x_2)}^{-1}|x_1-x_2||v|^{-1}e^{-\frac{\nu(v)}{|v|}d_{(x_1, x_2)}}.\label{tb2}
\end{align}
Let $v_1, v_2\in\mathbb{R}^3$ with $|v_1|\geq |v_2|>0$ and $\bar v_2=\frac{|v_1|}{|v_2|}v_2$, then
\begin{align}
\ & \int_0^{t_b(x,v_1)}\left|e^{-\nu(v_2)s}-e^{-\nu(v_1)s}\right|ds \lesssim |v_1-v_2| ;\label{tb3}\\
\ & \left| \int_{t_b(x, \bar v_2)}^{t_b(x, v_1)} \nu(v_1) e^{-\nu(v_1)s}ds\right| \lesssim d_x^{-1}|v_1-v_2| \langle v_1\rangle^{-1}; \label{tb4}\\
\ & \left| \int_{t_b(x, \bar v_2)}^{t_b(x, v_2)}\nu(v_2) e^{-\nu(v_2)s}ds\right| \lesssim d_{x}^{-1}|v_1-v_2| \langle v_1\rangle^{-1}. \label{tb5}
\end{align}
\end{lemma}

\begin{proof}
For \eqref{tb1}, we divide the proof into two cases. \\ 
$(i)$  $|v|\leq 1$. One has $\nu(v)\sim 1$ and $t_b(x,v)=\frac{\left|\overline{xx_b(x,v)}\right|}{|v|}\leq \frac{R}{|v|}$, where $R$ is the diameter of $\Omega$. We then have
\begin{equation*}
\int_0^{t_b(x,v)} s^{a}e^{-\nu(v)s}ds\leq \int_0^\infty s^a e^{-C_\gamma s} ds=C_{\gamma, a}.
\end{equation*}
$(ii)$ $|v|>1$. One has $\nu(v)\sim |v|^\gamma$. We then have
\begin{equation*}
\int_0^{t_b(x,v)} s^{a}e^{-\nu(v)s}ds\leq \int_0^{\frac{R}{|v|}} s^a e^{-C_\gamma |v|^\gamma s} ds=|v|^{-1-a}\int_0^Rt^a dt= C_{a, \Omega} |v|^{-1-a}.
\end{equation*}

For \eqref{tb2}, using the mean value theorem and \eqref{tbx}, we have
\begin{align*}
\left| \int_{t_b(x_2, v)}^{t_b(x_1, v)} e^{-\nu(v)s}ds\right| &\leq \left|t_b(x_1, v)-t_b(x_2, v)\right|e^{-\nu(v)\min\left\{t_b(x_1, v), t_b(x_2, v)\right\}}\\
\ &\lesssim d_{(x_1, x_2)}^{-1} |x_1-x_2| |v|^{-1}e^{-\frac{\nu(v)}{|v|}d_{(x_1, x_2)}}\,.
\end{align*}

For \eqref{tb3}, using the mean value theorem and \eqref{tb1}, we have
\begin{align*}
\ &\int_0^{t_b(x,v_1)}\left|e^{-\nu(v_2)s}-e^{-\nu(v_1)s}\right|ds =\int_0^{t_b(x,v_1)}\left| \int_{\nu(v_1)}^{\nu(v_2)}se^{-st}dt\right|ds \\
\leq & \int_0^{t_b(x,v_1)}s|\nu(v_2)-\nu(v_1)| e^{-\min\{\nu(v_1), \nu(v_2)\}s} ds\\
\leq &\ |v_1-v_2|\left| \nabla_v \nu(v(t))\right|\int_0^{t_b(x,v_1)} se^{-\frac{\nu_1}{\nu_2}\nu(v_1)s}ds\\
\lesssim &\ |v_1-v_2|,
\end{align*}
where $v(t)= v_1+t(v_2-v_1)$ for some $0\leq t\leq 1$.

For \eqref{tb4}, let $\theta_{(v_1, \bar v_2)}$ be the angle between $v_1$ and $\bar v_2$. One can see that $|v_1-\bar v_2|\leq 2|v_1-v_2|$ and
\begin{equation*}
\sin\left(\theta_{(v_1, \bar v_2)}/2\right)=\frac{\frac12|v_1-\bar v_2|}{|v_1|}\geq \frac{1}{\pi}\theta_{(v_1, \bar v_2)}.
\end{equation*}
Thus
\begin{equation*}
\theta_{(v_1, \bar v_2)}\lesssim |v_1-v_2||v_1|^{-1}.
\end{equation*}
Now we divide the proof into two cases: \\ 
$(i)$  $\theta_{(v_1, \bar v_2)}\geq \frac13\pi$. By using \eqref{e2}, we have
\begin{align*}
\left| \int_{t_b(x, \bar v_2)}^{t_b(x, v_1)} \nu(v_1) e^{-\nu(v_1)s}ds\right| &=\left|e^{-\nu(v_1)t_b(x, v_1)}-e^{-\nu(v_1)t_b(x, \bar v_2)}\right|\\
\ &\lesssim  \theta_{(v_1, \bar v_2)}e^{-\frac{\nu(v_1)}{|v_1|}\min\{|\overline{xx_b(x, v_1)}|, |\overline{xx_b(x, v_2)}|\}}\\
\ &\lesssim  |v_1-v_2| |v_1|^{-1}e^{-\frac{\nu(v_1)}{|v_1|}d_x}\\
\ &\lesssim  d_x^{-1}|v_1-v_2| \langle v_1\rangle^{-1}. 
\end{align*}
$(ii)$  $\theta_{(v_1, \bar v_2)}<\frac13\pi$. From \cite[Lemma 2]{G10}, we have
\begin{equation}\label{tbd}
\nabla_v t_b(x,v)=\frac{t_b(x, v) n(x_b(x, v))}{v\cdot n(x_b(x, v))}.
\end{equation}
Let $Y$ be the projection of $x$ on the tangent plane $T_{x_b(x,v)}(\partial\Omega)$. By the convexity of domain $\Om$, $\overline{xY}$ intersects $\partial\Omega$ at one point $X$. Then
\begin{equation}\label{XY}
\frac{|v\cdot n(x_b(x, v))|}{|v|}=\frac{|\overline{xY}|}{|\overline{xx_b(x, v)}|}\geq \frac{|\overline{xX}|}{|\overline{xx_b(x, v)}|}\geq \frac{d_x}{|\overline{xx_b(x, v)}|}.
\end{equation}
Let $v(t)=v_1+t(\bar v_2-v_1)$ for some $0\leq t\leq 1$. Then $\frac{\sqrt{3}}{2}|v_1|\leq |v(t)|\leq |v_1|$. By using the mean value theorem, \eqref{e2}, \eqref{tbd}, and \eqref{XY}, we have
\begin{align*}
\left| \int_{t_b(x, \bar v_2)}^{t_b(x, v_1)} \nu(v_1)e^{-\nu(v_1)s}ds\right|&=\left|e^{-\nu(v_1)t_b(x, v_1)}-e^{-\nu(v_1)t_b(x, \bar v_2)}\right|\\
\ &=\left|(v_1-\bar v_2)\cdot e^{-\nu(v_1)t_b(x, v(t))}\nu(v_1)\nabla_v t_b(x, v(t))\right|\\
\ &\lesssim |v_1-v_2| \nu(v_1)\frac{|\overline{xx_b(x, v(t))}|}{d_x|v(t)|}t_b(x, v(t))e^{-\frac{\nu(v_1)}{|v(t)|}}|\overline{xx_b(x, v(t))}|\\
\ &\lesssim d_x^{-1}|v_1-v_2| \langle v_1\rangle^{\gamma-2}. 
\end{align*}

For \eqref{tb5}, note that we have $x_b(x, \bar v_2)=x_b(x, v_2)$ and $|\bar v_2-v_2|\leq |v_1-v_2|$.
We divide the proof into two cases: \\ 
$(i)$ $\frac{|v_1|}{2}\leq |v_2|\leq |v_1|$. By using the mean value theorem and \eqref{e2}, we have
\begin{align*}
\ &\left| \int_{t_b(x, \bar v_2)}^{t_b(x, v_2)} \nu_2)e^{-\nu(v_2)s}ds\right| \\
\leq &\ \nu(v_2)\left(\frac{|\overline{xx_b(x, v_2)}|}{|v_2|}-\frac{|\overline{xx_b(x,\bar v_2)}|}{|\bar v_2|}\right)e^{-\nu(v_2)\min\{t_b(x, v_2), t_b(x, \bar v_2)\}}\\
\leq &\ \nu(v_2)|v_1-v_2| \frac{|\overline{xx_b(x, v_2)}|}{|v_1|^2}e^{-\frac{\nu(v_2)}{2|v_1|}|\overline{xx_b(x, v_2)}|}\\
\lesssim &\ d_x^{-1}|v_1-v_2| e^{-c\frac{\nu(v_1)}{|v_1|}d_x}\langle v_1\rangle^{\gamma-2}.
\end{align*}
$(ii)$ $|v_2|< \frac{|v_1|}{2}$. One has $|v_1|\sim |v_1-v_2|$. Thus, by \eqref{e2}, we have
\begin{align*}
\left| \int_{t_b(x, \bar v_2)}^{t_b(x, v_2)} \nu(v_2)e^{-\nu(v_2)s}ds\right|&= \left(e^{-\nu(v_2)t_b(x,\bar v_2)}-e^{-\nu(v_2)t_b(x, v_2)}\right)\\
\ &\leq \left[t_b(x, \bar v_2)\right]^{-1}\left[t_b(x, \bar v_2)\right]e^{-\nu(v_2)t_b(x,\bar v_2)}\\
\ &\leq \frac{|v_1|}{|\overline{xx_b(x,v_2)}|}\langle v_1\rangle^{-1}\\
\ &\lesssim d_x^{-1}|v_1-v_2| \langle v_1\rangle^{-1}.
\end{align*}
This complete the proof of the lemma.
\end{proof}


\section{Approximate Steady Boltzmann equation}\label{app-sol}

In this section, we will construct the solution to problem \eqref{approx-1}.
The precise statement of the main result in this section is as follows:
\begin{proposition}\label{LEU}
Let $-3<\gamma<0$, and let $\tau>3+|\gamma|$ and $0<\ka<\frac 18$. Assume that $\|wg_1\|_\infty+\|\nu^{-1}wg_2\|_\infty+|wr|_{\infty, -}<\infty$. Then there exists $\delta_1>0$ such that if $\|wg_1\|_\infty\leq \delta_1$, there is a unique solution $f$ to problem \eqref{approx-1} satisfying
\begin{equation*}
\|wf\|_{\infty}+|wf|_\infty\leq C\left(\|\nu^{-1}wg_2\|_\infty+|wr|_{\infty, -}\right).
\end{equation*}
Moreover, if $\Omega$ is strictly convex, $g_1$ and $g_2$ are continuous on $\Omega\times\mathbb{R}^3$, and $r$ is continuous on $\Gamma_-$, then $f$ is continuous away from $\Gamma_0$.
\end{proposition}

To prove this proposition, following the approach in \cite{DHWZ}, we turn to consider the solvability of the approximate boundary-value problem:
\begin{equation}\label{1elf1}
\begin{cases}
\epsilon f+v\cdot\nabla_x f+\nu(v)f-\lambda Kf+\varphi(g_1) f=g_2, & \text{ in } \Omega\times\mathbb{R}^3,\\
f=r, & \text{ on } \Gamma_-.
\end{cases}
\end{equation}

In what follows we give the {\it a priori $L^\infty$-estimate} and {\it a priori $L^2$-estimate} for our problem.
Consider the following approximate linearized steady problem
\begin{equation}\label{1elf}
\begin{cases}
\epsilon f+v\cdot\nabla_x f+\nu(v)f-\lambda Kf=g, & \text{ in } \Omega\times\mathbb{R}^3,\\
f=r, & \text{ on } \Gamma_-.
\end{cases}
\end{equation}
Let $h(x, v)=w(v)f(x, v)$. Then problem \eqref{1elf} for $h$ reads as
\begin{equation}\label{1elh}
\begin{cases}
\epsilon h+v\cdot\nabla_x h+\nu(v)h=\lambda K_{w}h+wg, & \text{ in } \Omega\times\mathbb{R}^3,\\
h=wr, & \text{ on } \Gamma_-,
\end{cases}
\end{equation}
where $K_{w}h=wK(h/w)$. For the result of the $L^\infty$ bound of solutions to problem \eqref{1elh}, we have the following lemma whose proof is essential same as that in \cite[Lemma 3.4]{DHWZ} and is omitted here.

\begin{lemma}\label{DHWZ1}
Let $\epsilon>0$, $0\leq\lambda\leq 1$, $\tau>3$, and $0<\ka<\frac18$. Assume that $h$ is the solution of \eqref{1elh}. Then
\begin{equation*}
\|h\|_\infty+|h|_{\infty}\leq C\left( \|\nu^{-1}wg\|_\infty+|wr|_{\infty, -}+\left\| \frac{\nu^{\frac12}h}{w}\right\|_2\right),
\end{equation*}
where $C>0$ does not depend on $\epsilon$ and $\lambda$.
\end{lemma}

For the $L^\infty$ bound of solutions to problem \eqref{1elf1}, we have the following result.

\begin{corollary}\label{Co1}
Let $\epsilon>0$, $0\leq\lambda\leq 1$, $\tau>3$, and $0<\ka<\frac 18$. There exists $\delta_1>0$ independent of $\epsilon$ and $\lambda$ such that if $\|wg_1\|_\infty\leq \delta_1$ and $f$ is a solution of \eqref{1elf1}, then
\begin{equation}\label{ap1}
\|wf\|_\infty+|wf|_{\infty}\leq C\left( \|\nu^{-1}wg_2\|_\infty+|wr|_{\infty, -}+\| \nu^{\frac12}f\|_2\right),
\end{equation}
where $C>0$ does not depend on $\epsilon$ and $\lambda$.
\end{corollary}

\begin{proof}
Note that, from Lemma \ref{Ne1}, we have
\begin{equation}\label{ng1}
|\varphi(g)(x, v)| \leq \widehat C\|w g\|_{\infty} \nu(v),
\end{equation}
where $\widehat C>0$ depends only on $\gamma$.
Applying Lemma \ref{DHWZ1} with $g=g_2-\varphi(g_1)f$, it follows from \eqref{ng1} that
\begin{align*}
\ &\ \|wf\|_\infty+|wf|_\infty \\
\leq&\ C \left(\|\nu^{-1}w\left[ g_2-\varphi(g_1)f \right] \|_\infty+|wr|_{\infty, -}+\| \nu^{\frac12}f \|_2\right)\\
\leq &\ C\widehat C\|wg_1\|_\infty\|wf\|_\infty +C \left(\|\nu^{-1}wg_2\|_\infty+|wr|_{\infty, -}+\| \nu^{\frac12}f \|_2\right).
\end{align*}
Thus, the inequality \eqref{ap1} holds provided that $\|wg_1\|_\infty$ is small enough.
\end{proof}

The following lemma gives an {\it a priori $L^2$ estimate}.

\begin{lemma}\label{L2E}
Let $\epsilon>0$ and $0\leq \lambda\leq 1$. Assume that $f$ is a solution of
\begin{equation*}
\begin{cases}
\epsilon f+v\cdot\nabla_x f+\lambda Lf=g, & \text{ in } \Omega\times\mathbb{R}^3,\\
f=r, & \text{ on } \Gamma_-,
\end{cases}
\end{equation*}
in the weak sense of
\begin{equation*}
\int_{\Gamma} \psi f (v\cdot n(x)) dvdS(x) -\int_{\Omega\times\mathbb{R}^3}(v\cdot\nabla_x\psi) f dvdx=\int_{\Omega\times\mathbb{R}^3}\left(-\epsilon f-\lambda Lf+g\right)\psi dvdx,
\end{equation*}
for any $\psi \in H^1(\Omega\times\mathbb{R}^3)$.
Then
\begin{equation*}
\| Pf\|_2\leq C \left(\epsilon\|f\|_2+\lambda\|\nu^{\frac12}(I-P) f\|_2+\|g\|_2+|f|_{2,+}+|r|_{2, -}\right),
\end{equation*}
where $C>0$ does not depend on $\epsilon$ and $\lambda$.
\end{lemma}

\begin{proof}
Recall that $Pf (x, v)=\left(a(x)+b(x)\cdot v+c(x)\frac{|v|^2-3}{\sqrt{6}}\right)M^{\frac 12}(v)$ on $\Omega\times\mathbb{R}^3$.
The estimates of $b(x)$ and $c(x)$ are essential same as that in \cite[Lemma 3.3]{EGKM13}, except the estimate of $a(x)$. For the estimate of $a(x)$, it is similar as that of $c(x)$ under the case of in-flow boundary condition. Therefore, we omit the details of the proof here for simplicity of presentation.
\end{proof}

Now we are in a position to prove Proposition \ref{LEU}:
\begin{proof}[Proof of Proposition \ref{LEU}]
Given $\epsilon>0$. We first consider the solvability of the following boundary value problem for $\lambda$:
\begin{equation}\label{1el}
\begin{cases}
\mathbb{L}_\lambda:=\epsilon f+v\cdot\nabla_x f+\nu(v)f+\varphi(g_1) f-\lambda Kf=g_2, & \text{ in } \Omega\times\mathbb{R}^3,\\
f=r, & \text{ on } \Gamma_-.
\end{cases}
\end{equation}
Moreover, we denote $\mathbb{L}_\lambda^{-1}$ to be the solution operator associated with problem \eqref{1el}.
Because the proof is very long, we divide it into several steps.

{\bf Step 1. The existence of $\mathbb{L}_0^{-1}$.} Let $h(x, v)= w(v) f(x, v)$, then the problem \eqref{1el} for $h$ with $\lambda=0$ reads as
 \begin{equation*}\label{1el'}
\begin{cases}
\epsilon h+v\cdot\nabla_x h+\nu(v)h+\varphi(g_1)h =g_2, & \text{ in } \Omega\times\mathbb{R}^3,\\
h=wr, & \text{ on } \Gamma_-.
\end{cases}
\end{equation*}
By the method of characteristics, for almost every $(x, v)\in\overline\Omega\times\mathbb{R}^3\backslash \left(\Gamma_0\cup \Gamma_-\right)$, one has
\begin{align*}
h(x, v) &= w(v)r(x_b(x,v), v) e^{-(\epsilon+\nu(v))t_b(x, v)-\int_0^{t_b(x,v)}\varphi(g_1)(x-v\ell, v)d\ell}\\
\ &\ +\int_0^{t_b(x,v)} e^{-(\epsilon+\nu(v))s-\int_0^{s}\varphi(g_1)(x-v\ell, v)d\ell}w(v)g_2(x-vs, v)ds
\end{align*}
and for $(x, v)\in\Gamma_-$, one has
$$h(x, v)=w(v)r(x, v).$$
By \eqref{ng1}, we have
\begin{equation*}
e^{(-\epsilon+\nu(v))s-\int_0^{s}\varphi(g_1)(x-v\ell, v)d\ell}\leq e^{-(\epsilon+\frac{\nu(v)}{2})s}
\end{equation*}
provided that $\|wg_1\|_\infty$ is small enough. Thus, we deduce
\begin{equation*}
\|h\|_\infty+|h|_\infty \leq |wr|_{\infty, -}+ C_1\epsilon^{-1}\|wg_2\|_\infty.
\end{equation*}
Let $(x, v)\in\overline\Omega\times\mathbb{R}^3\backslash \Gamma_0$. If $\Omega$ is strictly convex, then it holds that $v\cdot n(x_b(x,v))<0$, which implies that $t_b(x, v)$ and $x_b(x, v)$ are smooth by \cite[Lemma 2]{G10}. Therefore, if $g_1, g_2$, and $r$ are continuous, we conclude that $h$ is continuous away from $\Gamma_0$.

{\bf Step 2. The existence of $\mathbb{L}_\lambda^{-1}$ for $\lambda \in [0, \lambda^*]$.}
We first give an \textit{a priori estimate}. For any given $\lambda\in[0, \lambda^*]$ with $\lambda^*=1-\widehat C\|wg_1\|_\infty$, where $\widehat C$ is as in \eqref{ng1}. Multiplying  the first equation of \eqref{1el} by $f$ and using the Green's identity, one has
\begin{align*}
\epsilon \|f\|_2^2+\frac 12|f|_{2,+}^2+\lambda\langle Lf, f\rangle+(1-\lambda)\|\nu^{\frac12}f\|_2^2+\langle \varphi(g_1)f, f\rangle=\langle g_2 ,f\rangle+\frac12|r|_{2,-}^2,
\end{align*}
where $\langle \cdot, \cdot\rangle$ denotes the $L^2(\Omega\times\mathbb{R}^3)$ inner product.
It follows from $\langle Lf, f\rangle\geq 0$ and \eqref{ng1} that
\begin{align*}
\epsilon \|f\|_2^2+\frac 12|f|_{2,+}^2+(1-\lambda-\widehat C\|wg_1\|_\infty )\|\nu^{\frac12}f\|_2^2\leq \ \frac{\epsilon}{2}\|f\|_2^2+\frac{1}{2\epsilon}\|g_2\|_2^2+\frac12|r|_{2,-}^2.
\end{align*}
Thus,
\begin{equation}\label{*1}
\|f\|_2+|f|_{2,+}=\|\mathbb{L}_\lambda^{-1}g_2\|_2+|\mathbb{L}_\lambda^{-1}g_2|_{2,+} \leq C_\epsilon\left (\|g_2\|_2+|r|_{2, -}\right).
\end{equation}
Furthermore, by Corollary \ref{Co1} and \eqref{*1}, we have
\begin{align}\label{*2}
\ &\ \|wf\|_\infty+|wf|_\infty \\
\leq &\ C_2\left( \|\nu^{-1}w g_2\|_\infty+|wr|_{\infty, -}+\|\nu^{\frac12}f\|_2\right)\notag\\
\leq &\ C_\epsilon \left( \|\nu^{-1}w g_2\|_\infty+|wr|_{\infty, -}+\|g_2\|_2+|r|_{2, -}\right)\notag\\
\leq &\ C_\epsilon \left( \|\nu^{-1}w g_2\|_\infty+|wr|_{\infty, -}\right).\notag
\end{align}
Let $\nu^{-1}w\bar g_2\in L^\infty(\Omega\times\mathbb{R}^3)$ and $\nu^{-1}w\tilde g_2\in L^\infty(\Omega\times\mathbb{R}^3)$. Assume that $\bar f=\mathbb{L}_\lambda^{-1}\bar g_2$ and $\tilde f=\mathbb{L}_\lambda^{-1}\tilde g_2$ are the solutions to \eqref{1el} with $g_2$ replaced by $\bar g_2$ and $\tilde g_2$, respectively. Then
\begin{equation*}\label{1el-}
\begin{cases}
\epsilon (\bar f-\tilde f)+v\cdot\nabla_x (\bar f-\tilde f)+\nu(v)(\bar f-\tilde f)+\varphi(g_1)(\bar f-\tilde f)-\lambda K(\bar f-\tilde f)=\bar g_2-\tilde g_2, \\
\bar f-\tilde f |_{\Gamma_-}=0.
\end{cases}
\end{equation*}
By \eqref{*1} and \eqref{*2}, we have
\begin{align}
\|\bar f-\tilde f\|_2+|\bar f-\tilde f|_{2,+}&\leq C_\epsilon\|\bar g_2-\tilde g_2\|_\infty; \label{*3} \\
\|w(\bar f-\tilde f)\|_\infty+|w(\bar f-\tilde f)|_\infty & \leq C_\epsilon \|\nu^{-1}w(\bar g_2-\tilde g_2)\|_\infty.\label{*4}
\end{align}
Evidently, the uniqueness of solutions to \eqref{1el} follows from \eqref{*3}--\eqref{*4}. Note that the constant $C_\epsilon$ in \eqref{*1}--\eqref{*4} does not depend on $\lambda \in [0, \lambda^*]$, which is crucial to extend $\mathbb{L}_0^{-1}$ to $\mathbb{L}_\lambda^{-1}$ by a bootstrap argument. We now define a Banach space
\begin{equation*}
B=\left\{ f(x, v) : wf\in L^\infty(\Omega\times\mathbb{R}^3)\cap L^\infty(\Gamma), f(x, v)|_{\Gamma_-}=r(x, v) \right\}.
\end{equation*}
Then we define a map on $B$ by $T_\lambda f=\mathbb{L}_0^{-1}(\lambda Kf+g_2).$ By \eqref{*2}, one has $T_\lambda : B \to B$.
For any $f_1, f_2\in B$, by using \eqref{*4}, we have
\begin{align*}
\ & \|w(T_\lambda f_1-T_\lambda f_2)\|_\infty\\
\leq&\ C_\epsilon \|\nu^{-1}w(\lambda Kf_1+g_2-\lambda Kf_2-g_2)\|_\infty\\
\leq &\ \lambda C_\epsilon\|w(f_1-f_2)\|_\infty.
\end{align*}
We can take $\lambda_\epsilon>0$ small sufficiently such that $\lambda_\epsilon C_\epsilon\leq \frac12$, then $T_\lambda :B \to B$ is a contraction mapping for $\lambda\in [0, \lambda_\epsilon]$. Hence, $T_\lambda$ has a fixed point, i.e., there is $f\in B$ such that $f=T_\lambda f=\mathbb{L}_0^{-1}(\lambda Kf+g_2)$.
Define $T_{\lambda_\epsilon+\lambda}f=\mathbb{L}_{\lambda_\epsilon}^{-1}(\lambda Kf+g_2)$. Note that estimates \eqref{*1}--\eqref{*4} for $\mathbb{L}_{\lambda_\epsilon}^{-1}$ are independent of $\lambda_{\epsilon}$. By repeating the previous procedures, we conclude that $T_{\lambda_\epsilon+\lambda}: B\to B$ is a contraction mapping for $\lambda \in [0,\lambda_\epsilon]$. Step by step, we can obtain the existence of $\mathbb{L}_\lambda^{-1}$ for $\lambda\in [0, \lambda^*]$ with $\lambda^*=1-\widehat C\|wg_1\|_\infty$. Evidently, the operator $\mathbb{L}_\lambda^{-1}$ satisfies \eqref{*1}--\eqref{*4}.
Furthermore, the fixed point $f$ can be found by defining a sequence $\{f_n\}$ with $f_0\equiv 0$ and $f_{n+1}=T_\lambda(f_n)$ for $n\geq 1$, then $\|f_n-f\|_\infty \to 0$.
Since the sequence under consideration always converges in $L^\infty$, we conclude that the solution is continuous away from $\Gamma_0$ when $\Omega$ is strictly convex.

{\bf Step 3. The existence of $\mathbb{L}_\lambda^{-1}$ for $\lambda \in [\lambda^*, 1]$.}
Note that $\lambda^*> \frac12$ for $\|wg_1\|_\infty$ sufficiently small.
We first give an \textit{a priori estimate}.
For any given $\lambda\in[\lambda^*, 1]$ with $\lambda^*=1-\widehat C\|wg_1\|_\infty$, multiplying the equation of \eqref{1el} by $f$ and using the Green's identity, one has
\begin{align*}
\epsilon \|f\|_2^2+\frac 12|f|_{2,+}^2+\lambda\langle Lf, f\rangle+(1-\lambda)\|\nu^{\frac12}f\|_2^2=-\langle \varphi(g_1)f, f\rangle+\langle g_2 ,f\rangle+\frac12|r|_{2,-}^2.
\end{align*}
It follows from \eqref{LP1} and \eqref{ng1} that
\begin{align}\label{*5}
\ &\ \epsilon \|f\|_2^2+\frac 12|f|_{2,+}^2+\lambda C_0\|\nu^{\frac12}(I-P)f\|_2^2 \\
\leq &\ (\lambda-\lambda^*)\|\nu^{\frac12}f\|_2^2+ | \langle g_2 ,f\rangle |+\frac12|r|_{2,-}^2.\notag
\end{align}
Applying Lemma \ref{L2E} with $g=g_2-\varphi(g_1)f-(1-\lambda)\nu(v)f$, we have
\begin{align}\label{*6}
\| Pf\|_2 &\leq C_3 \left[ \epsilon\|f\|_2+\lambda\|\nu^{\frac12}(I-P) f\|_2+\|g\|_2+|f|_{2,+}+|r|_{2, -}\right]  \\
\ & \leq C_3 \left[ \epsilon\|f\|_2+\lambda\|\nu^{\frac12}(I-P) f\|_2+\left(\widehat C\|wg_1\|_\infty+1-\lambda\right) \|\nu f\|_2 \right. \notag\\
\ &\ \quad\quad \left.  +\|g_2\|_2+|f|_{2,+}+|r|_{2, -} \right]. \notag
\end{align}
Combining \eqref{*5} and \eqref{*6}, we deduce from $\lambda\geq\lambda^*>\frac 12$ that
\begin{align*}
\ &\ \|\nu^{\frac12}f\|_2+|f|_{2, +} \\
\leq &\ \|Pf\|_2+\|\nu^{\frac12}(I-P)f\|_2+|f|_{2, +}\notag\\
\leq &\ C_4 \left[ \epsilon\|f\|_2+\lambda\|\nu^{\frac12}(I-P) f\|_2+\left(\widehat C\|wg_1\|_\infty+1-\lambda\right) \|\nu f\|_2 \right.  \notag\\
\ &\quad\quad \left. +\|g_2\|_2+|f|_{2,+}+|r|_{2, -} \right] \notag\\
\leq &\ \overline C_4 \left[ \left(\widehat C\|wg_1\|_\infty+1-\lambda+(\lambda-\lambda^*)\right)\|\nu^{\frac12}f\|_2+|\langle g_2, f\rangle|^{\frac12}+\|g_2\|_2+|r|_{2, -} \right]\notag\\
\leq &\ \overline C_4 \left(\|wg_1\|_\infty \|\nu^{\frac12}f\|_2+ \delta \|\nu^{\frac12}f\|_2+C_\delta\|\nu^{-\frac12}g_2\|_2+|r|_{2,-} \right).\notag
\end{align*}
Taking $\|wg_1\|_\infty$ and $\delta>0$ small enough, we have
\begin{align}\label{*8}
\|\nu^{\frac12}f\|_2+|f|_{2, +} \leq \overline{C}_\delta\left(\|\nu^{-\frac12}g_2\|_2+|r|_{2,-}\right) \leq \widetilde C_\delta\left( \|\nu^{-1}wg_2\|_\infty+|wr|_{\infty, -} \right).
\end{align}
Applying Corollary \ref{Co1} with \eqref{*8}, we have
\begin{equation}\label{*9}
\|wf\|_\infty+|wf|_\infty\leq C_5 \left(\|\nu^{-1}w g_2\|_\infty+|wr|_{\infty, -}\right).
\end{equation}
Let $\nu^{-1}w\bar g_2\in L^\infty(\Omega\times\mathbb{R}^3)$ and $\nu^{-1}w\tilde g_2\in L^\infty(\Omega\times\mathbb{R}^3)$. Assume that $\bar f=\mathbb{L}_\lambda^{-1}\bar g_2$ and $\tilde f=\mathbb{L}_\lambda^{-1}\tilde g_2$ are two solutions to \eqref{1el} with $g_2$ replaced by $\bar g_2$ and $\tilde g_2$, respectively.
Applying \eqref{*8} and \eqref{*9} gives
\begin{align}
\|\nu^{\frac12}(\bar f-\tilde f)\|_2+|\bar f-\tilde f|_{2,+}&\leq C_6\|\nu^{-\frac12}(\bar g_2-\tilde g_2)\|_2;\label{*10}\\
\|w(\bar f-\tilde f)\|_\infty+|w(\bar f-\tilde f)|_\infty &\leq C_6\|\nu^{-1}w(\bar g_2-\tilde g_2)\|_\infty.\label{*11}
\end{align}
Note that the constant $C_6$ in \eqref{*10}--\eqref{*11} does not depend on $\epsilon$ and $\lambda \in [\lambda^*, 1]$.
We define a map on $B$ by $T_{\lambda^*+\lambda}f=\mathbb{L}_{\lambda^*}^{-1}(\lambda Kf+g_2)$. For any $f_1, f_2\in B$, by using \eqref{*11}, we have
\begin{align*}
\ &\ \|w(T_{\lambda^*+\lambda}f_1-T_{\lambda^*+\lambda}f_2)\|\infty\\
\leq &\ C_6\|\nu^{-1}w(\lambda Kf_1+g_2-\lambda Kf_2-g_2)\|_\infty \\
\leq &\ \lambda \overline C_6\|w(f_1-f_2)\|_\infty.
\end{align*}
We take $\lambda_0>0$ sufficiently small such that $\lambda_0\overline C_6\leq \frac12$, then $T_{\lambda^*+\lambda} : B\to B$ is a contraction mapping for $\lambda\in (0, \lambda_0]$. Step by step, we can obtain the existence of $\mathbb{L}_1^{-1}$ and $\mathbb{L}_1^{-1}$ satisfies \eqref{*8}--\eqref{*11}.
Furthermore, since the sequence under consideration always converges in $L^\infty$, the solution is continuous away from $\Gamma_0$ when $\Omega$ is strictly convex.

{\bf Step 4. The convergence of $f^\epsilon$ as $\epsilon \to 0^+$.}
Let $f^\epsilon$ be the solution of \eqref{1el} with $\lambda=1$ obtained by step 3. For any $\epsilon_1, \epsilon_2>0$, $f^{\epsilon_1}-f^{\epsilon_2}$ satisfies
\begin{equation*}
\begin{cases}
v\cdot \nabla_x(f^{\epsilon_1}-f^{\epsilon_2}) +L(f^{\epsilon_1}-f^{\epsilon_2}) +\varphi(g_1)(f^{\epsilon_1}-f^{\epsilon_2})= -\epsilon_1f^{\epsilon_1}+\epsilon_2f^{\epsilon_2},\\
f^{\epsilon_1}-f^{\epsilon_2}|_{\Gamma_-}=0.
\end{cases}
\end{equation*}
Repeating the similar arguments in \eqref{*5}--\eqref{*9}, we have
\begin{align}\label{*12}
\| \nu^{\frac12}(f^{\epsilon_1}-f^{\epsilon_2})\|_2^2+|f^{\epsilon_1}-f^{\epsilon_2}|_{2,+}^2 & \lesssim \|\nu^{-\frac12} (-\epsilon_1f^{\epsilon_1}+\epsilon_2f^{\epsilon_2})\|_2^2 \\
\ &\lesssim  (\epsilon_1^2+\epsilon_2^2)\left(\|\nu^{-\frac12}f^{\epsilon_1}\|_2^2+\|\nu^{-\frac12}f^{\epsilon_2}\|_2^2\right) \notag\\
\ &\lesssim  (\epsilon_1^2+\epsilon_2^2)\left(\| wf^{\epsilon_1}\|_\infty^2+\|\ wf^{\epsilon_2}\|_\infty^2\right) \notag\\
\ &\lesssim  (\epsilon_1^2+\epsilon_2^2)\left(\| \nu^{-1}wg_2\|_\infty^2+ |wr|_{\infty, -}^2 \right). \notag
\end{align}
Finally, applying Corollary \ref{Co1} to $\nu^{\frac12}(f^{\epsilon_1}-f^{\epsilon_2})$ and \eqref{*12}, it follows from $\tau>3+|\gamma|$ that
\begin{align*}
\ &\ \|w\nu^{\frac12}(f^{\epsilon_1}-f^{\epsilon_2})\|_\infty+|w \nu^{\frac12}(f^{\epsilon_1}-f^{\epsilon_2})|_\infty \\
\lesssim &\ \| w(\epsilon_1f^{\epsilon_1}-\epsilon_2f^{\epsilon_2})\|_\infty+\|\nu^{\frac 12}(f^{\epsilon_1}-f^{\epsilon_2})\|_2  \\
\lesssim &\ (\epsilon_1+\epsilon_2)\left(\| \nu^{-1}wg_2\|_\infty+ |wr|_{\infty, -} \right).
\end{align*}
This implies that there exists a function $f$ such that $\|\nu w(f^\epsilon-f)\|_\infty \to 0$ as $\epsilon\to 0^+$, and it holds that
\begin{equation*}
\|wf\|_\infty+|wf|_\infty\leq C \left( \|\nu^{-1}wg_2\|_\infty+|wr|_{\infty, -}\right).
\end{equation*}
Furthermore, since the sequence under consideration always converges in $L^\infty$, the solution is continuous away from $\Gamma_0$.
This completes the proof.
\end{proof}


\section{H\"older regularity of stationary linearized problem}\label{linear}

In this section we study the H\"older regularity of solutions to the stationary linearized Boltzmann problem \eqref{LSBE-0}. 
The main result of this section is stated as follows:
\begin{theorem}\label{T1}
Let $-3<\gamma<0$, and let $\tau>3+|\gamma|$ and $0<\ka<\frac18$. Assume that $| wr|_{\infty, -}<\infty$. Then there exists a unique solution $f$ to problem \eqref{LSBE-0} satisfying
\begin{equation*}
\|wf\|_\infty+|wf|_\infty \leq C|wr|_{\infty, -}
\end{equation*}
for some positive constant $C=C_{\gamma, \tau, \Omega}$. 
Moreover, if $\Omega$ is strictly convex and $r$ is continuous on $\Gamma_-$, then $f$ is continuous on $\overline\Omega\times\mathbb{R}^3\backslash\Gamma_0.$
Assume further that $r$ satisfies \eqref{rb} when $\Omega$ is strictly convex.
Let $0<\alpha< \min\{1, 3+\gamma\}$.
Then, there exist $\widetilde C=\widetilde C_{\alpha, \beta, \gamma, \Omega}>0$ and $c=c_{\gamma}>0$ such that  for every $x_1, x_2\in\Omega$ and $v_1, v_2\in\mathbb{R}^3\setminus\{0\}$,
\begin{align}\label{T1v}
\ & |f(x_1, v_1)-f(x_2, v_2)| \leq \widetilde C(|wr|_{\infty, -}+m) \\
\ & \times \left(d_{(x_1, x_2)}^{-1} (|x_1-x_2|+|v_1-v_2|)(1+|v_{1,2}|^{-1}e^{-c\frac{\nu(v_{1,2})}{|v_{1,2}|}d_{(x_1, x_2)}})\right)^{\min\{\alpha, \beta\}}. \notag
\end{align}
Moreover, it holds that
\begin{align}\label{T1x}
\ &|f(x_1, v_1)-f(x_2, v_2)|\\
\lesssim &\ (|wr|_{\infty, -}+m)\left(d_{(x_1, x_2)}^{-2}(|x_1-x_2|+|v_1-v_2|)\right)^{\min\{\alpha, \beta\}}. \notag
\end{align}
\end{theorem}

In order to prove the result, as mentioned in introduction, we will discuss the H\"older regularity for $F_1(f), F_2(f)$, and $F_3(f)$.
We first investigate the H\"older regularity to $F_{3}(f)$.
Define
\begin{equation*}
G(f)(x,v)=\int_{\mathbb{R}^3}\int_0^{t_b(x, \eta)}k(v,\eta)e^{-\nu(\eta)\tau}Kf(x-\eta\tau, \eta)d\tau d\eta.
\end{equation*}
Then
$$F_3(f)(x,v)=\int_0^{t_b(x,v)}e^{-\nu(v)s}G(x-vs, v)ds.$$
Following the idea of \cite{Ch}, we use the spherical coordinates $(\rho, \phi, \theta)$ by $\eta=(\rho\cos\theta\sin\phi,$ $\rho\sin\theta\sin\phi, \rho\cos\phi)$ and  $|\eta|=\rho$,
then do the change of variable from $\tau$ to $r$ by $r=\rho \tau.$
We then get
\begin{align*}
\ &G(f)(x,v)\\
=&\int_0^\infty\int_0^\pi\int_0^{2\pi}\int_0^{t_b(x, \eta)}k(v,\eta)e^{-\nu(\eta)\tau}Kf(x-\eta\tau, \eta)\rho^2\sin\phi\ d\tau d\theta  d\phi d\rho\\
=&\int_0^\infty\int_0^\pi\int_0^{2\pi}\int_0^{\left|\overline{xx_b(x, \eta)}\right|}k(v, \eta)e^{-\nu(\eta)\frac{r}{\rho}}Kf\left(x-\frac{\eta}{\rho}r, \eta\right)\rho\sin\phi\ dr d\theta  d\phi d\rho.
\end{align*}
From the convexity of the domain $\Omega$, we can parametrize $\Omega$ by $(r, \theta, \phi)$. Let $y=x-\frac{\eta}{\rho}r$, then $y\in\Omega$, $r=|x-y|$, $\eta=\frac{x-y}{|x-y|}\rho$, and
\begin{align*}
\ & G(f)(x, v)\\
=&\int_0^\infty\int_\Omega k\left(v, \frac{(x-y)}{|x-y|}\rho\right)e^{-\nu\left(\frac{x-y}{|x-y|}\rho\right)\frac{|x-y|}{\rho}}Kf\left(y,  \frac{(x-y)}{|x-y|}\rho\right)\frac{\rho}{|x-y|^2}dy d\rho.
\end{align*}

The next lemma is the H\"older regularity of $G(f)$ for space variable.
\begin{lemma}\label{Gx}
Let $-3<\gamma<0$ and $f\in L^\infty(\Omega\times\mathbb{R}^3)$. Let $0<\alpha<\min\{1, 3+\gamma\}$. Then there exists $C=C_{\gamma, \Omega, \alpha}>0$ such that
\begin{equation}\label{ieHx}
|G(f)(x_1, v)-G(f)(x_2, v)|\leq C\|f\|_{\infty}|x_1-x_2|^{\alpha}. 
\end{equation}
for every $x_1, x_2\in\Omega$ and $v\in\mathbb{R}^3\setminus\{0\}$.
\end{lemma}

\begin{proof}
Let $R$ be the diameter of $\Omega$.
By \eqref{A-1}, \eqref{tb1}, and $f\in L^\infty(\Omega\times\mathbb{R}^3)$, we have
\begin{equation}\label{GB}
\left| G(f)(x, v)\right| \lesssim \|f\|_\infty \langle v\rangle^{2\gamma-5}.
\end{equation}
\eqref{GB} implies that \eqref{ieHx} is obvious for $|x_1-x_2|\geq1$. Now, we only need to consider the case of $|x_1-x_2|<1$.
For simplicity of notation, let $X(x)=\frac{x-y}{|x-y|}\rho$.
Then
\begin{align*}
\ &|G(f)(x_1, v)-G(f)(x_2, v)|\\
\leq &\int_0^\infty\int_\Omega \left|k\left(v, X(x_1)\right) \frac {e^{-\nu\left(X(x_1)\right)\frac{|x_1-y|}{\rho}}}{|x_1-y|^2}-k\left(v, X(x_2)\right)\frac{e^{-\nu\left(X(x_2)\right)\frac{|x_2-y|}{\rho}}}{|x_2-y|^2}\right|  \\
\ &\qquad\qquad\qquad\qquad\qquad\times \left| Kf\left(y, X(x_2)\right)\right|\rho\ dy d\rho\\
\ &\ +\int_0^\infty\int_\Omega \left|k\left(v, X(x_1)\right)\right| e^{-\nu\left(X(x_1)\right)\frac{|x_1-y|}{\rho}}\frac{\rho}{|x_1-y|^2}   \\
\ &\qquad\qquad\qquad\qquad\qquad \times \left| Kf\left(y, X(x_1)\right)-Kf\left(y, X(x_2)\right)\right|dy d\rho\\
=:&\ G_k+G_K.
\end{align*}

For $G_k$, we divide the domain of the integration into three subdomains:
\begin{align*}
D_1&=\left\{(\rho, y)\in (0,\infty)\times\Omega : |y-x_1|\geq 4|x_1-x_2|, \left|v-X(x_1)\right|\geq 4\rho\frac{|x_1-x_2|}{|x_1-y|} \right\};\\
D_2&=\left\{(\rho, y)\in (0,\infty)\times\Omega : |y-x_1|\geq 4|x_1-x_2|, \left|v-X(x_1)\right|< 4\rho\frac{|x_1-x_2|}{|x_1-y|} \right\};\\
D_3&=\left\{(\rho, y)\in (0,\infty)\times\Omega : |y-x_1|< 4|x_1-x_2| \right\}.
\end{align*}
It follows from \eqref{A-1} that
\begin{align*}
G_k &\lesssim  \|f\|_{\infty} \int_0^\infty\int_\Omega \langle \rho\rangle^{\gamma-2}\rho \left|k\left(v, X(x_1)\right)\frac {e^{-\nu\left(X(x_1)\right)\frac{|x_1-y|}{\rho}}}{|x_1-y|^2}    \right.\\
\ &\qquad\qquad\qquad\qquad\qquad\qquad\qquad\left. -k\left(v, X(x_2)\right)\frac{e^{-\nu\left(X(x_2)\right)\frac{|x_2-y|}{\rho}}}{|x_2-y|^2}\right|dy d\rho\\
\ &\lesssim \|f\|_\infty \left( \iint_{D_1} \cdots\cdots +\iint_{D_2}\cdots\cdots+\iint_{D_3}\cdots\cdots\right)\\
\ &=: \|f\|_\infty \left( G_{k1}+G_{k2}+G_{k3}\right).
\end{align*}
Applying the fundamental theorem of calculus to $G_{k1}$, we have
\begin{align*}
G_{k1}=\iint_{D_1}\langle\rho\rangle^{\gamma-2}\rho \left|\int_0^1\frac{d}{ds}\left[k\left(v, X(x(s))\right)\frac{e^{-\nu\left(X(x(s))\right)\frac{|x(s)-y|}{\rho}}}{|x(s)-y|^2}\right] ds\right|dy d\rho,
\end{align*}
where $x(s)=x_2+s(x_1-x_2)$.
With straight calculations one has
\begin{align*}
\ &\left|\frac{d}{ds}\left[k\left(v, X(x(s))\right)\frac{e^{-\nu\left(X(x(s))\right)\frac{|x(s)-y|}{\rho}}}{|x(s)-y|^2}\right]\right|\notag\\
\leq &\ |x_1-x_2| e^{-\nu\left(X(x(s))\right)\frac{|x(s)-y|}{\rho}} \left( \left| \nabla_{\eta}k\left(v, X(x(s))\right)\right|  \frac{2\rho }{|x(s)-y|^3}\right.\notag\\
&\left. +\left|k\left(v, X(x(s))\right)\right| \left[\frac{2}{|x(s)-y|^3}+\frac{\nu\left(X(x(s))\right)}{\rho|x(s)-y|^2} +\frac{2\left|\nabla\nu\left(X(x(s))\right)\right|}{|x(s)-y|^2} \right]\right).\notag
\end{align*}
Thus,
\begin{align*}
\ & G_{k1}\\
\lesssim &\ |x_1-x_2| \int_0^1\int_{D_1} \left| \nabla_{\eta}k\left(v, X(x(s))\right)\right| e^{-\nu\left(X(x(s))\right)\frac{|x(s)-y|}{\rho}} \frac{\langle\rho\rangle^{\gamma-2}\rho^2}{|x(s)-y|^3}\ dy d\rho ds\\
\ &+|x_1-x_2|\int_0^1\int_{D_1} \left|k\left(v, X(x(s))\right)\right|e^{-\nu\left(X(x(s))\right)\frac{|x(s)-y|}{\rho}}\frac{\langle\rho\rangle^{\gamma-2}\rho}{|x(s)-y|^3} \ dy d\rho ds   \\
\ &+|x_1-x_2|\int_0^1\int_{D_1} \left|k\left(v, X(x(s))\right)\right|e^{-\nu\left(X(x(s))\right)\frac{|x(s)-y|}{\rho}}\frac{\langle\rho\rangle^{2\gamma-2}}{|x(s)-y|^2}\ dy d\rho ds   \\
\ &+|x_1-x_2|\int_0^1\int_{D_1} \left|k\left(v, X(x(s))\right)\right|e^{-\nu\left(X(x(s))\right)\frac{|x(s)-y|}{\rho}}\frac{\langle\rho\rangle^{2\gamma-3}\rho}{|x(s)-y|^2} \ dy d\rho ds   \\
=:&\ G_{k11}+G_{k12}+G_{k13}+G_{k14}.
\end{align*}
It is easy to observe that
\begin{align}
\left|\frac{x_1-y}{|x_1-y|}-\frac{x_2-y}{|x_2-y|}\right| & =\left|\frac{(x_1-x_2)|x_2-y|+(x_2-y)\left(|x_2-y|-|x_1-y|\right)}{|x_1-y||x_2-y|}\right| \label{Gk1}  \\
\ &\leq 2\frac{|x_1-x_2|}{|x_1-y|}.\notag
\end{align}
For $(\rho, y)\in D_1$, by \eqref{Gk1}, one has
\begin{align}
|y-x_2| & \geq\frac 34|y-x_1|\geq 3|x_1-x_2|; \label{Gk2}\\
|y-x(s)| & \geq \frac 12|y-x_1|;\label{Gk3}\\
\left|v-X(x(s))\right| & \geq 4\rho\frac{|x_1-x_2|}{|x_1-y|}-2\rho\frac{(1-s)|x_1-x_2|}{|x_1-y|}\geq \rho\frac{|x_1-x_2|}{|x(s)-y|}.\label{Gk4}
\end{align}
Applying the parametrizations $x(s)-y=(r\cos\theta\sin\phi, r\sin\theta\sin\phi, r\cos\phi)$ and $\eta=(\rho\cos\theta\sin\phi, \rho\sin\theta\sin\phi, \rho\cos\phi)$ and using \eqref{A-1}, \eqref{e2}, \eqref{Gk3}, and \eqref{Gk4}, we have
\begin{align*}
\ &G_{k11}\\
\lesssim &\ |x_1-x_2|\int_0^1\int_{D_1}\left| \nabla_{\eta}k\left(v, \eta\right)\right| e^{-\frac{\nu(\eta)}{\rho}r} \frac{\langle \rho \rangle^{\gamma-2}\rho^2}{r}\sin\phi \ dr d\theta d\phi  d\rho ds\\
\lesssim &\ |x_1-x_2|\langle v\rangle^\gamma \int_0^1\int_{D_1} |v-\eta|^{\bar\gamma-1}e^{-\frac{1}{16}\left(|v-\eta|^2+\frac{\left(|v|^2-|\eta|^2\right)^2}{|v-\eta|^2}\right)}\frac{\langle \rho \rangle^{\gamma-2}\rho^2}{r} \\
\ &\qquad\qquad\qquad\qquad\qquad  e^{-\nu(\eta)\frac{r}{\rho}}\sin\phi\ dr d\theta d\phi d\rho ds\\
\lesssim &\ |x_1-x_2|\langle v\rangle^\gamma \int_0^1\int_{D_1} |v-\eta|^{\max\{\bar\gamma-1, -3+\epsilon \}}e^{-\frac{1}{16}\left(|v-\eta|^2+\frac{\left(|v|^2-|\eta|^2\right)^2}{|v-\eta|^2}\right)}\frac{\langle \rho \rangle^{\gamma-2}\rho^2}{r}\\
\ & \qquad\qquad \qquad\qquad\qquad \left(\frac{\rho}{r}|x_1-x_2|\right)^{\min\{0, 2+\gamma-\epsilon\}} e^{-\nu(\eta)\frac{r}{\rho}}\sin\phi\ dr d\theta d\phi d\rho ds\\
\lesssim &\ |x_1-x_2|^{ \min\{1, 3+\gamma-\epsilon \} } \langle v\rangle^\gamma \left(\int_{2|x_1-x_2|}^R r^{-1+a_1} dr\right) \times \\
\ &\left( \int_{\mathbb{R}^3}|v-\eta|^{\max\{\bar\gamma-1, -3+\epsilon \}}e^{-\frac{1}{16}\left(|v-\eta|^2+\frac{\left(|v|^2-|\eta|^2\right)^2}{|v-\eta|^2}\right)}\langle \eta \rangle^{\gamma-2}|\eta|^{-a_1} d\eta\right) \\
\lesssim &\  |x_1-x_2|^{ \min\{1-\epsilon, 3+\gamma-\epsilon \} }, 
\end{align*}
where $\bar\gamma$ is as in \eqref{bg}, $a_1=0$ if $-2<\gamma<0$ and $a_1=\frac{\epsilon}{2}$ if $-3<\gamma \leq -2$.

By using \eqref{A-1}, \eqref{e2}, and \eqref{Gk3}, we deduce
\begin{align*}
\ &G_{k12}\\
\lesssim &\ |x_1-x_2|\int_{D_1} \left| k\left(v, \eta\right)\right|  e^{-\nu(\eta)\frac{r}{\rho}}\frac{\langle \rho \rangle^{\gamma-2}\rho}{r}\sin\phi \ dr  d\theta d\phi d\rho \\
\lesssim &\ |x_1-x_2|\int_{\mathbb{R}^3}\left| k\left(v, \eta \right)\right|\frac{ \langle \eta \rangle^{\gamma-2}}{|\eta|} \left(\int_{2|x_1-x_2|}^R \frac{1}{r} \left(\frac{|\eta|}{r}\right)^{a_2}\left(\frac{r}{|\eta|}\right)^{a_2}e^{-\frac{\nu(\eta)}{|\eta|}r} dr\right)d\eta\\
\lesssim &\ |x_1-x_2|\left(\int_{\mathbb{R}^3}\left| k\left(v, \eta\right)\right| \langle \eta\rangle^{\gamma-2} |\eta|^{-1+a_2} d\eta\right)\left(\int_{2|x_1-x_2|}^R r^{-1-a_2} dr\right)\\
\lesssim &\ |x_1-x_2|^{\min\{1-\epsilon, 3+\gamma-\epsilon\}}, 
\end{align*}
\begin{align*}
\ &G_{k13}\\
\lesssim &\ |x_1-x_2|\int_{\mathbb{R}^3}\left| k\left(v, \eta\right)\right|\frac{ \langle \eta \rangle^{2\gamma-2}}{|\eta|^{2}}\left(\int_{2|x_1-x_2|}^R \left(\frac{|\eta|}{r}\right)^{1+a_2}\left(\frac{r}{|\eta|}\right)^{1+a_2}e^{-\frac{\nu(\eta)}{|\eta|}r} dr\right)d\eta\\
\lesssim &\ |x_1-x_2| \left(\int_{\mathbb{R}^3}\left| k\left(v, \eta\right)\right| \langle \eta \rangle^{2\gamma-2} |\eta|^{-1+a_2}  d\eta\right)\left(\int_{2|x_1-x_2|}^R r^{-1-a_2} dr\right)\\
\lesssim &\  |x_1-x_2|^{\min\{1-\epsilon, 3+\gamma-\epsilon\}}, 
\end{align*}
and
\begin{align*}
\ &G_{k14}\\
\lesssim &\ |x_1-x_2|\int_{\mathbb{R}^3}\left| k\left(v, \eta\right)\right| \langle \eta \rangle^{2\gamma-3}|\eta|^{-1}\left(\int_{2|x_1-x_2|}^R \left(\frac{r}{|\eta|}\right)^{-1 }\left(\frac{r}{|\eta|}\right) e^{-\frac{\nu(\eta)}{|\eta|}r} dr\right)d\eta\\
\lesssim &\ |x_1-x_2| \left(\int_{\mathbb{R}^3}\left| k\left(v, \eta\right)\right| \langle \eta \rangle^{2\gamma-3} d\eta\right)\left(\int_{2|x_1-x_2|}^R r^{-1} dr\right)\\
\lesssim &\  |x_1-x_2|^{1-\epsilon}, 
\end{align*}
where $a_2=\max\{|\gamma|-2-\epsilon, 0\}$.

For $(\rho, y)\in D_2$, by \eqref{Gk1}, one has \eqref{Gk2},
\begin{align}
|x_2-y| & \leq |x_1-y|+|x_1-x_2|\leq \frac 54|x_1-y|; \notag\\
\left|v-\frac{x_2-y}{|x_2-y|}\rho\right| & <\left|v-\frac{x_1-y}{|x_1-y|}\rho\right|+ \left| \frac{x_1-y}{|x_1-y|}\rho-\frac{x_2-y}{|x_2-y|}\rho\right|<\frac{15}{2}\rho\frac{|x_1-x_2|}{|x_2-y|}.\label{Gk6}
\end{align}
Applying the parametrizations $x_i-y=(r\cos\theta\sin\phi, r\sin\theta\sin\phi, r\cos\phi)$ and $\eta=(\rho\cos\theta\sin\phi, \rho\sin\theta\sin\phi, \rho\cos\phi)$ and using \eqref{A-1}, \eqref{e2}, \eqref{Gk2}, and \eqref{Gk6}, we have
\begin{align*}
\ &\ G_{k2}\\
\lesssim & \sum_{i=1}^2\int_{D_2}\left| k\left(v, X(x_i)\right)\right| e^{-\nu\left(X(x_i)\right)\frac{|x_i-y|}{\rho}}\langle\rho\rangle^{\gamma-2}\rho |x_i-y|^{-2}  dy d\rho\\
\lesssim&  \int_{\tiny\begin{array}{l}
r\geq 3|x_1-x_2| \\
|v-\eta|<8\frac{\rho}{r}|x_1-x_2|
\end{array}}
\left| k\left(v, \eta\right)\right| e^{-\nu(\eta)\frac{r}{\rho}} \langle \rho\rangle^{\gamma-2}\rho \sin\phi \ dr  d\theta d\phi d\rho \\
\lesssim& \int_{\tiny\begin{array}{l}
r\geq 3|x_1-x_2| \\
|v-\eta|<8\frac{\rho}{r}|x_1-x_2|
\end{array}}
\left| k\left(v, \eta\right)\right| \frac{\left(\frac{\rho}{r}|x_1-x_2|\right)^{a_3}}{|v-\eta|^{a_3}} \langle \rho\rangle^{\gamma-2}\rho e^{-\nu(\eta)\frac{r}{\rho}}\sin\phi \ dr  d\theta d\phi d\rho\\
\lesssim &\ |x_1-x_2|^{a_3} \int_{\mathbb{R}^3}\left| k\left(v, \eta\right)\right| |v-\eta|^{-a_3}\langle \eta\rangle^{\gamma-2}|\eta|^{-1+a_3}\\
\ & \quad \times \left(\int_{3|x_1-x_2|}^R r^{-a_3}\left(\frac{|\eta|}{r}\right)^{\max\{0, |\gamma|-2+\frac{\epsilon}{2}\}}\left(\frac{r}{|\eta|}\right)^{\max\{0, |\gamma|-2+\frac{\epsilon}{2}\}} e^{-\frac{\nu(\eta)}{|\eta|}r}  dr \right) d\eta \\
\lesssim &\  |x_1-x_2|^{a_3} \left( \int_{\mathbb{R}^3}\left| k\left(v, \eta\right)\right| |v-\eta|^{-a_3}\langle \eta\rangle^{\gamma-2}|\eta|^{-a_1} d\eta\right) \left(\int_{3|x_1-x_2|}^R r^{-1+a_1} dr\right) \\
\lesssim &\ |x_1-x_2|^{\min\{1-\epsilon, 3-|\gamma|-\epsilon\} }, 
\end{align*}
where $a_1$ is as above and $a_3=\min\{1, 3+\gamma-\epsilon\}>0$.

For $G_{k3}$, we have $|y-x_2|\leq |y-x_1|+|x_1-x_2|< 5|x_1-x_2|$. Thus
\begin{align*}
\ &G_{k3}\\
\lesssim &\ \int_0^\infty\int_0^{2\pi}\int_0^\pi\int_{0}^{5|x_1-x_2|} \left| k\left(v, \eta\right)\right| \langle \rho\rangle^{\gamma-2}\rho e^{-\nu(\eta) \frac{r}{\rho}}\sin\phi \ dr  d\theta d\phi d\rho \\
\lesssim &\ \int_{\mathbb{R}^3}\left| k\left(v, \eta\right)\right| \langle \eta\rangle^{\gamma-2}|\eta|^{-1}\left(\int_0^{5|x_1-x_2|}\left(\frac{r}{|\eta|}\right)^{-a_2}\left(\frac{r}{|\eta|}\right)^{a_2}e^{-\nu(\eta)\frac{r}{|\eta|}} dr\right)d\eta\\
\lesssim &\ \left(\int_{\mathbb{R}^3}\left| k\left(v, \eta\right)\right| \langle \eta\rangle^{\gamma-2}|\eta|^{-1-a_2} d\eta\right)\left(\int_0^{5|x_1-x_2|} r^{-a_2} dr\right)\\
\lesssim &\ |x_1-x_2|^{\min\{1, 3+\gamma-\epsilon\}}, 
\end{align*}
where $a_2$ is as above.
Hence
\begin{equation*}
G_k\lesssim\|f\|_{\infty}|x_1-x_2|^{\min\{1-\epsilon, 3+\gamma-\epsilon\}}. 
\end{equation*}

For $G_K$, we divide the domain of the integration into two subdomains: $\widetilde D_1=\left\{y\in \Omega : |y-x_1|\geq |x_1-x_2|\right\}$ and $\widetilde D_2=\left\{y\in\Omega : |y-x_1|< |x_1-x_2|\right\}$. We name the corresponding integrals as $G_{K1}$ and $G_{K2}$, respectively.
Using Lemma \ref{kH}, \eqref{A-1}, \eqref{e2}, and \eqref{Gk1}, we deduce
\begin{align*}
\ &G_{K1}\\
\lesssim &\ \|f\|_{\infty}\int_0^\infty\int_{\widetilde D_1} \left|k\left(v, X(x_1)\right)\right| e^{-\nu\left(X(x_1)\right)\frac{|x_1-y|}{\rho}}\frac{\rho\langle \rho\rangle^{\gamma-1}}{|x_1-y|^2} \left| X(x_1)-X(x_2)\right|^{\widetilde\alpha_\gamma}dy d\rho\\
\lesssim &\ \|f\|_{\infty}|x_1-x_2|^{\widetilde\alpha_\gamma}\int_0^\infty\int_{\widetilde D_1} \left|k\left(v, X(x_1) \right)\right| e^{-\nu\left(X(x_1)\right)\frac{|x_1-y|}{\rho}} \frac{\langle\rho\rangle^{\gamma-1} \rho^{1+\widetilde\alpha_\gamma} }{|x_1-y|^{2+\widetilde\alpha_\gamma}} dyd\rho\\
\lesssim &\ \|f\|_{^\infty}|x_1-x_2|^{\widetilde\alpha_\gamma}\int_{\mathbb{R}^3}\left|k\left(v, \eta\right)\right|\langle \eta\rangle^{\gamma-1} |\eta|^{-1+\widetilde\alpha_\gamma}\\
\ &\qquad\qquad\qquad\qquad\qquad \left(\int_{|x_1-x_2|}^Rr^{-\widetilde\alpha_\gamma}\left(\frac{r}{|\eta|}\right)^{-1+\widetilde\alpha_\gamma}\left(\frac{r}{|\eta|}\right)^{1-\widetilde\alpha_\gamma}e^{-\frac{\nu(\eta)}{|\eta|}r}dr\right) d\eta\\
\lesssim &\ \|f\|_{\infty}|x_1-x_2|^{\widetilde\alpha_\gamma}\left(\int_{\mathbb{R}^3} \left|k\left(v, \eta\right)\right|\langle\eta\rangle^{\gamma-1} d\eta\right)\left(\int_{|x_1-x_2|}^Rr^{-1} dr\right)\\
\lesssim &\ \|f\|_{\infty}|x_1-x_2|^{\min\{1-\epsilon, 3+\gamma-\epsilon\}}. 
\end{align*}
For $G_{K2}$, by Lemma \ref{kH}, \eqref{A-1}, and \eqref{e2}, it follows from $\left| \frac{x_1-y}{|x_1-y|}-\frac{x_2-y}{|x_2-y|}\right|\leq 2$ that
\begin{align*}
\ &G_{K2}\\
\lesssim &\ \|f\|_{\infty}\int_0^\infty\int_{D_2} \left|k\left(v, X(x_1)\right)\right| e^{-\nu\left(X(x_1)\right)\frac{|x_1-y|}{\rho}}\frac{\rho\langle \rho\rangle^{\gamma-1}}{|x_1-y|^2} \left| X(x_1)-X(x_2)\right|^{\widetilde\alpha_\gamma}dy d\rho\\
\lesssim &\ \|f\|_{\infty}\int_0^\infty\int_0^{\pi}\int_0^{2\pi}\int_0^{|x_1-x_2|} \left|k\left(v, \eta \right)\right| e^{-\nu(\eta)\frac{r}{\rho}}\langle\rho\rangle^{\gamma-1}\rho^{1+\widetilde\alpha_\gamma}\sin\phi dr d\theta d\phi d\rho\\
\lesssim &\ \|f\|_{\infty}\int_{\mathbb{R}^3} \left|k\left(v, \eta\right)\right| \langle\eta\rangle^{\gamma-1} |\eta|^{\widetilde\alpha_\gamma-1}\int_0^{|x_1-x_2|}\left(\frac{r}{|\eta|} \right)^{-1+\widetilde\alpha_\gamma}\left(\frac{r}{|\eta|}\right)^{1-\widetilde\alpha_\gamma} e^{-\frac{\nu(\eta)}{|\eta|}r} dr d\eta\\
\lesssim &\ \|f\|_{\infty}\left(\int_{\mathbb{R}^3} \left|k\left(v, \eta\right)\right| \langle\eta\rangle^{\gamma-1} d\eta \right)\int_0^{|x_1-x_2|}r^{-1+\widetilde\alpha_\gamma}dr \\
\lesssim &\ \|f\|_\infty |x_1-x_2|^{\widetilde\alpha_\gamma}. 
\end{align*}
Hence
\begin{equation*}
G_K\lesssim\|f\|_{\infty}|x_1-x_2|^{\min\{1-\epsilon, 3+\gamma-\epsilon\}}. 
\end{equation*}
We complete the proof of the lemma.
\end{proof}

The next lemma is the H\"older regularity of $G(f)$ for velocity variable.
\begin{lemma}\label{Gv}
Let $-3<\gamma<0$ and $f\in L^\infty(\Omega\times\mathbb{R}^3)$. Then, for any $x\in\Omega$ and $v_1, v_2\in\mathbb{R}^3$
\begin{align*}
|G(f)(x, v_1)-G(f)(x,v_2)|&\lesssim  \|f\|_{\infty}|v_1-v_2|^{\widetilde\alpha_\gamma}, 
\end{align*}
where $\widetilde\alpha_\gamma$ is as in Lemma \ref{kH}.
\end{lemma}

\begin{proof}
Using Lemma \ref{kH}, \eqref{A-1}, and \eqref{tb1}, we have
\begin{align*}
\ &|G(f)(x, v_1)-G(f)(x, v_2)|\\
\leq &\ \int_{\mathbb{R}^3}\int_0^{t_b(x,\eta)}|k(v_1, \eta)-k(v_2, \eta)|e^{-\nu(\eta)\tau}\left|Kf(x-\eta\tau, \eta)\right| d\tau d\eta\\
\lesssim &\ \|f\|_{\infty}\int_{\mathbb{R}^3}|k(v_1, \eta)-k(v_2, \eta)|\langle\eta\rangle^{\gamma-2} \left(\int_0^{t_b(x, \eta)}e^{-\nu(\eta)\tau}d\tau\right)d\eta\\
\lesssim &\ \|f\|_{\infty}|v_1-v_2|^{\widetilde\alpha_\gamma}\max\{\langle v_1\rangle^{2\gamma-4}, \langle v_2\rangle^{2\gamma-4}\}.
\end{align*}
This completes the proof of the lemma.
\end{proof}

Now, we are ready to estimate $F_{3}(f)$.
\begin{lemma}\label{LF3}
Let $-3<\gamma<0$ and $f\in L^\infty(\Omega\times\mathbb{R}^3)$, and let $0<\alpha<\min\{1, 3+\gamma\}$. Then, for every $x, x_1, x_2\in\Omega$ and for every $v, v_1, v_2\in\mathbb{R}^3\setminus\{0\}$,
\begin{align}
\ & \left| F_3(f)(x_1,v)-F_3(f)(x_2, v)\right| \label{F3x} \\
\lesssim &\ \|f\|_{\infty}\left( d_{(x_1, x_2)}^{-1}|x_1-x_2|(1+|v|^{-1}e^{-\frac{\nu(v)}{|v|}d_{(x_1, x_2)}})\right)^{\alpha}; \notag\\
\ &\left| F_3(f)(x, v_1)-F_3(f)(x, v_2)\right| \lesssim \|f\|_{\infty}\left( d_{x}^{-1}|v_1-v_2|\right)^{\alpha}.  \label{F3v}
\end{align}
\end{lemma}

\begin{proof}
By \eqref{tb1} and \eqref{GB}, we have
\begin{equation}\label{F3B}
|F_3(f)(x,v)|\leq \int_0^{t_b(x, v)}e^{-\nu(v)s}\left| G(f)(x-vs, v)\right| ds\lesssim \|f\|_{\infty}\langle v\rangle^{2\gamma-6}.
\end{equation}
For \eqref{F3x}, since \eqref{F3B}, it is obvious for $d_{(x_1, x_2)}^{-1}|x_1-x_2| (1+|v|^{-1}e^{-\frac{\nu(v)}{|v|}d_{(x_1, x_2)}})\geq 1$.
Therefore, we consider the case where  $d_{(x_1, x_2)}^{-1}|x_1-x_2| (1+|v|^{-1}e^{-\frac{\nu(v)}{|v|}d_{(x_1, x_2)}})< 1$.
Moreover, we assume without loss of generality that $t_b(x_1, v)\geq t_b(x_2, v)$.
By using Lemma \ref{Gx}, \eqref{tb1}, \eqref{tb2}, and \eqref{GB}, we have
\begin{align*}
\ & \left|F_3(f)(x_1, v)-F_3(f)(x_2, v)\right| \label{F3-1}\\
\leq &\ \int_0^{t_b(x_2, v)}e^{-\nu(v)s}\left| G(f)(x_1-vs, v)-G(f)(x_2-vs, v)\right| ds \notag\\
\ &\ +\int_{t_b(x_2, v)}^{t_b(x_1, v)}e^{-\nu(v)s}\left| G(f)(x_1-vs, v)\right| ds \notag\\
\lesssim &\ \|f\|_{\infty}\left(|x_1-x_2|^{\alpha}+\langle v\rangle^{2\gamma-5}\int_{t_b(x_2, v)}^{t_b(x_1, v)}e^{-\nu(v)s} ds\right) \notag\\
\lesssim &\ \|f\|_{\infty}\left( |x_1-x_2|^{\alpha}+d_{(x_1, x_2)}^{-1}|x_1-x_2| |v|^{-1}e^{-\frac{\nu(v)}{|v|}d_{(x_1, x_2)}} \right). \notag \\
\lesssim &\ \|f\|_{\infty}\left(d_{(x_1, x_2)}^{-1}|x_1-x_2|(1+|v|^{-1}e^{-\frac{\nu(v)}{|v|}d_{(x_1, x_2)}})\right)^{\alpha}.
\end{align*}

For \eqref{F3v}, since \eqref{F3B}, it is obvious for $d_x^{-1}|v_1-v_2|\geq 1$. Therefore, we consider the case of $d_x^{-1}|v_1-v_2|< 1$.
Moreover, we only do the $|v_1|\geq |v_2|$ case.
Let $\bar v_2=\frac{|v_1|}{|v_2|}v_2$, it is easy to obtain that $|v_2-\bar v_2|\leq |v_1-v_2|$ and $|v_1-\bar v_2|\leq 2|v_1-v_2|$.
Then
\begin{align*}
\ &\left| F_3(f)(x, v_1)-F_3(f)(x, v_2) \right| \\
\leq &\ \left|F_3(f)(x, v_1)-F_3(f)(x, \bar v_2)\right|+\left| F_3(f)(x, \bar v_2)-F_3(f)(x, v_2)\right|.
\end{align*}
Then we assume without loss of generality that $t_b(x, v_1)\geq t_b(x, \bar v_2)$. Using Lemmas \ref{Gx} and \ref{Gv}, \eqref{tb1}, \eqref{tb3}, \eqref{tb4}, and \eqref{GB}, we have
\begin{align*}
\ &|F_3(f)(x, v_1)-F_3(f)(x, \bar v_2)|\\
\leq &  \int_0^{t_b(x, \bar v_2)}e^{-\nu(\bar v_2)s}\left| G(f)(x-v_1s, v_1)-G(f)(x- \bar v_2s, \bar v_2)\right| ds\\
\ & + \int_{0}^{t_b(x, \bar v_2)}\left|e^{-\nu(v_1)s}-e^{-\nu(\bar v_2)s}\right|\left| G(f)(x-v_1s, v_1)\right| ds \\
\ &+ \int_{t_b(x, \bar v_2)}^{t_b(x, v_1)}e^{-\nu(v_1)s}\left| G(f)(x-v_1s, v_1)\right| ds \\
\lesssim &\ \|f\|_{\infty}\left( \int_0^{t_b(x, \bar v_2)}e^{-\nu(\bar v_2)s}\left(| v_1s- \bar v_2s|^{\alpha}+| v_1- \bar v_2|^{\widetilde\alpha_\gamma}\right) ds\right.\\
\ &\ \left.\qquad + \langle v_1\rangle^{2\gamma-5}\int_{0}^{t_b(x, \bar v_2)}\left|e^{-\nu(v_1)s}-e^{-\nu(\bar v_2)s}\right| ds+ \langle v_1\rangle^{2\gamma-5} \int_{t_b(x, \bar v_2)}^{t_b(x, v_1)}e^{-\nu(v_1)s}ds\right)\\
\lesssim &\ \|f\|_{\infty}\left( |v_1-v_2|^{\alpha}+|v_1-v_2|^{\widetilde\alpha_\gamma}+|v_1-v_2|+d_x^{-1}|v_1-v_2|\right)\\
\lesssim &\ \|f\|_{\infty}\left( d_x^{-1} |v_1-v_2|\right)^{\alpha}.
\end{align*}
For the remain term, we now have $t_b(x, v_2)\geq t_b(x, \bar v_2)$. Using Lemmas \ref{Gx} and \ref{Gv}, \eqref{tb1}, \eqref{tb3}, \eqref{tb5}, and \eqref{GB}, we have
\begin{align*}
\ &|F_3(f)(x, v_2)-F_3(f)(x, \bar v_2)|\\
\leq &  \int_0^{t_b(x, \bar v_2)}e^{-\nu(\bar v_2)s}\left| G(f)(x-v_2s, v_2)-G(f)(x-\bar v_2s, \bar v_2)\right| ds\\
\ & + \int_{0}^{t_b(x, \bar v_2)}\left|e^{-\nu(v_2)s}-e^{-\nu(\bar v_2)s}\right|\left| G(f)(x-v_2s, v_2)\right| ds \\
\ &+ \int_{t_b(x, \bar v_2)}^{t_b(x, v_2)}e^{-\nu(v_2)s}\left| G(f)(x-v_2s, v_2)\right| ds \\
\lesssim &\ \|f\|_{\infty}\left( \int_0^{t_b(x, \bar v_2)}e^{-\nu(\bar v_2)s}\left(|v_2s-\bar v_2s|^{\alpha}+|v_2-\bar v_2|^{\widetilde\alpha_\gamma}\right) ds \right.\\
\ &\ \left. \qquad\qquad+ \int_{0}^{t_b(x, \bar v_2)}\left|e^{-\nu(v_2)s}-e^{-\nu(\bar v_2)s}\right| ds+ \int_{t_b(x, \bar v_2)}^{t_b(x, v_2)}e^{-\nu(v_2)s}ds\right)\\
\lesssim &\ \|f\|_{\infty}\left( |v_1-v_2|^{\alpha}+|v_1-v_2|^{\widetilde\alpha_\gamma}+|v_1-v_2|+d_x^{-1}|v_1-v_2|\right)\\
\lesssim &\ \|f\|_{\infty}\left( d_x^{-1} |v_1-v_2|\right)^{\alpha}.
\end{align*}
This completes the proof.
\end{proof}




The following lemma is the H\"older regularity of $F_1(f)$. 

\begin{lemma}\label{LF1}
Let $-3<\gamma<0$ and $0<\beta\leq 1$, and let $0<\alpha<\min\{1, 3+\gamma\}$. Suppose that $f\in L^\infty(\Gamma_-)$ and there exists $m>0$ such that
\begin{equation}\label{fb}
|f(y_1, v_1)-f(y_2, v_2)|\leq m\left(|y_1-y_2|+|v_1-v_2|\right)^\beta
\end{equation}
for every $(y_i, v_i)\in \Gamma_-$. Then for every $x, x_1, x_2\in\Omega$ and $v, v_1, v_2\in\mathbb{R}^3\setminus\{0\}$,
\begin{align}
\ & \left|F_1(f)(x_1, v)-F_1(f)(x_2, v)\right|  \label{F1x}\\
\lesssim &\ (m+|f|_{\infty, -})\left(d_{(x_1, x_2)}^{-1} |x_1-x_2|(1+|v|^{-1}e^{-\frac{\nu(v)}{|v|}d_{(x_1, x_2)}})\right)^{\min\{\alpha, \beta\} }; \notag\\
\ & \left|F_1(f)(x, v_1)-F_1(f)(x, v_2)\right| \label{F1v}\\
\lesssim &\ (m+|f|_{\infty, -}) \left( d_x^{-1} |v_1-v_2|(1+|v_{1,2}|^{-1}e^{-\frac{\nu(v_{1,2})}{|v_{1,2}|}d_{x}} )\right)^{\min\{\alpha, \beta \}}. \notag
\end{align}
\end{lemma}

\begin{proof}
By $f\in L^\infty(\Gamma_-)$, we have
\begin{equation}\label{F1B}
\left| F_1(f)(x, v) \right| \leq |f|_{\infty, -}.
\end{equation}
For \eqref{F1x}, since \eqref{F1B}, it  is obvious for $d_{(x_1, x_2)}^{-1} |x_1-x_2|(1+|v|^{-1}e^{-\frac{\nu(v)}{|v|}d_{(x_1, x_2)}})\geq 1$.
Therefore, we consider the case where $d_{(x_1, x_2)}^{-1} |x_1-x_2|(1+|v|^{-1}e^{-\frac{\nu(v)}{|v|}d_{(x_1, x_2)}})< 1$.
By using \eqref{xbx}, \eqref{tb2}, and \eqref{fb}, we have
\begin{align*}
\ & \left| F_1(f)(x_1, v)-F_1(f)(x_2, v)\right| \\
\leq&\ \left|f(x_b(x_1,v), v)-f(x_b(x_2,v), v)\right|e^{-\nu(v)t_b(x_1,v)}\notag \\
\ &\ +\left|f(x_b(x_2,v), v)\right|\left|e^{-\nu(v)t_b(x_1,v)}-e^{-\nu(v)t_b(x_2,v)}\right| \notag \\
\leq&\ m\left|x_b(x_1, v)-x_b(x_2,v)\right|^\beta + |f|_{\infty, -}\left|\int_{t_b(x_2,v)}^{t_b(x_1, v)}-\nu(v)e^{-\nu(v)s}ds\right| \notag \\
\lesssim &\ m\left(d_{(x_1, x_2)}^{-1} |x_1-x_2|\right)^\beta +|f|_{\infty, -}d_{(x_1, x_2)}^{-1} |x_1-x_2| |v|^{-1}e^{-\frac{\nu(v)}{|v|}d_{(x_1, x_2)}} \notag \\
\lesssim &\ (m+|f|_{\infty, -})\left(d_{(x_1, x_2)}^{-1} |x_1-x_2|(1+|v|^{-1}e^{-\frac{\nu(v)}{|v|}d_{(x_1, x_2)}})\right)^{\min\{\alpha, \beta\} }.
\end{align*}

For \eqref{F1v}, we only do the case of $|v_1|\geq |v_2|$. From \eqref{F1B}, \eqref{F1v} is obvious for $d_x^{-1} |v_1-v_2|(1+|v_1|^{-1}e^{-\frac{\nu(v_1)}{|v_1|}d_{x}} )\geq 1$. Therefore, we consider the case for $d_x^{-1} |v_1-v_2|(1+|v_1|^{-1}e^{-\frac{\nu(v_1)}{|v_1|}d_{x}} )< 1$.
Let $\bar v_2:=\frac{|v_1|}{|v_2|}v_2$. Then
\begin{align*}
\ &\ \left|F_1(f)(x, v_1)-F_1(f)(x, v_2)\right| \\
\leq &\ \left|F_1(f)(x, v_1)-F_1(f)(x, \bar v_2)\right|+\left|F_1(f)(x, \bar v_2)-F_1(f)(x, v_2)\right|.
\end{align*}
Note that $|v_2-\bar v_2|\leq |v_1-v_2|$, $|v_1-\bar v_2|\leq 2|v_1-v_2|$, and $x_b(x, v_2)=x_b(x, \bar v_2)$.
By using \eqref{xbv}, \eqref{tb1}, \eqref{tb3}, \eqref{tb4}, \eqref{fb}, and the mean value theorem, we have
\begin{align*}
\ &\left| F_1(f)(x, v_1)-F_1(f)(x, \bar v_2)\right|\\
\leq &\ \left|f(x_b(x, v_1), v_1)-f(x_b(x, \bar v_2), \bar v_2)\right|e^{-\nu(v_1)t_b(x, v_1)}+\left|f(x_b(x, \bar v_2), \bar v_2)\right| \\
\ &\ \times \left(\left|e^{-\nu(v_1)t_b(x, v_1)}-e^{-\nu(v_1)t_b(x, \bar v_2)}\right|+\left|e^{-\nu(v_1)t_b(x, \bar v_2)}-e^{-\nu(\bar v_2)t_b(x, \bar v_2)}\right| \right)\\
\leq &\ m\left(|x_b(x, v_1)-x_b(x, \bar v_2)|+|v_1-\bar v_2|\right)^\beta e^{-\nu(v_1)t_b(x, v_1)}\\
\ & +|f|_{\infty, -}\left| \int_{t_b(x, \bar v_2)}^{t_b(x, v_1)}-\nu(v_1)e^{-\nu(v_1)s}ds\right|\\
\ & +|f|_{\infty, -}|\bar v_1-v_2| t_b(x, v_2)  |\nabla \nu(v_1(t))|  e^{-\nu(v_1(t))t_b(x, v_2)} \\
\lesssim &\ m\left(d_x^{-1}\theta_{(v_1, \bar v_2)}+|v_1-v_2|\right)^\beta e^{-\nu(v_1)t_b(x, v_1)}+|f|_{\infty, -}d_x^{-1}|v_1-v_2| +|f|_{\infty, -}|v_1-v_2|\\
\lesssim &\ m\left(d_x^{-1}\frac{|v_1-v_2|}{|v_1|}+|v_1-v_2|\right)^\beta e^{-\nu(v_1)t_b(x, v_1)}+|f|_{\infty, -}d_x^{-1}|v_1-v_2|,
\end{align*}
where $v_1(t)=v_1+t(\bar v_2-v_1)$ for some $t\in [0,1]$.
Using \eqref{tb1}, \eqref{tb3}, \eqref{tb5}, \eqref{fb} and the mean value theorem, we have
\begin{align*}
\ &\left| F_1(f)(x, v_2)-F_1(f)(x, \bar v_2)\right|\\
\leq &\ \left|f(x_b(x,v_2), v_2)-f(x_b(x, \bar v_2), \bar v_2)\right|e^{-\nu(v_2)t_b(x, v_2)}+\left|f(x_b(x,\bar v_2), \bar v_2)\right|\\
\ &\ \times \left(\left|e^{-\nu(v_2)t_b(x,v_2)}-e^{-\nu(v_2)t_b(x,\bar v_2)}\right|+\left|e^{-\nu(v_2)t_b(x, \bar v_2)}-e^{-\nu(\bar v_2)t_b(x,\bar v_2)}\right| \right)\\
\leq &\ m \left(|x_b(x, v_2)-x_b(x, \bar v_2)|+|v_2-\bar v_2|\right)^\beta+|f|_{\infty, -}\left| \int_{t_b(x, \bar v_2)}^{t_b(x, v_2)}-\nu(v_2)e^{-\nu(v_2)s}ds\right|\\
\ &\ +|f|_{\infty, -}|v_2-\bar v_2| t_b(x, \bar v_2)  |\nabla \nu(v_2(t))|  e^{-\nu(v_2(t))t_b(x, \bar v_2)}  \\
\lesssim &\ m |v_1-v_2|^\beta+|f|_{\infty, -} d_x^{-1}|v_1-v_2|,
\end{align*}
where $v_2(t)=v_2+t(\bar v_2- v_2)$ for some $t\in [0,1]$.
Thus,
\begin{align*}
\ & \left| F_1(f)(x, v_1)-F_1(f)(x, v_2) \right| \\
\lesssim &\ m \left(d_x^{-1}|v_1-v_2| (1+|v_1|^{-1}e^{-\frac{\nu(v_1)}{|v_1|}d_x}) \right)^\beta+|f|_{\infty, -}d_x^{-1}|v_1-v_2| \\
\lesssim &\ \left(m+|f|_{\infty, -}\right) \left(d_x^{-1}|v_1-v_2| (1+|v_1|^{-1}e^{-\frac{\nu(v_1)}{|v_1|}d_x}) \right)^{\min\{\alpha, \beta\}}.
\end{align*}
This completes the proof of the theorem.
\end{proof}


We give the H\"older regularity of $F_2(f)$ in the following lemma. Note that
\begin{align*}
F_2(f)(x,v) 
=\int_0^{t_b(x,v)}e^{-\nu(v)s}\int_{\mathbb{R}^3}k(v, \eta)F_1(f)(x-vs, \eta)d\eta ds.
\end{align*}

\begin{lemma}\label{LF2}
Let $-3<\gamma<0$ and $0<\beta\leq 1$. Let $0<\alpha<\min\{1, 3+\gamma\}$. Suppose that $f\in L^\infty(\Gamma_-)$ satisfies the regularity condition \eqref{fb}.
Then, for every $x, x_1, x_2\in\Omega$ and for every $v, v_1, v_2\in\mathbb{R}^3\setminus\{0\}$,
\begin{align}
\ & \left|F_2(f)(x_1, v)-F_2(f)(x_2, v)\right| \label{F2x}\\
\lesssim &\ (m+|f|_{\infty, -}) \left(d_{(x_1, x_2)}^{-1} |x_1-x_2|(1+|v|^{-1} e^{-\frac{\nu(v)}{4|v|}d_{(x_1, x_2)}})\right)^{\min\{\alpha, \beta\}}; \notag\\
\ & \left|F_2(f)(x, v_1)-F_2(f)(x, v_2)\right| \lesssim (m+|f|_{\infty, -}) \label{F2v}\\
\lesssim &\ (m+|f|_{\infty, -}) \left( d_x^{-1} |v_1-v_2|(1+| v_{1,2}|^{-1}e^{-c\frac{\nu(v_{1,2})}{|v_{1,2}|}d_x})\right)^{\min\{\alpha, \beta\}}. \notag
\end{align}
\end{lemma}

\begin{proof}
By \eqref{A-1} and \eqref{tb1}, we have
\begin{equation}\label{F2B}
|F_2(f)(x, v)| \leq C_{\gamma, \Omega} |f|_{\infty, -}\langle v\rangle^{\gamma-3}.
\end{equation}
For \eqref{F2x}, from \eqref{F2B}, it is obvious for $d_{(x_1, x_2)}^{-1} |x_1-x_2|(1+|v|^{-1} e^{-\frac{\nu(v)}{4|v|}d_{(x_1, x_2)}})\geq 1$.
Therefore, we consider the case where $d_{(x_1, x_2)}^{-1} |x_1-x_2|(1+|v|^{-1} e^{-\frac{\nu(v)}{4|v|}d_{(x_1, x_2)}})<1$.
We divide $\Omega$ as follows: let $n\in\mathbb{N}$,
\begin{align*}
\Omega_0 &=\left\{x\in\Omega : d(x, \partial\Omega)\leq 2^{-1} d_{(x_1, x_2)}\right\};\\
\Omega_n &=\left\{x\in\Omega : d(x, \partial\Omega)\leq 2^{-n} d_{(x_1, x_2)} \min\{1, |v|\}\right\}.
\end{align*}
Denote $y_n$ the intersection of $\overline{x_1x_b(x_1, v)}$ with $\partial \Omega_n \backslash \partial\Omega$ and denote $z_n$ the intersection of $\overline{x_2x_b(x_2, v)}$ with $\partial \Omega_n \backslash \partial\Omega$.
From Proposition \ref{P3.2}, we have
\begin{align}
|\overline{y_nx_b(x_1, v)}|,  |\overline{z_nx_b(x_2, v)}| &\leq R2^{-n}\min\{1, |v|\}, n\geq 1;\label{321}\\
|\overline{y_0x_b(x_1, v)}|,  |\overline{z_0x_b(x_2, v)}|& \leq R2^{-n}.\notag 
\end{align}
Note that one has $|\overline{x_1y_n}|, |\overline{x_2z_n}| \geq 2^{-1}d_{(x_1, x_2)}$ for $n=0, 1,\ldots$.
Let $N_1\in\mathbb{N}$ satisfy
\begin{equation}\label{N1}
2^{-N_1}\leq \left(d_{(x_1, x_2)}^{-1} |x_1-x_2| (1+|v|^{-1}) \right)^{\min\{\alpha, \beta\}} 
\end{equation}
For simplicity of notations, we denote $A_n=\min\left\{\frac{|\overline{x_1y_n}|}{|v|}, \frac{|\overline{x_2z_n}|}{|v|}\right\}$ and $a=\min\{\alpha, \beta\}$.
Without loss of generality, we assume $t_b(x_1, v)\geq t_b(x_2, v)$, then
\begin{align*}
\ &\left| F_2(f)(x_1, v)-F_2(f)(x_2, v)\right| \\
\leq& \sum_{n=1}^{N_1-1}\int_{A_n}^{A_{n+1}}e^{-\nu(v)s}\int_{\mathbb{R}^3}|k(v,\eta)|\left| F_1(f)(x_1-vs, \eta)-F_1(f)(x_2-vs, \eta)\right| d\eta ds\\
\ &\ +\int_{A_0}^{A_1}e^{-\nu(v)s}\int_{\mathbb{R}^3}|k(v,\eta)|\left| F_1(f)(x_1-vs, \eta)-F_1(f)(x_2-vs, \eta)\right| d\eta ds\\
\ &\ +\int_{0}^{A_0}e^{-\nu(v)s}\int_{\mathbb{R}^3}|k(v,\eta)|\left| F_1(f)(x_1-vs, \eta)-F_1(f)(x_2-vs, \eta)\right| d\eta ds\\
\ &\ +\int_{A_{N_1}}^{t_b(x_2, v)}e^{-\nu(v)s}\int_{\mathbb{R}^3}|k(v,\eta)|\left| F_1(f)(x_1-vs, \eta)-F_1(f)(x_2-vs, \eta)\right| d\eta ds\\
\ &\ + \int_{t_b(x_2, v)}^{t_b(x_1, v)}e^{-\nu(v)s}\int_{\mathbb{R}^3}\left| k(v,\eta)F_1(f)(x_1-vs, \eta) \right| d\eta ds \\
=:&\ I_1+I_2+I_3+I_4+I_5.
\end{align*}
It is easy to see that
\begin{align}\label{S1}
d_{(x_1-vs, x_2-vs)}^{-1} \leq 2^{n+1}d_{(x_1, x_2)}^{-1}(1+|v|^{-1})
\end{align}
for $s\in \left[A_n, A_{n+1}\right]$, $n=0, 1,\ldots.$
Moreover, by \eqref{321}, one has
\begin{align}\label{323}
A_{n+1}-A_n \leq \max\{|\overline{y_nx_b(x_1, v)}|, |\overline{z_nx_b(x_2, v)}|\}|v|^{-1}\leq R2^{-n}, n\geq 1.
\end{align}

Let's estimate $I_1$.
By using the mean value theorem, \eqref{A-1}, \eqref{F1x}, \eqref{S1}, and \eqref{323}, it follows from $a<1$ that
\begin{align*}
I_1 & \lesssim (m+|f|_{\infty, -})|x_1-x_2|^a \\
\ &\qquad \times \sum_{n=1}^{N_1-1} \int_{A_n}^{A_{n+1}}e^{-\nu(v)s} \int_{\mathbb{R}^3}|k(v,\eta)| \left[d_{(x_1-vs, x_2-vs)}^{-1}(1+|\eta|^{-1})\right]^{a} d\eta ds\\
\ &\lesssim (m+|f|_{\infty, -}) \left( d_{(x_1, x_2)}^{-1}|x_1-x_2|(1+|v|^{-1})\right)^{a} \sum_{n=1}^{N_1-1} 2^{a n} \left(A_{n+1}-A_n\right)e^{-\nu(v)A_n} \\
\ &\lesssim (m+|f|_{\infty, -}) \left(d_{(x_1, x_2)}^{-1}|x_1-x_2|(1+|v|^{-1})\right)^{a}  \\
\ &\qquad \times \left(\sum_{n=1}^{N_1-1} 2^{(a-1) n}\right)e^{-\frac{\nu(v)}{|v|}\min\{|\overline{x_1y_n}|, |\overline{x_2z_n}|\}}\\
\ &\lesssim (m+|f|_{\infty, -}) \left( d_{(x_1, x_2)}^{-1}|x_1-x_2|(1+|v|^{-1})\right)^{a}e^{-\frac{\nu(v)}{2|v|}d_{(x_1, x_2)}}.
\end{align*}
For $I_2$ and $I_3$, by applying \eqref{A-1}, \eqref{tb1}, \eqref{F1x}, and \eqref{S1}, we have
\begin{align*}
I_2 
&\lesssim (m+|f|_{\infty, -})\left(d_{(x_1, x_2)}^{-1}|x_1-x_2|(1+|v|^{-1}\right)^a e^{-\frac12\nu(v)A_0} \int_{A_0}^{A_{1}}e^{-\frac12\nu(v)s} ds\\
\ &\lesssim (m+|f|_{\infty, -})\left(d_{(x_1, x_2)}^{-1}|x_1-x_2|(1+|v|^{-1}\right)^a e^{-\frac{\nu(v)}{4|v|}d_{(x_1, x_2)}},
\end{align*}
and
\begin{align*}
I_3 &\ \lesssim (m+|f|_{\infty, -})|x_1-x_2|^{a} \int_{0}^{A_{0}}e^{-\nu(v)s} \langle v\rangle^{\gamma-2} \left[ 2^{-1}d_{(x_1, x_2)}\}\right]^{-a} ds\\
\ &\lesssim (m+|f|_{\infty, -})\left(d_{(x_1, x_2)}^{-1}|x_1-x_2|\right)^a.
\end{align*}
For $I_4$, by using \eqref{321} and \eqref{N1}, we have
\begin{align*}
I_{4} &\lesssim  |f|_{\infty, -}\int_{A_{N_1}}^{t_b(x_2,v)}\langle v\rangle^{\gamma-2} e^{-\nu(v)s} ds\\
\ &\lesssim |f|_{\infty, -}\left(\frac{|\overline{x_2x_b(x_2,v)}|}{|v|}-\min\left\{\frac{|\overline{x_1y_{N_1}}|}{|v|}, \frac{|\overline{x_2z_{N_1}}|}{|v|}\right\}\right)\langle v\rangle^{\gamma-2}e^{-\nu(v)A_{N_1}}\\
\ &\lesssim |f|_{\infty, -}\max\{|\overline{y_{N_1}x_b(x_1, v)}|, |\overline{z_{N_1}x_b(x_2, v)}|\}\frac{\langle v\rangle^{\gamma-2}}{|v|}e^{-\frac{\nu(v)}{2|v|}d_{(x_1, x_2)}} \\
\ &\lesssim |f|_{\infty, -}2^{-N_1}R \min\{1, |v|\}\frac{\langle v\rangle^{\gamma-2}}{|v|}e^{-\frac{\nu(v)}{2|v|}d_{(x_1, x_2)}} \\
\ &\lesssim  |f|_{\infty, -}\left( d_{(x_1, x_2)}^{-1}|x_1-x_2|(1+|v|^{-1}) \right)^a e^{-\frac{\nu(v)}{2|v|}d_{(x_1, x_2)}}.
\end{align*}
For $I_5$, by \eqref{tb2}, we have
\begin{align*}
I_5 &\lesssim |f|_{\infty, -}d_{(x_1, x_2)}^{-1}|x_1-x_2| |v|^{-1} e^{-\frac{\nu(v)}{|v|}d_{(x_1, x_2)}}. 
\end{align*}
Summing up, one get \eqref{F2x}.


For \eqref{F2v}, we only do the case of $|v_1|\geq |v_2|$.
Since \eqref{F2B}, \eqref{F2v} is obvious for $d_x^{-1}|v_1-v_2|(1+e^{-c\frac{\nu(v_1)}{|v_1|}d_x})\geq 1$. Therefore, we consider the case for $d_x^{-1}|v_1-v_2|(1+e^{-c\frac{\nu(v_1)}{|v_1|}d_x}) < 1$.
We now divide $\Omega$ as follows: let $n\in\mathbb{N}$,
\begin{align*}
\Omega_0' &=\left\{y\in\Omega : d(y, \partial\Omega)\leq 2^{-1}d_{x}\right\},\\
\Omega_n' &=\left\{y\in\Omega : d(y, \partial\Omega)\leq 2^{-n}d_{x}\min\{1, |v_1|\}\right\}.
\end{align*}
Denote $y_n'$ the intersection of $\overline{xx_b(x, v_1)}$ with $\partial \Omega_n' \backslash \partial\Omega$ and denote $z_n'$ the intersection of $\overline{xx_b(x, v_2)}$ with $\partial \Omega_n' \backslash \partial\Omega$.
From Proposition \ref{P3.2}, we have
\begin{align}
\max\left\{|\overline{y_n'x_b(x, v_1)}|,  |\overline{z_n'x_b(x, v_2)}|  \right\} &\leq R 2^{-n}\min\{1, |v_1|\}, n\geq1; \label{v1} \\
\max\left\{|\overline{y_0'x_b(x, v_1)}|,  |\overline{z_0'x_b(x, v_2)}|  \right\} &\leq R 2^{-1}. \notag 
\end{align}
Let $N_2=N_2(x, v_1, v_2)\in\mathbb{N}$ satisfy
\begin{equation}\label{N2}
2^{-N_2}\leq \left( d_{x}^{-1}|v_1-v_2|(1+|v_1|^{-1})\right)^a. 
\end{equation}
Let $\bar v_2=\frac{|v_1|}{|v_2|}v_2$. Note that $|v_1-\bar v_2|\leq 2|v_1-v_2|$ and $|v_2- \bar v_2|\leq |v_1-v_2|$.
Then
\begin{align*}
\left|F_2(x,v_1)-F_2(x, v_2)\right| & \leq \left|F_2(x,v_1)-F_2(x, \bar v_2)\right| + \left|F_2(x,\bar v_2)-F_2(x, v_2)\right|\\
\ &=:II+III.
\end{align*}

For $II$, without loss of generality, we assume that $t_b(x, v_1)\geq t_b(x, \bar v_2)$. Then
\begin{align*}
II &\leq \int_0^{t_b(x, \bar v_2)}e^{-\nu(\bar v_2)s}\int_{\mathbb{R}^3}|k(\bar v_2, \eta)| \left| F_1(f)(x-v_1s, \eta)-F_1(f)(x-\bar v_2s, \eta)\right| d\eta ds\\
\ &\ + \int_0^{t_b(x, \bar v_2)}e^{-\nu(\bar v_2)s}\int_{\mathbb{R}^3}|k(v_1, \eta)-k(\bar v_2, \eta)| \left| F_1(f)(x-v_1s, \eta)\right| d\eta ds\\
\ &\ + \int_{0}^{t_b(x, \bar v_2)}\left|e^{-\nu(v_1)s}-e^{-\nu(\bar v_2)s}\right|\int_{\mathbb{R}^3}|k(v_1, \eta)|\left| F_1(f)(x-v_1s, \eta)\right| d\eta ds\\
\ &\ + \int_{t_b(x, \bar v_2)}^{t_b(x, v_1)}e^{-\nu(v_1)s}\int_{\mathbb{R}^3}|k(v_1, \eta)|\left| F_1(f)(x-v_1s, \eta)\right| d\eta ds\\
\ &=:II_1+II_2+II_3+II_4.
\end{align*}
For $II_1$, let $B_n=\min\left\{\frac{|\overline{xy_{n}'}|}{|v_1|}, \frac{|\overline{xz_{n}'}|}{|\bar v_2|} \right\}$, then
\begin{align*}
II_1&= \sum_{n=1}^{N_2-1}\int_{B_n}^{B_{n+1}}\cdots\cdots +\int_{B_0}^{B_1}\cdots\cdots+\int_{0}^{B_0}\cdots\cdots +\int_{B_{N_2}}^{t_b(x, v_2)}\cdots\cdots \\
\ &=: II_{11}+II_{12}+II_{13}+II_{14}.
\end{align*}
It is easy to see that
\begin{align}\label{S2}
d_{(x-v_1s, x-\bar v_2s)}^{-1} \leq 2^{n+1}d_{x}^{-1}(1+|v_1|^{-1})
\end{align}
for $s\in \left[B_n, B_{n+1}\right]$, $n=0, 1,\ldots.$
Moreover, by \eqref{v1}, one has
\begin{align}\label{324}
B_{n+1}-B_n \leq \max\{|\overline{y_n'x_b(x, v_1)}|, |\overline{z_n'x_b(x, v_2)}|\}|\bar v_2|^{-1}\leq R2^{-n}, n\geq 1.
\end{align}
For $II_{11}$, by using the mean value theorem, \eqref{A-1}, \eqref{F1x}, \eqref{S2}, and \eqref{324}, we have
\begin{align*}
II_{11} & \lesssim (m+|f|_{\infty, -})|v_1-\bar v_2|^a \\
\ &\qquad \times \sum_{n=1}^{N_2-1}\int_{B_n}^{B_{n+1}}s^ae^{-\nu(\bar v_2)s} \int_{\mathbb{R}^3}|k(\bar v_2,\eta)| \left[d_{(x-v_1s, x-\bar v_2s)}^{-1}(1+|\eta|^{-1})\right]^a d\eta ds\\
\ &\lesssim (m+|f|_{\infty, -})\left(d_{x}^{-1}|v_1-v_2| (1+|v_1|^{-1})\right)^{a}\langle v_1\rangle^{\gamma-2} \\
\ &\qquad \times \sum_{n=1}^{N_2-1} 2^{a n}\left(B_{n+1}-B_n\right)\left(\frac{s_n}{|\bar v_2|}\right)^a e^{-\frac{\nu(\bar v_2)}{|\bar v_2|}s_n}\\
\ & \lesssim (m+|f|_{\infty, -})\left(d_{x}^{-1}|v_1-v_2| (1+|v_1|^{-1})\right)^{a} e^{-c\frac{\nu(v_1)}{|v_1|}d_x} \left(\sum_{n=1}^{N_2-1} 2^{(a-1) n}\right) \\
\ &\lesssim (m+|f|_{\infty, -})\left(d_{x}^{-1}|v_1-v_2| (1+|v_1|^{-1})\right)^{a} e^{-c\frac{\nu(v_1)}{|v_1|}d_x} ,
\end{align*}
where $\min\left\{|\overline{xy_{n}'}|, |\overline{xz_{n}'}| \right\}\leq s_n\leq \min \{|\overline{xy_{n+1}'}|, |\overline{xz_{n+1}'}| \}$.

For $II_{12}$ and $II_{13}$, by using \eqref{A-1}, \eqref{tb1}, \eqref{F1x}, and \eqref{S2}, we have
\begin{align*}
\ &\ II_{12} \\
\lesssim &\ (m+|f|_{\infty, -})\left(d_x^{-1}|v_1-v_2|(1+|v_1|^{-1})\right)^a e^{-\frac12\nu(\bar v_2)B_0}  \int_{B_0}^{B_1}s^a e^{-\frac12\nu(\bar v_2)s} ds\\
\lesssim &\ (m+|f|_{\infty, -})\left(d_x^{-1}|v_1-v_2|(1+|v_1|^{-1})\right)^a e^{-c\frac{\nu(v_1)}{|v_1|}d_x}
\end{align*}
and
\begin{align*}
II_{13} &\lesssim (m+|f|_{\infty, -})|v_1- v_2|^a  \int_{0}^{B_0}s^a e^{-\nu(\bar v_2)s} \left(2^{-1}d_x \right)^{-a} \langle \bar v_2 \rangle^{\gamma-2}ds\\
\ &\lesssim (m+|f|_{\infty, -})\left(d_x^{-1}|v_1-v_2| \right)^a.
\end{align*}
For $II_{14}$, by using the mean value theorem, \eqref{v1}, and \eqref{N2}, we have
\begin{align*}
II_{14}&\lesssim |f|_{\infty, -}\left(t_b(x, \bar v_2)-B_{N_2}\right)\langle v_1\rangle^{\gamma-2} e^{-\nu(\bar v_2)B_{N_2}} \\
\ &\lesssim |f|_{\infty, -}\max\{|\overline{y_{N_2}'x_b(x, v_1)}|, |\overline{z_{N_2}'x_b(x, v_2)}| \} |v_1|^{-1} \langle v_1\rangle^{\gamma-2} e^{-\frac{\nu(\bar v_2)}{2|\bar v_2|}d_x} \\
\ &\lesssim |f|_{\infty, -}R 2^{-N_2}\min\{1, |v_1|\}|v_1|^{-1} \langle v_1\rangle^{\gamma-2} e^{-\frac{\nu(\bar v_2)}{2|\bar v_2|}d_{x}} \\
\ &
\lesssim |f|_{\infty, -}\left( d_x^{-1}|v_1-v_2|(1+|v_1|^{-1}) \right)^a e^{-c\frac{\nu(v_1)}{|v_1|}d_{x}}.
\end{align*}
By using Lemma \ref{kH}, \eqref{A-1}, \eqref{tb1}, \eqref{tb3}, and \eqref{tb4}, we have
\begin{align*}
\ &\ II_2+II_3+II_4 \\
\lesssim &\ |f|_{\infty, -}\left(|v_1-v_2|^{\widetilde\alpha_\gamma}+|v_1-v_2|+d_x^{-1}|v_1-v_2| \right).
\end{align*}
Summing up, one get
\begin{align*}
II &  \lesssim (m+|f|_{\infty, -}) \left( d_x^{-1}|v_1-v_2|(1+|v_1|^{-1} e^{-c\frac{\nu(v_1)}{|v_1|}d_x})\right)^a.
\end{align*}

For $III$, one has that $t_b(x, v_2)\geq t_b(x, \bar v_2)$. Thus,
\begin{align*}
III &\leq \int_0^{t_b(x, \bar v_2)}e^{-\nu(\bar v_2)s}\int_{\mathbb{R}^3}|k(\bar v_2, \eta)| \left| F_1(f)(x-v_2s, \eta)-F_1(f)(x-\bar v_2s, \eta)\right| d\eta ds\\
\ &\ + \int_0^{t_b(x, \bar v_2)}e^{-\nu(\bar v_2)s}\int_{\mathbb{R}^3}|k(v_2, \eta)-k(\bar v_2, \eta)| \left| F_1(f)(x-v_2s, \eta)\right| d\eta ds\\
\ &\ + \int_{0}^{t_b(x, \bar v_2)}\left|e^{-\nu(v_2)s}-e^{-\nu(\bar v_2)s}\right|\int_{\mathbb{R}^3}|k(v_2, \eta)|\left| F_1(f)(x-v_2s, \eta)\right| d\eta ds\\
\ &\ + \int_{t_b(x, \bar v_2)}^{t_b(x, v_2)}e^{-\nu(v_2)s}\int_{\mathbb{R}^3}|k(\bar v_2, \eta)|\left| F_1(f)(x-v_2s, \eta)\right| d\eta ds\\
\ &=:III_1+III_2+III_3+III_4.
\end{align*}
For $III_1$, let $\widetilde B_n=\frac{|\overline{xz_{n}'}|}{|\bar v_2|}$, then
\begin{align*}
III_1& = \sum_{n=1}^{N_2-1}\int_{\widetilde B_n}^{\widetilde B_{n+1}}\cdots\cdots  +\int_{\widetilde B_0}^{\widetilde B_1}\cdots\cdots +\int_0^{\widetilde B_0}\cdots\cdots+\int_{\widetilde B_{N_2}}^{t_b(x, \bar v_2)}\cdots\cdots \\
\ & =: III_{11}+III_{12}+III_{13}+III_{14}.
\end{align*}
For $III_{11}$, by using the mean value theorem, \eqref{A-1}, \eqref{F1x}, \eqref{S2}, and \eqref{324}, we have
\begin{align*}
\ & III_{11}\\
\lesssim &\ (m+|f|_{\infty, -})\left(d_x^{-1}|v_1-v_2|(1+|v_1|^{-1})\right)^a \sum_{n=1}^{N_2-1} 2^{an} \int_{\widetilde B_n}^{\widetilde B_{n+1}}s^a e^{-\nu(\bar v_2)s}   ds \\
\lesssim &\ (m+|f|_{\infty, -})\left(d_x^{-1} |v_1-v_2|(1+|v_1|^{-1}) \right)^{a} \sum_{n=1}^{N_2-1} 2^{(a-1) n}\left(\frac{s_n}{|\bar v_2|}\right)^a e^{-\frac{\nu(\bar v_2)}{|\bar v_2|}s_n}\\
\lesssim &\ (m+|f|_{\infty, -})\left(d_x^{-1} |v_1-v_2|(1+|v_1|^{-1}) \right)^{a}e^{-c\frac{\nu(v_1)}{|v_1|} d_x}\left( \sum_{n=1}^{N_2-1} 2^{(a-1) n}\right) \\
\lesssim &\ (m+|f|_{\infty, -})\left(d_x^{-1} |v_1-v_2|(1+|v_1|^{-1}) \right)^{a}e^{-c\frac{\nu(v_1)}{|v_1|} d_x}, 
\end{align*}
where $|\overline{xy_n'}| \leq s_n \leq  |\overline{xy_{n+1}'}| $.
For $III_{12}$ and $III_{13}$, by using \eqref{A-1}, \eqref{tb1}, \eqref{F1x}, and \eqref{S2}, we have
\begin{align*}
\ &\ III_{12} \\
\lesssim &\ (m+|f|_{\infty, -})\left(d_x^{-1}|v_1-v_2|(1+|v_1|^{-1})\right)^a e^{-\frac12\nu(\bar v_2)\widetilde B_0}  \int_{\widetilde B_0}^{\widetilde B_1}s^a e^{-\frac12\nu(\bar v_2)s} ds\\
\lesssim &\ (m+|f|_{\infty, -})\left(d_x^{-1}|v_1-v_2|(1+|v_1|^{-1})\right)^a e^{-c\frac{\nu(v_1)}{|v_1|}d_x} 
\end{align*}
and
\begin{align*}
III_{13} &\lesssim (m+|f|_{\infty, -})|v_1- v_2|^a  \int_{0}^{\widetilde B_0}s^a e^{-\nu(\bar v_1)s} \left(2^{-1}d_x \right)^{-a} \langle \bar v_2 \rangle^{\gamma-2}ds\\
\ &\lesssim (m+|f|_{\infty, -})\left(d_x^{-1}|v_1-v_2| \right)^a. 
\end{align*}
For $III_{14}$, by using the mean value theorem, \eqref{v1}, and \eqref{N2}, we have
\begin{align*}
III_{14} 
&\lesssim |f|_{\infty, -}\langle v_1 \rangle^{\gamma-2} |\overline{y_{N_2}'x_b(x, v_2)}| |\bar v_2|^{-1} e^{-\frac{\nu(\bar v_2)}{|\bar v_2|}|\overline{xy_{N_2}'}|} \\
\ &\lesssim |f|_{\infty, -} \langle v_1 \rangle^{\gamma-2}R 2^{-N_2}\min\{1, |v_1|\}|v_1|^{-1} e^{-c\frac{\nu(v_1)}{|v_1|}d_{x}} \\
\ &\lesssim  |f|_{\infty, -}\left( d_{x}^{-1}|v_1-v_2|(1+|v_1|^{-1})\right)^a e^{-c\frac{\nu(v_1)}{|v_1|}d_{x}} \langle v_1 \rangle^{\gamma-2}.
\end{align*}
By using Lemma \ref{kH}, \eqref{A-1}, \eqref{tb1}, \eqref{tb3}, and \eqref{tb5}, we have
\begin{align*}
III_2+III_3+III_4 \lesssim |f|_{\infty, -} \left(|v_1-v_2|^{\widetilde\alpha_\gamma}+|v_1-v_2|+d_x^{-1}|v_1-v_2|\right).
\end{align*}
Summing up, one get
\begin{align*}
III & \lesssim (m+|f|_{\infty, -})\left( d_x^{-1}|v_1-v_2|(1+|v_1|^{-1} e^{-c\frac{\nu(v_1)}{|v_1|}d_x}) \right)^a.
\end{align*}
This completes the proof.
\end{proof}

From the results of Proposition \ref{LEU}, Lemmas \ref{LF3}--\ref{LF2}, and \eqref{e2}, we complete the proof of Theorem \ref{T1}.


\section{H\"older regularity of the approximate steady Boltzmann equation}\label{app-holder}

In this section, we will study the H\"older regularity of solution of the following approximate steady Boltzmann equation:
\begin{equation}\label{1f1g}
\begin{cases}
v\cdot\nabla_x f+Lf+\varphi(g) f=N_{+}(g, g), & \text{ in } \Omega\times\mathbb{R}^3,\\
f=r, & \text{ on } \Gamma_-,
\end{cases}
\end{equation}
where $\Omega$ is a bounded and strictly convex domain and $\varphi(g)$ is defined in \eqref{Psi}.
According to the regularity estimate \eqref{T1v} obtained in Theorem \ref{T1}, we suppose that $g$ satisfies the following estimate:
\begin{align}\label{gxv}
\ &\ |g(x_1, v_1)-g(x_2, v_2)|  \\
\leq &\ C_g\left(d_{(x_1, x_2)}^{-1} (|x_1-x_2|+|v_1-v_2|)(1+|v_{1,2}|^{-1}e^{-c\frac{\nu(v_{1,2})}{|v_{1,2}|}d_{(x_1, x_2)}})\right)^a, \notag
\end{align}
for every $x_1, x_2\in \Omega$ and $v_1, v_2\in\mathbb{R}^3\setminus\{0\}$,
where $0<a<\min\{1, 3+\gamma\}$ and $C_g>0$ depends on $g$.
Moreover, from Lemma \ref{Ne1}, we have
\begin{equation*}
\left| \varphi(g)(x,v) \right| \leq \widehat C\| w g\|_\infty \nu(v).
\end{equation*}
Assume that $g$ satisfies
\begin{equation}\label{g2}
\widehat C\| w g\|_\infty\leq \frac 12
\end{equation}
for $\tau>0$ and $0<\ka<\frac 18$.

The H\"older regularity estimate of problem (\ref{1f1g}) is stated as follows:
\begin{theorem}\label{T2}
Assume that $r\in L^\infty(\Gamma_-)$ satisfies \eqref{rb} and $g$ satisfies \eqref{gxv} and \eqref{g2} for $\tau>3+|\gamma|$ and $0<\kappa< \frac18$.
Let $f\in L^\infty(\Omega\times\mathbb{R}^3)$ be a solution of \eqref{1f1g}.
Then, there exists $C=C_{a, \beta, \gamma, \Omega}>0$ such that  for every $x_1, x_2\in\Omega$ and $v_1, v_2\in\mathbb{R}^3\setminus\{0\}$,
\begin{align*}
 |f(x_1, v_1)&-f(x_2, v_2)| \leq C(m+C_{f, g}+\widetilde C_{f,g}+\widetilde C_g ) \notag\\
\ & \times  \left(d_{(x_1, x_2)}^{-1}(|x_1-x_2|+|v_1-v_2|)(1+|v_{1,2}|^{-1}e^{-c\frac{\nu(v_{1,2})}{|v_{1,2}|}d_{(x_1, x_2)}})\right)^{\min\{a, \beta\}}, \notag
\end{align*}
where
\begin{align*}
C_{f,g}&=|f|_{\infty, -}\left(1+C_g+\| wg\|_\infty\right); \\
\widetilde C_{f,g} &=\|f\|_{\infty}\left(1+C_g+\| wg\|_\infty\right);\\
\widetilde C_g &= (C_g+C_g\| wg\|_\infty+\| wg\|_\infty+\| wg\|_\infty^2)\| wg\|_\infty.\\
\end{align*}
\end{theorem}

To prove the result, from the method of characteristics, for every $x\in\Omega$ and $v \in \mathbb{R}^3\setminus\{0\}$, one has
\begin{align*}
\ &\ f(x,v)\\
=&\ f\left(x_b(x,v), v\right)e^{-(\nu(v)t_b(x,v)+\int_0^{t_b(x,v)}\varphi(g)(x-v\ell, v)d\ell)}\\
\ &+\int_0^{t_b(x,v)}   e^{-(\nu(v)s+\int_0^s\varphi(g)(x-v\ell, v)d\ell)} \left(Kf+N_+(g,g)\right)(x-vs, v) dv ds.
\end{align*}
For simplicity of notation, we denote
\begin{equation*}
\Phi(g)(x, v, s)= \nu(v)s+\int_0^s \varphi(g)(x-v\ell, v)d\ell.
\end{equation*}
We iterate the term $Kf$ once more to get
\begin{align*}
\ & f(x, v) \\
= &\ M_1(f, g)(x,v)+M_2(f, g)(x,v)+M_3(f, g)(x,v)+N_1(g)(x, v)+N_2(g)(x, v),
\end{align*}
where
\begin{align*}
M_1(f, g)(x,v) &:= f\left(x_b(x,v), v\right)e^{-\Phi(g)(x, v, t_b(x, v) )};\\
M_2(f, g)(x,v) &:= \int_0^{t_b(x,v)} e^{-\Phi(g)(x, v, s)}\int_{\mathbb{R}^3}k(v, \eta) M_1(f, g)(x-vs, \eta) d\eta ds;\\
M_3(f, g)(x,v) &:= \int_0^{t_b(x,v)} e^{-\Phi(g)(x, v, s)}\int_{\mathbb{R}^3}k(v, \eta) \\
\ &\qquad \times \int_0^{t_b(x-vs, \eta)} e^{-\Phi(g)(x-vs, \eta, \tau)} Kf\left(x-vs-\eta\tau, \eta \right) d\tau d\eta ds;\\
N_1(g)(x, v)   &:= \int_0^{t_b(x,v)} e^{-\Phi(g)(x, v, s)}N_+(g ,g)(x-vs, v)ds;  \\
N_2(g)(x, v)   &:= \int_0^{t_b(x,v)} e^{-\Phi(g)(x, v, s)} \int_{\mathbb{R}^3}k(v, \eta) N_1(g)(x-vs, \eta) d\tau d\eta ds.
\end{align*}

The following lemma shows the H\"older regularity for the nonlinear operators.

\begin{lemma}
 Let $\tau>0$ and $0<\kappa< \frac 18$. Assume that $g$ satisfies \eqref{gxv}. Then for every $x, x_1, x_2\in\Omega$ and $v, v_1, v_2\in\mathbb{R}^3\setminus\{0\}$, we have
\begin{align}
\ &\ \left|N_+(g, g)(x_1, v)-N_+(g, g)(x_2, v)\right| \label{N+x} \\
\lesssim &\ C_g\| wg\|_\infty\left(d_{(x_1, x_2)}^{-1}|x_1-x_2| \right)^a \langle v\rangle^{\gamma-2} ; \notag \\
\ &\ \left|N_+(g, g)(x, v_1)-N_+(g, g)(x, v_2)\right|  \label{N+v}\\
\lesssim &\ \| wg\|_\infty \left(C_g \left[ d_{x}^{-1}|v_1-v_2| \right]^a +\| w g\|_\infty|v_1-v_2|\right)\max\{\langle v_1\rangle^{\gamma-2}, \langle v_2 \rangle^{\gamma-2} \}; \notag \\
\ & \left|\varphi(g)(x_1, v)-\varphi(g)(x_2, v)\right| \label{N-x}\\
\lesssim &\ C_g   \left(d_{(x_1, x_2)}^{-1}|x_1-x_2| \right)^a \langle v\rangle^\gamma; \notag \\
\ & \left|\varphi(g)(x, v_1)-\varphi(g)(x, v_2)\right| \label{N-v} \\
\lesssim &\ \| wg\|_\infty |v_1-v_2| \langle v_1\rangle^\gamma + C_g  \left( d_{x}^{-1}|v_1-v_2| \right)^a\max\{\langle v_1\rangle^\gamma, \langle v_2\rangle^\gamma\} . \notag
\end{align}
\end{lemma}

\begin{proof}
By applying Lemma \ref{NH-1}, \eqref{A-2}, and \eqref{gxv}, we get \eqref{N+x}--\eqref{N-v}.
\end{proof}

We need the following two lemmas, and their proofs are most same as that of Lemma \ref{LF2} and are omitted here for simplicity of presentation.

\begin{lemma}
 Let $\tau>0$ and $0<\kappa< \frac 18$. Assume that $g$ satisfies \eqref{gxv}. Let $x_1, x_2 \in \Omega$ and $v\in\mathbb{R}^3\setminus\{0\}$ with $t_b(x_1, v)\geq t_b(x_2, v)$. Then
\begin{align}
\ & \int_0^{t_b(x_2, v)}e^{-\frac12 \nu(v)s}\int_0^s \left| \varphi(g)(x_1-v\ell, v) - \varphi(g)(x_2-v\ell, v)\right| d\ell  ds  \label{N11}\\
\lesssim &\ (C_g+\| wg\|_\infty) \left(d_{(x_1, x_2)}^{-1}|x_1-x_2| (1+|v|^{-1} e^{-\frac{\nu(v)}{8|v|}d_{(x_1, x_2)}}) \right)^a;\notag\\
\ & \int_0^{t_b(x_2, v)}e^{-\frac12 \nu(v)s}\left| N_+(g)(x_1-vs, v) - N_+(g)(x_2-vs, v)\right| ds  \label{N12}\\
\lesssim &\ (C_g+\| wg\|_\infty)\| wg\|_\infty \left(d_{(x_1, x_2)}^{-1}|x_1-x_2| (1+|v|^{-1} e^{-\frac{\nu(v)}{8|v|}d_{(x_1, x_2)}})\right)^a.\notag
\end{align}
Let $x\in \Omega$ and $v_1, v_2\in\mathbb{R}^3$ with $0<|v_2|\leq |v_1|$ and $t_b(x, v_1)\geq t_b(x, \bar v_2)$, where $\bar v_2=\frac{|v_1|}{|v_2|}v_2$. Then
\begin{align}
\ & \int_0^{t_b(x, \bar v_2)}e^{-\frac12 \nu(\bar v_2)s}\int_0^s \left| \varphi(g)(x-v_1\ell, v_1) - \varphi(g)(x- \bar v_2\ell, \bar v_2)\right| d\ell  ds  \label{N13}\\
\lesssim &\ (C_g+\| wg\|_\infty) d_{x}^{-a}|v_1-v_2|^a (1+|v_1|^{-a} e^{-c\frac{\nu(v_1)}{|v_1|}d_x})+\| wg\|_\infty |v_1-v_2| ;\notag\\
\ & \int_0^{t_b(x, \bar v_2)}e^{-\frac12 \nu(\bar v_2)s}\left| N_+(g)(x-v_1s, v_1) - N_+(g)(x-\bar v_2s, \bar v_2)\right| ds  \label{N14}\\
\lesssim &\ (C_g+\| wg\|_\infty)\| wg\|_\infty d_{x}^{-a}|v_1-v_2|^a (1+|v_1|^{-a} e^{-c\frac{\nu(v_1)}{|v_1|}d_x})+\| wg\|_\infty^2 |v_1-v_2|;\notag\\
\ & \int_0^{t_b(x, \bar v_2)}e^{-\frac12 \nu(\bar v_2)s}\int_0^s \left| \varphi(g)(x-v_2\ell, v_2) - \varphi(g)(x- \bar v_2\ell, \bar v_2)\right| d\ell  ds  \label{N15}\\
\lesssim &\ (C_g+\| wg\|_\infty) d_{x}^{-a}|v_1-v_2|^a (1+|v_1|^{-a} e^{-c\frac{\nu(v_1)}{|v_1|}d_x})+\| wg\|_\infty |v_1-v_2|;\notag\\
\ & \int_0^{t_b(x, \bar v_2)}e^{-\frac12 \nu(\bar v_2)s}\left| N_+(g)(x-v_2s, v_2) - N_+(g)(x-\bar v_2s, \bar v_2)\right| ds  \label{N16}\\
\lesssim &\ (C_g+\| wg\|_\infty)\| wg\|_\infty d_{x}^{-a}|v_1-v_2|^a (1+|v_1|^{-a} e^{-c\frac{\nu(v_1)}{|v_1|}d_x})+\| wg\|_\infty^2|v_1-v_2|.\notag
\end{align}
\end{lemma}

\begin{lemma}
 Let $\tau>0$ and $0<\kappa< \frac 18$. Assume that $g$ satisfies \eqref{gxv}. Let $x_1, x_2 \in \Omega$ and $v\in\mathbb{R}^3\setminus\{0\}$ with $t_b(x_1, v)\geq t_b(x_2, v)$. Then
\begin{align}
\ & \left(\int_0^{t_b(x_2, v)} \left| \varphi(g)(x_1-v\ell, v) - \varphi(g)(x_2-v\ell, v)\right|  d\ell \right) e^{-\frac12 \nu(v)t_b(x_2, v)}  \label{M11}\\
\lesssim &\ (C_g+\| wg\|_\infty) \left( d_{(x_1, x_2)}^{-1}|x_1-x_2| (1+|v|^{-1})\right)^a e^{-\frac{\nu(v)}{4|v|}d_{(x_1, x_2)}}). \notag
\end{align}
Let $x\in \Omega$ and $v_1, v_2\in\mathbb{R}^3$ with $0<|v_2|\leq |v_1|$ and $t_b(x, v_1)\geq t_b(x, \bar v_2)$, where $\bar v_2=\frac{|v_1|}{|v_2|}v_2$. Then
\begin{align}
\ & \left(\int_0^{t_b(x, \bar v_2)}\left| \varphi(g)(x-v_1\ell, v_1) - \varphi(g)(x- \bar v_2\ell, \bar v_2)\right| d\ell\right) e^{-\frac12 \nu(v_1)t_b(x, \bar v_2)} \label{M12}\\
\lesssim &\ C_g \left( d_{x}^{-1}|v_1-v_2|(1+|v_1|^{-1}) \right)^a e^{-c\frac{\nu(v_1)}{|v_1|}d_x}+\| wg\|_\infty |v_1-v_2| ;\notag\\
\ & \left(\int_0^{t_b(x, \bar v_2)} \left| \varphi(g)(x-v_2\ell, v_2) - \varphi(g)(x- \bar v_2\ell, \bar v_2)\right| d\ell \right) e^{-\frac12 \nu(\bar v_2)t_b(x, \bar v_2)}  \label{M13}\\
\lesssim &\ C_g \left(d_{x}^{-1}|v_1-v_2| (1+|v_1|^{-1})\right)^a e^{-c\frac{\nu(v_1)}{|v_1|}d_x}+\| wg\|_\infty |v_1-v_2|.\notag
\end{align}
\end{lemma}



In what follows we study the H\"older regularity for $N_1(g), N_2(g)$, $M_1(f, g), M_2(f, g)$, and $M_3(f, g)$. We begin by showing the H\"older regularity of $N_1(g)$.

\begin{lemma}\label{LN1}
Assume that $g$ satisfies \eqref{gxv} and \eqref{g2}.
Then for every $x, x_1, x_2\in\Omega$ and $v, v_1, v_2\in\mathbb{R}^3\setminus \{0\}$,
\begin{align}
\ &\ \left|N_1(g)(x_1, v)-N_1(g)(x_2, v)\right| \label{N1x}\\
\lesssim &\ \widetilde C_g  \left(d_{(x_1, x_2)}^{-1} |x_1-x_2|(1+|v|^{-1}e^{-\frac{\nu(v)}{8|v|}d_{(x_1, x_2)}})\right)^{a};\notag \\ 
\ &\ \left|N_1(g)(x, v_1)-N_1(g)(x, v_2)\right| \label{N1v}\\
\lesssim &\ \widetilde C_g \left(d_x^{-1}|v_1-v_2|(1+|v_{1,2}|^{-1}e^{-c\frac{\nu(v_{1,2})}{|v_{1,2}|}d_x})\right)^{a}. \notag 
\end{align}
where $\widetilde C_g$ is as in Theorem \ref{T2}.
\end{lemma}

\begin{proof}
From \eqref{g2}, we have
\begin{equation} \label{Phi1}
\Phi(g)(x, v, s)\geq \frac 12\nu(v) s.
\end{equation}
Moreover, it follows from Lemma \ref{Ne1} and \eqref{tb1} that
\begin{align}
\left| N_+(g, g)(x, v) \right| &\lesssim  w^{-1}(v)\langle v\rangle^\gamma \| wg\|_\infty^2; \label{N+B} \\
|N_1(g)(x, v)| & \lesssim w^{-1}(v) \langle v\rangle^{\gamma-1} \| wg\|_\infty^2.\label{N1B}
\end{align}

For \eqref{N1x}, from \eqref{N1B}, it is obvious for $d_{(x_1, x_2)}^{-1} |x_1-x_2|(1+|v|^{-1}e^{-\frac{\nu(v)}{8|v|}d_{(x_1, x_2)}}) \geq 1$. Therefore, we consider the case where $d_{(x_1, x_2)}^{-1} |x_1-x_2|(1+|v|^{-1}e^{-\frac{\nu(v)}{8|v|}d_{(x_1, x_2)}})< 1$.
Without loss of generality, we assume that $t_b(x_1, v)\geq t_b(x_2, v)$.
By the mean value theorem, \eqref{tb2}, \eqref{N11}, \eqref{N12}, \eqref{Phi1}, and \eqref{N+B}, we have
\begin{align*}
\ &\left| N_1(g)(x_1, v)-N_1(g)(x_2, v)\right|  \\
\leq& \int_{0}^{t_b(x_2, v)} \left| e^{-\Phi(g)(x_1, v, s)}-e^{-\Phi(g)(x_2, v, s)}\right| \left|N_+(g, g)(x_1-vs, v)\right| ds \notag\\
\ & +\int_{0}^{t_b(x_2, v)}e^{-\Phi(g)(x_2, v, s)} \left| N_+(g, g)(x_1-vs, v)-N_+(g, g)(x_2-vs, v)\right|  ds \notag\\
\ & +\int_{t_b(x_2, v)}^{t_b(x_1, v)}e^{-\Phi(g)(x_1, v, s)} \left| N_+(g, g)(x_1-vs, v) \right| ds \notag \\
\lesssim &\ \frac{\nu(v)}{w(v)} \| wg\|_\infty^2  \int_0^{t_b(x_2, v)}e^{-\frac12\nu(v)s}\int_0^s \left| \varphi(g)(x_1-v\ell, v) - \varphi(g)(x_2-v\ell, v)\right| d\ell  ds \notag \\
\ & +\int_{0}^{t_b(x_2, v)}e^{-\frac12\nu(v)s} \left| N_+(g, g)(x_1-vs, v)-N_+(g, g)(x_2-vs, v)\right|  ds \notag\\
\ & +\int_{t_b(x_2, v)}^{t_b(x_1, v)}e^{-\frac12\nu(v)s} \left| N_+(g, g)(x_1-vs, v) \right| ds \notag\\
\lesssim &\ \| wg\|_\infty^2 (C_g+\| wg\|_\infty)d_{(x_1, x_2)}^{-a}|x_1-x_2|^a(1+|v|^{-a}e^{-\frac{\nu(v)}{8|v|}d_{(x_1, x_2)}}) \notag\\
\ &+ (C_g+\| wg\|_\infty) \| wg\|_\infty d_{(x_1, x_2)}^{-a}|x_1-x_2|^a(1+|v|^{-a}e^{-\frac{\nu(v)}{8|v|}d_{(x_1, x_2)}}) \notag \\
\ &+ \| wg\|_\infty^2 d_{(x_1, x_2)}^{-1}|x_1-x_2| |v|^{-1}e^{-\frac{\nu(v)}{|v|}d_{(x_1, x_2)}} \\
\lesssim &\ \widetilde C_g  \left(d_{(x_1, x_2)}^{-1} |x_1-x_2|(1+|v|^{-1}e^{-\frac{\nu(v)}{8|v|}d_{(x_1, x_2)}})\right)^{a}.
\end{align*}

For \eqref{N1v}, we only do the $|v_1|\geq |v_2|$ case.
Since \eqref{N1B}, \eqref{N1v} is obvious for $d_x^{-1}|v_1-v_2|(1+|v_{1}|^{-1}e^{-c\frac{\nu(v_{1})}{|v_{1}|}d_x}) \geq 1$, and thus we consider the case where $d_x^{-1}|v_1-v_2|(1+|v_{1}|^{-1}e^{-c\frac{\nu(v_{1})}{|v_{1}|}d_x})<1$.
Let $\bar v_2=\frac{|v_1|}{|v_2|}v_2 $. Then
\begin{align*}
\ & \left| N_1(g)(x, v_1)- N_1(g)(x, v_2)\right| \\
\leq &\ \left| N_1(g)(x, v_1)- N_1(g)(x, \bar v_2)\right| +\left| N_1(g)(x, \bar v_2)- N_1(g)(x, v_2)\right|=: I+II.
\end{align*}
Without loss of generality, we assume that $t_b(x, v_1)\geq t_b(x, \bar v_2)$.
By using the mean value theorem, \eqref{tb1}, \eqref{tb4}, \eqref{N13}, \eqref{N14}, \eqref{Phi1}, and \eqref{N+B}, we have
\begin{align*}
I &\leq \int_{0}^{t_b(x, \bar v_2)} \left| e^{-\Phi(g)(x, v_1, s)}-e^{-\Phi(g)(x, \bar v_2, s)}\right| \left|N_+(g, g)(x-v_1s, v_1)\right| ds\\
\ &\ +\int_{0}^{t_b(x, \bar v_2)}e^{-\Phi(g)(x, \bar v_2, s)} \left| N_+(g, g)(x-v_1s, v_1)-N_+(g, g)(x-\bar v_2s, \bar v_2)\right|  ds\\
\ &\ +\int_{t_b(x, \bar v_2)}^{t_b(x, v_1)}e^{-\Phi(g)(x, v_1, s)} \left| N_+(g, g)(x-v_1s, v_1) \right| ds\\
\ &\lesssim  \| wg\|_\infty^2 \int_0^{t_b(x, \bar v_2)}s e^{-\frac12\nu(v_1)s} |v_1-\bar v_2||\nabla \nu(v_1(t))| ds \\
\ &\ + \| wg\|_\infty^2  \int_0^{t_b(x, \bar v_2)}e^{-\frac12\nu(v_1)s} \left[ \int_0^s\left| \varphi(g)(x-v_1\ell, v_1) - \varphi(g)(x-\bar v_2\ell, \bar v_2)\right| d\ell \right] ds \\
\ &\ + \int_0^{t_b(x, \bar v_2)}e^{-\frac12\nu(\bar v_2)s}\left| N_+(g, g)(x-v_1s, v_1)-N_+(g, g)(x-\bar v_2s, \bar v_2) \right| ds\\
\ &\ +\| wg\|_\infty^2\langle v_1\rangle^\gamma \int_{t_b(x, \bar v_2)}^{t_b(x, v_1)}e^{-\frac12 \nu(v_1)s} ds \\
\ &\lesssim \widetilde C_g \left(d_x^{-a}|v_1-v_2|^a(1+|v_1|^{-a}e^{-c\frac{\nu(v_1)}{|v_1|}d_x}) + d_x^{-1}|v_1-v_2|\right)\\
\ &\lesssim \widetilde C_g \left(d_x^{-1}|v_1-v_2|(1+|v_1|^{-1}e^{-c\frac{\nu(v_1)}{|v_1|}d_x})\right)^a,
\end{align*}
where $v_1(t)=v_1+t(\bar v_2-v_1)$ for some $t\in[0, 1]$.
For $II$, we now have $t_b(x, v_2)\geq t_b(x, \bar v_2)$.
By using the mean value theorem, \eqref{tb1}, \eqref{tb5}, \eqref{N13}, \eqref{N14}, \eqref{Phi1}, and \eqref{N+B}, we have
\begin{align*}
II &\leq \int_{0}^{t_b(x, \bar v_2)} \left| e^{-\Phi(x, v_2, s)}-e^{-\Phi(x, \bar v_2, s)}\right| \left|N_+(g, g)(x-v_2s, v_2)\right| ds\\
\ &\ +\int_{0}^{t_b(x, \bar v_2)}e^{-\Phi(g)(x, \bar v_2, s)} \left| N_+(g, g)(x-v_2s, v_2)-N_+(g, g)(x-\bar v_2s, \bar v_2)\right|  ds\\
\ &\ +\int_{t_b(x, \bar v_2)}^{t_b(x, v_2)}e^{-\Phi(g)(x, v_2, s)} \left| N_+(g, g)(x-v_2s, v_2) \right| ds\\
\ &\lesssim \| wg\|_\infty^2 \int_0^{t_b(x, \bar v_2)}s e^{-\frac12\nu(\bar v_2)s} |v_2-\bar v_2||\nabla \nu(v_2(t))| ds \\
\ &\ + \| wg\|_\infty^2  \int_0^{t_b(x, \bar v_2)}e^{-\frac12\nu(\bar v_2)s} \left[ \int_0^s\left| \varphi(g)(x-v_2\ell, v_2) - \varphi(g)(x-\bar v_2\ell, \bar v_2)\right| d\ell \right] ds \\
\ &\ + \int_0^{t_b(x, \bar v_2)}e^{-\frac12\nu(\bar v_2)s}\left| N_+(g, g)(x-v_2s, v_2)-N_+(g, g)(x-\bar v_2s, \bar v_2) \right| ds\\
\ &\ +\| wg\|_\infty^2\langle v_2\rangle^\gamma \int_{t_b(x, \bar v_2)}^{t_b(x, v_2)}e^{-\frac12 \nu(v_2)s} ds \\
\ &\lesssim \widetilde C_g \left(d_x^{-a}|v_1-v_2|^a(1+|v_1|^{-a}e^{-c\frac{\nu(v_1)}{|v_1|}d_x}) + d_x^{-1}|v_1-v_2|\right)\\
\ &\lesssim \widetilde C_g \left(d_x^{-1}|v_1-v_2|(1+|v_1|^{-1}e^{-c\frac{\nu(v_1)}{|v_1|}d_x})\right)^a,
\end{align*}
where $v_2(t)=v_2+t(\bar v_2-v_2)$ for some $t\in[0, 1]$.
This completes the proof.
\end{proof}

\begin{lemma}\label{LN2}
Assume that $g$ satisfies \eqref{gxv} and \eqref{g2}.
Then for every $x, x_1, x_2\in\Omega$ and $v, v_1, v_2\in\mathbb{R}^3\setminus\{0\}$,
\begin{align*}
\left|N_2(g)(x_1, v)-N_2(g)(x_2, v)\right| &\lesssim \widetilde C_g \left(d_{(x_1, x_2)}^{-1} |x_1-x_2|(1+|v|^{-1}e^{-\frac{\nu(v)}{8|v|}d_{(x_1, x_2)}})\right)^{a}; \notag \\  
\left|N_2(g)(x, v_1)-N_2(g)(x, v_2)\right| &\lesssim \widetilde C_g \left(d_x^{-1}|v_1-v_2|(1+|v_{1,2}|^{-1}e^{-c\frac{\nu(v_{1,2})}{|v_{1,2}|}d_x})\right)^{a}. \notag 
\end{align*}
where $ \widetilde C_g$ is as in Theorem \ref{T2}.
\end{lemma}

\begin{proof}
Note that
\begin{equation*}
N_2(g)(x, v)=\int_0^{t_b(x, v)}e^{-\Phi(g)(x, v, s)}\int_{\mathbb{R}^3}k(v, \eta) N_1(g)(x-vs, \eta) d\eta ds.
\end{equation*}
By using \eqref{A-1}, \eqref{N1x}, and \eqref{N1v}, the proof is most same as that of Lemmas \ref{LF2} and \ref{LN1} and is omitted here.
We complete the proof.
\end{proof}

\begin{lemma}\label{LM1}
Assume that $g$ satisfies \eqref{gxv} and \eqref{g2} and $f\in L^\infty(\Gamma_-)$ satisfies \eqref{fb}.
Then for every $x, x_1, x_2\in\Omega$ and $v, v_1, v_2\in\mathbb{R}^3\setminus\{0\}$,
\begin{align}
\ &\ \left|M_1(f, g)(x_1, v)-M_1(f, g)(x_2, v)\right| \label{M1x}\\
\lesssim &\ (m+C_{f,g})\left(d_{(x_1, x_2)}^{-1} |x_1-x_2|(1+|v|^{-1}e^{-\frac{\nu(v)}{2|v|}d_{(x_1, x_2)}})\right)^{\min\{a, \beta\}}; \notag\\
\ &\ \left|M_1(f)(x, v_1)-M_1(f)(x, v_2)\right| \label{M1v}\\
\lesssim &\ (m+C_{f,g})\left( d_x^{-1} |v_1-v_2|(1+|v_{1,2}|^{-1}e^{-c\frac{\nu(v_{1,2})}{|v_{1,2}|}d_x})\right)^{\min\{a, \beta\}}, \notag
\end{align}
where $C_{f,g}$ is as in Theorem \ref{T2}.
\end{lemma}

\begin{proof}
From \eqref{g2}, one has \eqref{Phi1}. It follows from \eqref{Phi1} that
\begin{equation}\label{M1B}
\left| M_1(f, g)(x, v) \right| \leq |f|_{\infty, -}e^{-\frac12\nu(v) t_b(x, v)}\leq |f|_{\infty, -}.
\end{equation}
For \eqref{M1x}, since \eqref{M1B}, it is obvious for $d_{(x_1, x_2)}^{-1} |x_1-x_2|(1+|v|^{-1}e^{-\frac{\nu(v)}{2|v|}d_{(x_1, x_2)}})\geq 1$. Therefore, we consdier the case for $d_{(x_1, x_2)}^{-1} |x_1-x_2|(1+|v|^{-1}e^{-\frac{\nu(v)}{2|v|}d_{(x_1, x_2)}})< 1$.
Without loss of generality, we assume that $t_b(x_1, v)\geq t_b(x_2, v)$.
By using the mean value theorem, \eqref{xbx}, \eqref{tbx}, \eqref{tb2}, \eqref{fb}, \eqref{M11}, and \eqref{Phi1}, we have
\begin{align*}
\ & \left| M_1(f, g)(x_1, v)-M_1(f, g)(x_2, v)\right| \\
\leq&\ \left|f(x_b(x_1,v), v)-f(x_b(x_2,v), v)\right|e^{-\frac12\nu(v)t_b(x_1,v)}\notag \\
\ &+|f|_{\infty, -}\nu(v)\left(t_b(x_1,v)-t_b(x_2,v)\right) e^{-\frac12\nu(v)t_b(x_2, v)}\notag\\
\ &+|f|_{\infty, -} \left(\int_{t_b(x_2, v)}^{t_b(x_1, v)}\left| \varphi(g)(x_1-v\ell, v)\right| d\ell \right) e^{-\frac12\nu(v)t_b(x_2, v)} \notag \\
\ &+|f|_{\infty, -} \left(\int_0^{t_b(x_2, v)}\left| \varphi(g)(x_1-v\ell, v)-\varphi(g)(x_2-v\ell, v)\right| d\ell \right) e^{-\frac{\nu(v)}{2}t_b(x_2, v)} \notag \\
\lesssim&\ m \left(d_{(x_1, x_2)}^{-1} |x_1-x_2|\right)^\beta +|f|_{\infty, -}(1+\| wg\|_\infty)d_{(x_1, x_2)}^{-1} |x_1-x_2||v|^{-1}e^{-\frac{\nu(v)}{2|v|}d_{(x_1, x_2)}} \notag \\
\ &+|f|_{\infty, -}(C_g+\| wg\|_\infty) \left(d_{(x_1, x_2)}^{-1} |x_1-x_2|(1+|v|^{-1})\right)^{a}e^{-\frac{\nu(v)}{2|v|}d_{(x_1, x_2)}} \\
\lesssim &\ (m+C_{f,g})\left(d_{(x_1, x_2)}^{-1} |x_1-x_2|(1+|v|^{-1}e^{-\frac{\nu(v)}{2|v|}d_{(x_1, x_2)}})\right)^{\min\{a, \beta\}}.
\end{align*}

For \eqref{M1v}, we only do the $|v_1|\geq |v_2| $ case. Since \eqref{M1B}, \eqref{M1v} is obvious for $d_x^{-1}|v_1-v_2|(1+|v_{1}|^{-1}e^{-c\frac{\nu(v_{1})}{|v_{1}|}d_x})\geq 1$.
Therefore, we consider the case where $d_x^{-1}|v_1-v_2|(1+|v_{1}|^{-1}e^{-c\frac{\nu(v_{1})}{|v_{1}|}d_x})< 1$.
Let $\bar v_2:=\frac{|v_1|}{|v_2|}v_2$. Then
\begin{align*}
\ &\ \left|M_1(f, g)(x, v_1)-M_1(f, g)(x, v_2)\right| \\
\leq &\ \left| M_1(f, g)(x, v_1)-M_1(f, g)(x, \bar v_2)\right|+\left| M_1(f, g)(x, \bar v_2)-M_1(f, g)(x, v_2)\right| \\
=:& \ II+III.
\end{align*}
Note that $|v_1-\bar v_2|\leq 2|v_1-v_2|$, $|\bar v_2-v_2|\leq |v_1-v_2|$, and $x_b(x, v_2)=x_b(x, \bar v_2)$.

For $II$, we assume without loss of generality that $t_b(x, v_1) \geq t_b(x, \bar v_2)$. Then
\begin{align*}
II & \leq \left|f(x_b(x,v_1), v_1)-f(x_b(x, \bar v_2), \bar v_2)\right|e^{-\Phi(g)(x, v_1, t_b(x, v_1))}\\
\ &\quad +\left|f(x_b(x,\bar v_2), \bar v_2)\right| \left|e^{-\Phi(g)(x, v_1, t_b(x, v_1))}-e^{-\Phi(g)(x, v_1, t_b(x,\bar v_2))}\right|  \\
\ &\quad +\left|f(x_b(x,\bar v_2), \bar v_2)\right| \left|e^{-\Phi(g)(x, v_1, t_b(x, \bar v_2))}-e^{-\Phi(g)(x, \bar v_2, t_b(x,\bar v_2))}\right| \\
& =:  II_1+II_2+II_3.
\end{align*}
For $II_1$, by using \eqref{xbv},  \eqref{fb}, and \eqref{Phi1}, we have
\begin{align*}
\ II_1 &\lesssim m \left(d_x^{-1}\theta_{(v_1, \bar v_2)}+|v_1-v_2|\right)^\beta e^{-\frac12\nu(v_1)t_b(x, v_1)} \\
\ & \lesssim m \left[(d_x^{-1}|v_1-v_2|(1+|v_1|^{-1})\right)^\beta e^{-\frac{\nu(v_1)}{2|v_1|}d_x}.
\end{align*}
For $II_2$, we divide it into the following two case. \\
$(i)$ $\theta_{(v_1, \bar v_2)}\geq \frac{\pi}{3}$. By \eqref{e2} and \eqref{Phi1}, we have
\begin{align*}
II_2 &\leq |f|_{\infty, -}\left( e^{-\frac12\nu(v_1)t_b(x, v_1)}+ e^{-\frac12\nu(v_1)t_b(x, \bar v_2)}  \right) \lesssim |f|_{\infty, -}\theta_{(v_1, \bar v_2)} e^{-\frac{\nu(v_1)}{2|v_1|}d_x}\\
\ &\lesssim |f|_{\infty, -}|v_1-v_2||v_1|^{-1} e^{-\frac{\nu(v_1)}{2|v_1|}d_x}\lesssim |f|_{\infty, -}d_x^{-1}|v_1-v_2|.
\end{align*}
 $(ii)$ $\theta_{(v_1, \bar v_2)}< \frac{\pi}{3}$. One has that $|v_1(t)| \sim |v_1|$, where $v_1(t)=:v_1+t(\bar v_2-v_1)$ for $t\in [0, 1]$. By applying the mean value theorem, \eqref{e2}, \eqref{tbd}, \eqref{XY}, and \eqref{g2}, we have
\begin{align*}
II_2 &\leq |f|_{\infty, -}|v_1-\bar v_2| e^{-\Phi(g)(x, v_1, t_b(x, v_1(t)))} \\
\ & \qquad\qquad \times \left| \left[\nu(v_1)+\varphi(g)\left(x-v_1t_b(x, v_1(t)), v_1\right) \right]\nabla_v t_b(x, v_1(t)) \right| \\
\ &\lesssim |f|_{\infty, -}\langle v_1\rangle^\gamma d_x^{-1}|v_1-v_2| [t_b(x, v_1(t))]^2 e^{-\frac12\nu(v_1)t_b(x, v_1(t))}\\
\ &\lesssim  |f|_{\infty, -} d_x^{-1}|v_1-v_2|. 
\end{align*}
For $II_3$, by using the mean value theorem, \eqref{M12}, and \eqref{Phi1}, we have
\begin{align*}
II_3&\leq |f|_{\infty, -} |\nu(v_1)-\nu(\bar v_2)|t_b(x, \bar v_2)e^{-\frac12\nu(v_1)t_b(x, \bar v_2)} \\
\ &\quad +f|_{\infty, -} \left(\int_0^{t_b(x, \bar v_2)}|\varphi(g)(x-v_1\ell, v_1)-\varphi(g)(x-\bar v_2\ell, \bar v_2)| d\ell\right) e^{-\frac12\nu(v_1)t_b(x, \bar v_2)} \\
\ &\lesssim |f|_{\infty, -}|v_1-\bar v_2| |\nabla \nu(v_1(t))|t_b(x, \bar v_2)e^{-\frac12\nu(v_1)t_b(x, \bar v_2)}\\
\ &\quad + |f|_{\infty, -}\left( C_g\left[ d_x^{-1}|v_1-v_2|(1+|v_1|^{-1})\right]^a e^{-c\frac{\nu(v_1)}{|v_1|}d_x}+\| wg\|_\infty|v_1-v_2| \right) \\
\ &\lesssim |f|_{\infty, -}\left( C_g\left[ d_x^{-1}|v_1-v_2|(1+|v_1|^{-1})\right]^a e^{-c\frac{\nu(v_1)}{|v_1|}d_x}+(1+\| wg\|_\infty)|v_1-v_2| \right)
\end{align*}

For $III$, we have $t_b(x, v_2)\geq t_b(x, \bar v_2)$. Then
\begin{align*}
III & \leq \left|f(x_b(x,v_2), v_2)-f(x_b(x, \bar v_2), \bar v_2)\right|e^{-\Phi(x, v_2, t_b(x, v_2))}\\
\ &\quad +\left|f(x_b(x,\bar v_2), \bar v_2)\right| \left|e^{-\Phi(x, v_2, t_b(x, v_2))}-e^{-\Phi(x, v_2, t_b(x,\bar v_2))}\right|  \\
\ &\quad +\left|f(x_b(x,\bar v_2), \bar v_2)\right| \left|e^{-\Phi(x, v_2, t_b(x, \bar v_2))}-e^{-\Phi(x, \bar v_2, t_b(x,\bar v_2))}\right| \\
\ &=: III_1+III_2+III_3.
\end{align*}
For $III_1$, by using \eqref{fb} and \eqref{Phi1}, we have
\begin{align*}
III_1\leq m\left(|x_b(x, v_2)-x_b(x, \bar v_2)|+|v_2-\bar v_2|\right)^\beta e^{-\frac12\nu(v_2)t_b(x, v_2)} \lesssim m |v_1-v_2|^\beta.
\end{align*}
For $III_2$, we divide it into the following two cases. \\
$(i)$ $|v_2|< \frac{|v_1|}{2}$. One has $|v_1-v_2|\sim |v_1|$.
By using the mean value theorem and \eqref{Phi1}, we have
\begin{align*}
III_2 &\leq |f|_{\infty, -}\left( e^{-\frac12\nu(v_2)t_b(x, v_2)}+ e^{-\frac12\nu(v_2)t_b(x, \bar v_2)}  \right) \\
\ &\lesssim |f|_{\infty, -}\left[t_b(x, \bar v_2)\right]^{-1}\left[t_b(x, \bar v_2)\right] e^{-\frac12\nu(v_2)t_b(x, \bar v_2)}\\
\ &\lesssim |f|_{\infty, -}d_x^{-1}|v_1-v_2|.
\end{align*}
 $(ii)$ $\frac{|v_1|}{2}\leq |v_2|\leq |v_1|$. By applying the mean value theorem, \eqref{e2}, and \eqref{Phi1}, we have
\begin{align*}
III_2 &\leq |f|_{\infty, -}\left(\frac{\left| \overline{xx_b(x, v_2)}\right|}{|v_2|}- \frac{\left| \overline{xx_b(x, v_2)}\right|}{|\bar v_2|}\right) \left|\nu(v_2)+\varphi(g)(x-v_2\tilde t) \right| e^{-\Phi(g)(x, v_2, \tilde t)} \\
\ &\lesssim |f|_{\infty, -}\langle v_2\rangle^\gamma |v_1-v_2| \frac{\left| \overline{xx_b(x, v_2)} \right|}{|v_1||v_2|} e^{-\frac{\nu(v_2)}{|v_1|}\left| \overline{xx_b(x, v_2)} \right|}\\
\ &\lesssim  |f|_{\infty, -} d_x^{-1}|v_1-v_2|, 
\end{align*}
where $t_b(x, \bar v_2)\leq \tilde t\leq t_b(x, v_2)$.
For $III_3$, by using the mean value theorem, \eqref{M13}, and \eqref{Phi1}, we have
\begin{align*}
III_3&\leq |f|_{\infty, -} |\nu(v_2)-\nu(\bar v_2)|t_b(x, \bar v_2)e^{-\frac12\nu(\bar v_2)t_b(x, \bar v_2)} \\
\ &\quad +f|_{\infty, -} \left(\int_0^{t_b(x, \bar v_2)}|\varphi(g)(x-v_2\ell, v_2)-\varphi(g)(x-\bar v_2\ell, \bar v_2)| d\ell\right) e^{-\frac12\nu(\bar v_2)t_b(x, \bar v_2)} \\
\ &\lesssim |f|_{\infty, -}\left( C_g\left[ d_x^{-1}|v_1-v_2|(1+|v_1|^{-1})\right]^a e^{-c\frac{\nu(v_1)}{|v_1|}d_x}+(1+\| wg\|_\infty)|v_1-v_2| \right).
\end{align*}
Summing up, we get \eqref{M1v}.
This completes the proof.
\end{proof}

\begin{lemma}\label{LM2}
Assume that $g$ satisfies \eqref{gxv} and \eqref{g2} and $f\in L^\infty(\Gamma_-)$ satisfies \eqref{fb}.
Then for every $x, x_1, x_2\in\Omega$ and $v, v_1, v_2\in\mathbb{R}^3\setminus\{0\}$,
\begin{align*}
\ &\ \left|M_2(f, g)(x_1, v)-M_2(f, g)(x_2, v)\right| \\ 
\lesssim &\ (m+C_{f,g})\left(d_{(x_1, x_2)}^{-1} |x_1-x_2|(1+|v|^{-1}e^{-\frac{\nu(v)}{8|v|}d_{(x_1, x_2)}})\right)^{\min\{a, \beta\}}; \notag \\  
\ &\ \left|M_2(f, g)(x, v_1)-M_2(f, g)(x, v_2)\right| \\ 
\lesssim &\ (m+C_{f,g})\left(d_x^{-1}|v_1-v_2|(1+|v_{1,2}|^{-1}e^{-c\frac{\nu(v_{1,2})}{|v_{1,2}|}d_x})\right)^{\min\{a, \beta\}}. \notag 
\end{align*}
where $C_{f,g}$ is as in Theorem \ref{T2}.
\end{lemma}

\begin{proof}
Note that
\begin{equation*}
M_2(f, g)(x, v)=\int_0^{t_b(x, v)}e^{-\Phi(g)(x, v, s)}\int_{\mathbb{R}^3}k(v, \eta) M_1(f, g)(x-vs, \eta) d\eta ds.
\end{equation*}
By using \eqref{A-1}, \eqref{M1x}, and \eqref{M1v}, the proof is most same as that of Lemmas \ref{LF2} and \ref{LN1} and is omitted here.
We complete the proof.
\end{proof}

\begin{lemma}\label{LM3}
Assume that $g$ satisfies \eqref{gxv} and \eqref{g2}.
Then for every $x, x_1, x_2\in\Omega$ and $v, v_1, v_2\in\mathbb{R}^3\setminus\{0\}$,
\begin{align}
\ &\ \left|M_3(f, g)(x_1, v)-M_3(f, g)(x_2, v)\right| \label{M3x}\\
\lesssim &\ \widetilde C_{f, g} \left(d_{(x_1, x_2)}^{-1} |x_1-x_2|(1+|v|^{-1}e^{-\frac{\nu(v)}{|v|}d_{(x_1, x_2)}})\right)^{a};\notag\\  
\ &\ \left|M_3(f, g)(x, v_1)-M_3(f, g)(x, v_2)\right| \label{M3v}\\
\lesssim &\ \widetilde C_{f, g} \left(d_x^{-1}|v_1-v_2|(1+|v_{1,2}|^{-1}e^{-c\frac{\nu(v_{1,2})}{|v_{1,2}|}d_x})\right)^{a}, \notag 
\end{align}
where $\widetilde C_{f, g}$ is as in Theorem \ref{T2}.
\end{lemma}

\begin{proof}
Define
\begin{equation*}
G(f, g)(x, v)=\int_{\mathbb{R}^3}\int_0^{t_b(x, \eta)}k(v,\eta)e^{-\Phi(g)(x, \eta, \tau)}Kf(x-\eta\tau, \eta)d\tau d\eta.
\end{equation*}
Then
\begin{equation*}
M_3(f, g)(x, v)=\int_0^{t_b(x, v)}e^{-\Phi(x, v, s)}G(f, g)(x-vs, v) ds.
\end{equation*}
It follows from the approach of studying $G(f)$ in Section 4 that
\begin{align}\label{Gfg}
\ & G(f, g)(x, v) \\
=&\int_0^\infty\int_\Omega k\left(v, X(x)\right)e^{-\Phi (g)(x, X(x), \frac{|x-y|}{\rho})}Kf\left(y, X(x)\right)\frac{\rho}{|x-y|^2}dy d\rho, \notag
\end{align}
where $X(x):=\frac{x-y}{|x-y|}\rho$ for simplicity.
To estimate \eqref{M3x}--\eqref{M3v}, we need to complete the following estimates:
\begin{align}
|G(f, g)(x_1, v)-G(f, g)(x_2, v)| &\lesssim \widetilde C_{f, g}\left( d_{(x_1, x_2)}^{-1}|x_1-x_2| \right)^a;  \label{Gfgx}  \\
 |G(f, g)(x, v_1)-G(f, g)(x, v_2)| &\lesssim \|f\|_\infty|v_1-v_2|^{a}.  \label{Gfgv}
\end{align}

For \eqref{Gfgx}, we divide the integration domain $\Omega$ into two subdomains: $D_1=\{y\in \Omega : |y-x_1|\leq |y-x_2|\}$ and $D_2=\{y\in \Omega : |y-x_1|> |y-x_2|\}$.
We shall only do the $D_1$ case since the other is similar.
Applying \eqref{Gfg} to the difference $G(f, g)(x_1, v)-G(f, g)(x_2, v)$, we have
\begin{align*}
\ &|G(f, g)(x_1, v)-G(f, g)(x_2, v)|_{D_1}\\
\leq & \int_0^\infty\int_{D_1} \left|k\left(v, X(x_1)\right)Kf\left(y, X(x_1)\right)\right| \frac{\rho}{|x_1-y|^{2}} \\
\ &\qquad\qquad\qquad \times \left| e^{-\Phi(g)(x_1, X(x_1), \frac{|x_1-y|}{\rho})} - e^{-\Phi(g)(x_2, X(x_2), \frac{|x_2-y|}{\rho})} \right| dy d\rho  \\
\ & +\int_0^\infty\int_{D_1}\rho e^{-\Phi(g)(x_2, X(x_2), \frac{|x_2-y|}{\rho})} \left|k\left(v, X(x_1)\right)Kf\left(y, X(x_1)\right)|x_1-y|^{-2} \right. \\
\ & \quad\quad \left.  -k\left(v, X(x_2)\right)Kf\left(y, X(x_2)\right) |x_2-y|^{-2}\right| dy d\rho\\
=:&\ G_1+G_2.
\end{align*}

For $G_1$, we divide again the integration domain $D_1$ into two subdomains :
\begin{align*}
D_{11}&=\{y\in\Omega : |y-x_1|\leq |y-x_2|, |y-x_2|\geq 2|x_1-x_2|\}; \\
D_{12}&=\{y\in\Omega : |y-x_1|\leq |y-x_2|< 2|x_1-x_2|\},
\end{align*}
and name the corresponding integrals as $G_{11}$ and $G_{12}$, respectively.
For $G_{11}$, by using the mean value theorem and \eqref{A-1}, we have
\begin{align*}
G_{11}\lesssim \|f\|_\infty \int_0^\infty\int_{D_{11}} \left|k\left(v, X(x_1)\right)\right| \frac{\langle\rho\rangle^{\gamma-2}\rho}{|x_1-y|^{2}}e^{-A(x_1, x_2, \rho ,y)} \sum_{i=1}^3 H_i(x_1, x_2, \rho, y)  dy d\rho
\end{align*}
where $A(x_1, x_2, \rho, y)$ is between $\Phi(g)(x_1, X(x_1), \frac{|x_1-y|}{\rho})$ and $\Phi(g)(x_2, X(x_2), \frac{|x_2-y|}{\rho})$,
\begin{align*}
H_1(x_1, x_2, \rho, y) &=\int_0^{\frac{|x_1-y|}{\rho}} \left| \varphi(g)(x_1-X(x_1)\ell, X(x_1))-\varphi(g)(x_2-X(x_2)\ell, X(x_2))\right| d\ell; \\
H_2(x_1, x_2, \rho, y) &=\int_{\frac{|x_1-y|}{\rho}}^{\frac{|x_2-y|}{\rho}} \left| \varphi(g)(x_2-X(x_2)\ell, X(x_2))\right| d\ell; \\
H_3(x_1, x_2, \rho, y) &= \left| \nu(X(x_1))\frac{|x_1-y|}{\rho}- \nu(X(x_2))\frac{|x_2-y|}{\rho} \right|.
\end{align*}
By \eqref{g2}, we now have
\begin{align}\label{A1}
A(x_1, x_2, \rho, y)\geq \min\left\{\nu\left(X(x_i)\right) \frac{|x_i-y|}{2\rho} : i=1, 2 \right\}\geq \frac{\nu_1\langle \rho\rangle^\gamma}{2\rho}|x_1-y|.
\end{align}

For $H_1$, let $\widetilde\Omega_0=\Omega$ and
\begin{align*}
\widetilde{\Omega}_n=\{x\in \Omega : d_x<2^{-n}d_{(x_1, x_2)}\min\{1, \rho\}\}, n \in\mathbb{N},
\end{align*}
and let $N\in\mathbb{N}$ be such that
\begin{equation*}
2^{-N}\leq \left(d_{(x_1, x_2)}^{-1}|x_1-x_2|(1+\rho^{-1})\right)^a.
\end{equation*}
Because of the convexity, there is a unique intersection point $X_i$ of the ray $\overrightarrow{x_i y}$ with $\partial\Omega$. Denote $x_{in}$ be the intersection point of $\overline{x_iX_i}$ with $\partial \widetilde\Omega_n \setminus \partial\Omega$.
By Proposition \ref{P3.2}, \eqref{Gk1}, \eqref{N-x}, and \eqref{N-v}, we have
\begin{equation}
|x_{in}-X_i| \lesssim 2^{-n}\min\{1, \rho\}  \label{M31}
\end{equation}
and
\begin{align}
\ &\left| \varphi(g)(x_1-X(x_1)\ell, X(x_1))-\varphi(g)(x_2-X(x_2)\ell, X(x_2))\right| \label{M32}\\
\lesssim &\ C_g|x_1-x_2|^a \langle\rho\rangle^{\gamma}d_{(x_1-X(x_1)\ell, x_2-X(x_2)\ell)}^{-a}\left(1+\frac{\rho^a\ell^a}{|x_1-y|^a}+\frac{\rho^a}{|x_1-y|^a}\right) \notag\\
\ & +\|wg\|_\infty\langle\rho\rangle^\gamma \rho \frac{|x_1-x_2|}{|x_1-y|}. \notag
\end{align}
Let $\widetilde{A}_n=\min\left\{ \frac{|x_1-x_{1n}|}{\rho}, \frac{|x_2-x_{2n}|}{\rho}\right\}$.
By applying the idea in the proof of Lemma \ref{LF2} and using \eqref{M31} and \eqref{M32}, we can obtain that for any $y\in \widetilde\Omega_n\setminus \widetilde\Omega_{n+1}$, $0\leq n\leq N$,
\begin{align*}
\ &H_1(x_1, x_2, \rho, y) \\
= & \int_0^{\widetilde{A}_1}\cdots\cdots +\int_{\widetilde{A}_1}^{\widetilde{A}_2}\cdots\cdots+ \ldots+\int_{\widetilde{A}_n}^{\frac{|x_1-y|}{\rho}}\cdots\cdots \\
\lesssim &\ C_g\langle \rho\rangle^\gamma\left( d_{(x_1, x_2)}^{-1}|x_1-x_2|(1+\rho^{-1})  \right)^\alpha\left[ \frac{|x_1-y|}{\rho}+\frac{|x_1-y|^{1-\alpha}}{\rho^{1-\alpha}} \right.\\
\ &\quad    \left. +\left(\sum_{j=1}^n 2^{(\alpha-1)j}\right)\left(1+\frac{\rho^\alpha}{|x_1-y|^{\alpha}}\right) \right] + \| wg\|_\infty\langle\rho\rangle^\gamma|x_1-x_2|,
\end{align*}
and for any $y\in\widetilde\Omega_{N+1}$,
\begin{align*}
\ &H_1(x_1, x_2, \rho, y) \\
= & \int_0^{\widetilde{A}_1}\cdots\cdots +\int_{\widetilde{A}_1}^{\widetilde{A}_2}\cdots\cdots+ \ldots+\int_{\widetilde{A}_{N+1}}^{\frac{|x_1-y|}{\rho}}\cdots\cdots \\
\lesssim &\ C_g\langle \rho\rangle^\gamma\left( d_{(x_1, x_2)}^{-1}|x_1-x_2|(1+\rho^{-1})  \right)^a\left[ \frac{|x_1-y|}{\rho}+\frac{|x_1-y|^{1-a}}{\rho^{1-a}} \right.\\
\ &\ \left. +\left(\sum_{j=1}^{N} 2^{(a-1)j}\right)\left(1+\frac{\rho^a}{|x_1-y|^{a}}\right) \right] + \| wg\|_\infty\langle\rho\rangle^\gamma|x_1-x_2| \\
\ &\ + \| wg\|_\infty \langle\rho\rangle^\gamma 2^{-N} \min\{1, \rho\}\rho^{-1}.
\end{align*}
This together with \eqref{A1} gives
\begin{align*}
\ &\ H_1(x_1, x_2, \rho, y)e^{-A(x_1, x_2, \rho, y)}\\
\lesssim&\ (C_g+\| wg\|_\infty)\langle \rho\rangle^\gamma\left( d_{(x_1, x_2)}^{-1}|x_1-x_2|(1+\rho^{-1})  \right)^a \left[ 1+\frac{\rho^a}{|x_1-y|^a} \right] e^{-c \frac{\langle \rho\rangle^\gamma}{\rho}|x_1-y|},
\end{align*}
and thus, by \eqref{A-1} and \eqref{e2}, we have
\begin{align*}
\ &\|f\|_\infty \int_0^\infty\int_{D_{11}} \left|k\left(v, X(x_1)\right)\right| \frac{\langle\rho\rangle^{\gamma-2}\rho}{|x_1-y|^{2}}e^{-A(x_1, x_2, \rho ,y)} H_1(x_1, x_2, \rho, y)  dy d\rho \\
\lesssim &\  \|f\|_\infty(C_g+\| wg\|_\infty)\left( d_{(x_1, x_2)}^{-1}|x_1-x_2|\right)^a  \\
\ &\times \int_0^\infty\int_{D_{11}} \left|k\left(v, X(x_1)\right)\right| \frac{\langle\rho\rangle^{2\gamma-2}\rho}{|x_1-y|^{2}}(1+\rho^{-1})^a \left[ 1+\frac{\rho^a}{|x_1-y|^a} \right] e^{-c \frac{\langle \rho\rangle^\gamma}{\rho}|x_1-y|} dy d\rho \\
\lesssim &\ \|f\|_\infty(C_g+\| wg\|_\infty) \left( d_{(x_1, x_2)}^{-1}|x_1-x_2|\right)^a \int_{\mathbb{R}^3}|k(v, \eta)| \frac{\langle \eta\rangle^{2\gamma-2}}{|\eta|}(1+|\eta|^{-1})^a \\
\ &\times \int_{|x_1-x_2|}^R \left[\left(\frac{|\eta|}{r}\right)^{1-\beta_1}\left(\frac{r}{|\eta|}\right)^{1-\beta_1}+\left(\frac{|\eta|}{r}\right)^{a}\left(\frac{|\eta|}{r}\right)^{\beta_2}\left(\frac{r}{|\eta|}\right)^{\beta_2} \right] e^{-c \frac{\langle \eta\rangle^\gamma}{|\eta|}r} dr d\eta \\
\lesssim &\ \|f\|_\infty(C_g+\| wg\|_\infty)\left( d_{(x_1, x_2)}^{-1}|x_1-x_2|\right)^a \\
\ &\qquad\qquad \times \int_{\mathbb{R}^3}|k(v, \eta)| \langle \eta\rangle^{2\gamma-2}(1+|\eta|^{-1})^a\left( |\eta|^{-\beta_1}+|\eta|^{a+\beta_2-1}\right) d\eta \\
\lesssim &\ \|f\|_\infty(C_g+\| wg\|_\infty)\left( d_{(x_1, x_2)}^{-1}|x_1-x_2|\right)^a, 
\end{align*}
where we have chosen $0<\beta_1<\min\{1, 3+\gamma-a\}$ and $\max\{0, |\gamma|-2\}<\beta_2<1-a$.

For $H_2$ and $H_3$, by the mean value theorem, we have
\begin{align*}
H_2(x_1, x_2, \rho, y) & \lesssim \|wg\|_\infty |x_1-x_2| \langle \rho\rangle^\gamma \rho^{-1}; \\
H_3(x_1, x_2, \rho, y) &=\left| \int_0^1\frac{d}{ds}\left[ \nu\left(X(x(s))\right)\frac{|x_1-y|}{\rho} \right] ds \right| \\
\ &\lesssim |x_1-x_2| \left(\langle\rho\rangle^\gamma \rho^{-1}+ \langle \rho\rangle^{\gamma-1}\right).
\end{align*}
By repeating the estimate for $G_{k13}$ in the proof of Lemma \ref{Gx}, we have
\begin{align*}
\ & \|f\|_\infty \int_0^\infty\int_{D_{11}} \left|k\left(v, X(x_1)\right)\right| \frac{\langle\rho\rangle^{\gamma-2}\rho}{|x_1-y|^{2}}e^{-A(x_1, x_2, \rho ,y)} \sum_{i=2}^3 H_i(x_1, x_2, \rho, y)  dy d\rho \\
\lesssim &\ (\| wg\|_\infty+1)\|f\|_\infty |x_1-x_2|^a.
\end{align*}

For $G_{12}$, by \eqref{A-1} and \eqref{A1} and by repeating the estimate for $G_{k3}$ in the proof of Lemma \ref{Gx},  we have
\begin{align*}
\ &\ G_{12}\\
\lesssim &\ \|f\|_\infty \int_0^\infty\int_{D_{12}} \left|k\left(v, X(x_1)\right)\right| \frac{\langle \rho\rangle^{\gamma-2}\rho}{|x_1-y|^{2}}\left(e^{-\nu\left(X(x_1)\right) \frac{|x_1-y|}{2\rho}}+e^{-\nu\left(X(x_2)\right) \frac{|x_2-y|}{2\rho}} \right) dy d\rho \\
\lesssim &\ \|f\|_\infty \int_0^\infty\int_{D_{12}} \left|k\left(v, X(x_1)\right)\right| \frac{\langle \rho\rangle^{\gamma-2}\rho}{|x_1-y|^{2}} e^{-c\frac{\langle\rho \rangle^\gamma}{\rho} |x_1-y|} dy d\rho \\
\lesssim &\ \|f\|_\infty|x_1-x_2|^a.
\end{align*}

For $G_2$, by \eqref{A1}, we have
\begin{align*}
\ & G_2 \\
\leq & \int_0^\infty\int_{D_1} \left| k\left(v, X(x_1)\right) \right| \frac{\rho}{|x_1-y|^2} e^{-c\frac{\langle\rho\rangle^\gamma}{\rho}|x_1-y|} \left| Kf\left(y, X(x_1)\right)-Kf\left(y, X(x_2)\right) \right| dy d\rho\\
\ &+ \int_0^\infty\int_{D_1} \rho \left| Kf\left(y, X(x_2)\right) \right|  \left| k\left(v, X(x_1)\right) e^{-\nu\left(X(x_1)\right) \frac{|x_1-y|}{2\rho}}|x_1-y|^{-2} \right. \\
\ &\qquad\qquad\qquad\qquad \left. -k\left(v,X(x_2)\right)e^{-\nu\left(X(x_2)\right) \frac{|x_2-y|}{2\rho}}|x_2-y|^{-2}\right| dy d\rho \\
\ &+\int_0^\infty\int_{D_1} \left| Kf\left(y, X(x_2)\right)k\left(v, X(x_1)\right)\right| \frac{\rho}{|x_1-y|^2}\\
\ &\qquad\qquad\qquad\qquad \times  \left| e^{-\nu\left(X(x_1)\right)\frac{|x_1-y|}{2\rho}} -e^{-\nu\left(X(x_2)\right)\frac{|x_2-y|}{2\rho}} \right| dy d\rho
\end{align*}
Then, by repeating the estimates of $G_k$ and $G_K$ in Lemma \ref{Gx} and the estimate of $G_{13}$ above, we can get
\begin{align*}
G_2 \lesssim \|f\|_\infty |x_1-x_2|^a. 
\end{align*}
Summing up, we get
\begin{align*}
\ &\left| G(f, g)(x_1, v)- G(f, g)(x_2, v) \right|  \\
\lesssim&\ \|f\|_\infty(C_g+\| wg\|_\infty)\left( d_{(x_1, x_2)}^{-1}|x_1-x_2|\right)^a +\|f\|_\infty|x_1-x_2|^a, \notag
\end{align*}
which implies \eqref{Gfgx}.

For \eqref{Gfgv}, since $G(f, g)$ is bounded, it is obvious for $|v_1-v_2|\geq 1$. Therefore, we consider the case for $|v_1-v_2|<1$. By Lemma \ref{kH} and \eqref{tb1}, we have
\begin{align*}
\ & \left| G(f, g)(x, v_1)-G(f, g)(x, v_2)\right| \\
\lesssim &\ \|f\|_\infty \int_{\mathbb{R}^3}|k(v_1, \eta)-k(v_2, \eta)|\langle \eta\rangle^{\gamma-3} d\eta \\
\lesssim &\ \|f\|_\infty |v_1-v_2|^{\widetilde\alpha_\gamma}\max\{\langle v_1\rangle^{2\gamma-4}, \langle v_2\rangle^{2\gamma-4} \} \\
\lesssim &\ \|f\|_\infty |v_1-v_2|^{a}. 
\end{align*}

Finally, by using \eqref{Gfgx}--\eqref{Gfgv} and repeating the proof of Lemma \ref{LN1}, we can get \eqref{M3x}--\eqref{M3v}.
This completes the proof.
\end{proof}

From the results of Lemmas \ref{LN1}-\ref{LM3}, we complete the proof of Theorem \ref{T2}.


\section{Proof of Theorem \ref{MT}}\label{Pf-Main}

{\bf Proof of Theorem \ref{MT}.} Let us consider the following iterative scheme
\begin{equation}\label{IS}
\begin{cases}
v\cdot\nabla_x f^{j+1}+Lf^{j+1}+\varphi(f^{j}) f^{j+1}=N_+(f^{j}, f^{j}), & \text{ in } \Omega\times\mathbb{R}^3,\\
f^{j+1}=r, & \text{ on } \Gamma_-,
\end{cases}
\end{equation}
for $j=0, 1, 2, \ldots$ and $f^0\equiv 0$.

For $j=0$, by Proposition \ref{LEU}, there is a unique solution $f^1$ to \eqref{IS} satisfying
\begin{equation}\label{f1b}
\| w f^1\|_\infty+| w f^1|_\infty\leq C_1| wr|_{\infty, -}\leq C_1\delta_0 \leq \min\{\delta_1, (2\widehat C)^{-1}\},
\end{equation}
where $\delta_1$ is as in Proposition \ref{LEU}, $\widehat C$ is as in \eqref{g2}, and $\delta_0>0$ is small enough such that $C_1\delta_0\leq \min\{\delta_1, (2\widehat C)^{-1}\}$.
Furthermore, if $\Omega$ is strictly convex and $r$ satisfies \eqref{rb}, it follows from Theorem \ref{T1} and \eqref{f1b} that
\begin{align}\label{f1h}
\ &\ |f^1(x_1, v_1)- f^1(x_2, v_2)|\\
\leq&\ 2C_2 m \left( d_{(x_1, x_2)}^{-1}(|x_1-x_2|+|v_1-v_2|)(1+|v_{1,2}|^{-1}e^{-c\frac{\nu(v_{1,2})}{|v_{1,2}|}d_{(x_1, x_2)}})\right)^{\min\{\alpha, \beta\}},\notag
\end{align}
where we choose $|wr|_{\infty, -}\leq \delta_0\leq m$.

For $j=1$, by Proposition \ref{LEU} with \eqref{f1b}, there is a unique solution $f^2$ to \eqref{IS} satisfying
\begin{align*}
\| w f^{2}\|_\infty+| w f^{2}|_{\infty, -} & \leq C_1 \left(| wr|_{\infty, -}+\|\nu^{-1} w N_+(f^{1}, f^{1})\|_\infty\right).
\end{align*}
It follows from Lemma \ref{Ne1} and \eqref{f1b} that
\begin{align}
\ &\| w f^{2}\|_\infty+| w f^{2}|_{\infty, -}  \label{f2b}\\
\leq&\ \widetilde C_1\left( | wr|_{\infty, -}+\| wf^{1}\|_\infty^2 \right) \leq \widetilde C_1|wr|_{\infty, -}\left(1+\widetilde C_1^{-1} C_1^2\delta_0\right)\notag\\
\leq&\ 2\widetilde C_1| wr|_{\infty, -}\leq 2\widetilde C_1 \delta_0 \leq \min\{\delta_1, (2\widehat C)^{-1}\}, \notag
\end{align}
where $\delta_0>0$ is small enough such that $\widetilde C_1^{-1} C_1^2\delta_0\leq 1$ and $2\widetilde C_1\delta_0\leq \min\{\delta_1, (2\widehat C)^{-1}\}$.
Furthermore, if $\Omega$ is strictly convex and $r$ satisfies \eqref{rb}, then $f^1$ satisfies \eqref{gxv} with $C_{f^1}=2C_2m$  by \eqref{f1h} and satisfies \eqref{g2} by \eqref{f1b}. It follows from Theorem \ref{T2}, \eqref{f1b}, and \eqref{f2b} that
\begin{align}
\ &\ \ |f^2(x_1, v_1)-f^2(x_2, v_2)|  \label{f2h} \\
\leq &\ \widetilde C_2\left(m+C_{f^2, f^1}+\widetilde C_{f^2, f^1}+\widetilde C_{f^1} \right) \notag \\
\ & \times \left( d_{(x_1, x_2)}^{-1}(|x_1-x_2|+|v_1-v_2|)(1+|v_{1,2}|^{-1}e^{-c\frac{\nu(v_{1,2})}{|v_{1,2}|}d_{(x_1, x_2)}})\right)^{\min\{\alpha, \beta\}} \notag \\
\leq &\ 2\widetilde C_2m \left(d_{(x_1, x_2)}^{-1}(|x_1-x_2|+|v_1-v_2|)(1+|v_{1,2}|^{-1}e^{-c\frac{\nu(v_{1,2})}{|v_{1,2}|}d_{(x_1, x_2)}})\right)^{\min\{\alpha, \beta\}}, \notag
\end{align}
where $\delta_0>0$ is small enough such that
\begin{align*}
\ & C_{f^2, f^1}+\widetilde C_{f^2, f^1}+\widetilde C_{f^1} \\
\leq&\ \delta_0\left[ (1+2\widetilde C_1)(1+2C_2m+C_1\delta_0) +2C_1C_2m+2C_1^2C_2m\delta_0+C_1^2\delta_0+C_1^3\delta_0 \right] \\
\leq &\ m.
\end{align*}

For $j=2$, by Proposition \ref{LEU} with \eqref{f2b}, there is a unique solution $f^3$ to \eqref{IS} satisfying
\begin{align}
\ &\| w f^{3}\|_\infty+| w f^{3}|_{\infty, -} \label{f3b}\\
\leq&\ C_1 \left(| wr|_{\infty, -}+\|\nu^{-1} w N_+(f^{2}, f^{2})\|_\infty\right) \leq \widetilde C_1\left( | wr|_{\infty, -}+\| wf^{2}\|_\infty^2 \right)\notag\\
\leq&\ \widetilde C_1| wr|_{\infty, -}\left(1+4\widetilde C_1\delta_0\right) \leq 2\widetilde C_1| wr|_{\infty, -}\leq 2\widetilde C_1 \delta_0 \leq \min\{ \delta_1, (2\widehat C)^{-1}\}, \notag
\end{align}
where $\delta_0>0$ is small enough such that
\begin{align}\label{D1}
4\widetilde C_1\delta_0\leq 1 \text{ and } 2\widetilde C_1 \delta_0 \leq \min\{ \delta_1, (2\widehat C)^{-1}\}.
\end{align}
Furthermore, if $\Omega$ is strictly convex and $r$ satisfies \eqref{rb}, then $f^2$ satisfies \eqref{gxv} with $C_{f^2}=2\widetilde C_2m$  by \eqref{f2h} and satisfies \eqref{g2} by \eqref{f2b}. It follows from Theorem \ref{T2}, \eqref{f2b}, and \eqref{f3b} that
\begin{align*}
\ & |f^3(x_1, v_1)-f^3(x_2, v_2)|\\
\leq &\ \widetilde C_2\left(m+C_{f^3, f^2}+\widetilde C_{f^3, f^2}+\widetilde C_{f^2} \right) \\
\ & \times \left( d_{(x_1, x_2)}^{-1}(|x_1-x_2|+|v_1-v_2|)(1+|v_{1,2}|^{-1}e^{-c\frac{\nu(v_{1,2})}{|v_{1,2}|}d_{(x_1, x_2)}})\right)^{\min\{\alpha, \beta\}} \\
\leq &\ 2\widetilde C_2m \left(d_{(x_1, x_2)}^{-1}(|x_1-x_2|+|v_1-v_2|)(1+|v_{1,2}|^{-1}e^{-c\frac{\nu(v_{1,2})}{|v_{1,2}|}d_{(x_1, x_2)}})\right)^{\min\{\alpha, \beta\}},
\end{align*}
where $\delta_0>0$ is small enough such that
\begin{align}
\ & C_{f^3, f^2}+\widetilde C_{f^3, f^2}+\widetilde C_{f^2} \label{D2}\\
\leq&\ \delta_0\left[ (1+2\widetilde C_1)(1+2\widetilde C_2m+2\widetilde C_1\delta_0) +4\widetilde C_1\widetilde C_2m+8\widetilde C_1^2\widetilde C_2m\delta_0+4\widetilde C_1^2\delta_0+8\widetilde C_1^3\delta_0 \right] \notag\\
\leq &\ m. \notag
\end{align}

By induction, we shall show that there exists a unique solution $f^{j+1}$ to \eqref{IS} for $j=2, 3, \ldots$ and $f^{j+1}$ satisfies
\begin{align}
\|  wf^{j+1}\|_\infty+| wf^{j+1}|_\infty & \leq 2\widetilde C_1 | wr|_{\infty, -}\leq 2\widetilde C_1\delta_0 \leq \min\{\delta_1, (2\widehat C)^{-1}\}. \label{i1}
\end{align}
Furthermore, if $\Omega$ is strictly convex and $r$ satisfies \eqref{rb}, then
\begin{align}\label{i2}
\ &\ |f^{j+1}(x_1, v_1)- f^{j+1}(x_2, v_2)|\\
\leq&\ 2\widetilde C_2m\left( d_{(x_1, x_2)}^{-1}(|x_1-x_2|+|v_1-v_2|)(1+|v_{1,2}|^{-1}e^{-c\frac{\nu(v_{1,2})}{|v_{1,2}|}d_{(x_1, x_2)}})\right)^{\min\{\alpha, \beta\}}.\notag
\end{align}
We have proved $j=2$.
Now we assume that \eqref{i1}--\eqref{i2} hold for $j=2, 3, \ldots, n-1$. For $j=n$, from Proposition \ref{LEU} with \eqref{i1}, there is a unique solution $f^{n+1}$ to \eqref{IS} satisfying
\begin{align}
\ &\|\widetilde w f^{n+1}\|_\infty+|\widetilde w f^{n+1}|_{\infty, -} \label{fjb}\\
\leq&\ C_1 \left(| wr|_{\infty, -}+\|\nu^{-1} w N_+(f^{n}, f^{n})\|_\infty\right) \leq \widetilde C_1\left( | wr|_{\infty, -}+\| wf^{n}\|_\infty^2 \right)\notag\\
\leq&\ \widetilde C_1| wr|_{\infty, -}\left(1+4\widetilde C_1\delta_0\right) \leq 2\widetilde C_1| wr|_{\infty, -}\leq 2\widetilde C_1 \delta_0 \leq \min\{ \delta_1, (2\widehat C)^{-1}\}, \notag
\end{align}
where $\delta_0>0$ satisfies \eqref{D1}.
Furthermore, if $\Omega$ is strictly convex, then we have that $f^n$ satisfies \eqref{gxv} with $C_{f^n}=2\widetilde C_2m$ by \eqref{i2} and satisfies \eqref{g2} by \eqref{i1}. It follows from Theorem \ref{T2}, \eqref{i1}, and \eqref{fjb} that
\begin{align*}
\ & |f^{n+1}(x_1, v_1)-f^{n+1}(x_2, v_2)|\\
\leq &\ \widetilde C_2\left(m+C_{f^{n+1}, f^n}+\widetilde C_{f^{n+1}, f^n}+\widetilde C_{f^n} \right) \\
\ & \times \left(d_{(x_1, x_2)}^{-1}(|x_1-x_2|+|v_1-v_2|)(1+|v_{1,2}|^{-1}e^{-c\frac{\nu(v_{1,2})}{|v_{1,2}|}d_{(x_1, x_2)}})\right)^{\min\{\alpha, \beta\}} \\
\leq &\ 2\widetilde C_2m \left( d_{(x_1, x_2)}^{-1}(|x_1-x_2|+|v_1-v_2|)(1+|v_{1,2}|^{-1}e^{-c\frac{\nu(v_{1,2})}{|v_{1,2}|}d_{(x_1, x_2)}})\right)^{\min\{\alpha, \beta\}},
\end{align*}
where $\delta_0>0$ satisfies \eqref{D2}.
Hence we now have \eqref{i1}--\eqref{i2} by induction.

It is easy to see that the difference $\bar f^j=f^{j+1}-f^j$ satisfies the following problem
\begin{equation}\label{IS-}
\begin{cases}
v\cdot\nabla_x \bar f^j+L\bar f^j +\varphi(f^{j})\bar f^j=-\varphi(\bar f^{j-1})f^j+N_+(\bar f^{j-1}, f^{j})+N_+(f^{j-1}, \bar f^{j-1}), \\
\bar f^j |_{\Gamma_-}=0.
\end{cases}
\end{equation}
Applying Proposition \ref{LEU} to \eqref{IS-} and using \eqref{i1}, we obtain
\begin{align}
\ & \| w \bar f^j\|_\infty+| w \bar f^j|_\infty \label{i-}\\
\leq &\ C_3\left( \|\nu^{-1} w\varphi(\bar f^{j-1})f^j \|_\infty +\|\nu^{-1} wN_+(\bar f^{j-1}, f^{j})\|_\infty+\|\nu^{-1} w N_+(f^{j-1}, \bar f^{j-1})\|_\infty  \right)\notag \\
\leq &\ \widetilde C_3\left( \| wf^j\|_\infty \| w\bar f^{j-1}\|_\infty + \| wf^{j-1}\|_\infty \| w\bar f^{j-1}\|_\infty \right) \notag\\
\leq &\ 2\widetilde C_1\widetilde C_3\delta_0 \| w\bar f^{j-1}\|_\infty \notag \\
\leq &\ \frac12 \| w\bar f^{j-1}\|_\infty, \notag
\end{align}
where $\delta_0>0$ is small enough such that $2\widetilde C_1\widetilde C_3\delta_0\leq 1/2$.
This implies that $\{f^j\}$ is a Cauchy sequence in $L^\infty$. Therefore, we can obtain the solution by taking the limit $f_*=\lim_{j \to \infty}f^j$.
The uniqueness is also obtained by using \eqref{i-}.

The proof of the non-negativity of the solution is most same as the argument in \cite[Lemma 3.1 and Theorem 3.1]{WZW} and is omitted here for simplicity of presentation.

If $\Omega$ is strictly convex, the continuity of $f_*$ is obtained by the $L^\infty$-convergence.
Moreover, if $r$ satisfies \eqref{rb}, by using \eqref{e2}, \eqref{i2}, and the $L^\infty$-convergence, we get
\begin{align*}
\left| f_*(x_1, v_1)-f_*(x_2, v_2)\right| \leq \overline C_2 m \left( d_{(x_1, x_2)}^{-2}(|x_1-x_2|+|v_1-v_2|)\right)^{\min\{\alpha, \beta\}}
\end{align*}
for $x_1, x_2\in\Omega$ and $v_1, v_2\in\mathbb{R}^3\setminus\{0\}$. Finally, by continuity of $f_*$, we can get
\begin{align*}
\left| f_*(x, v)-f_*(x, 0)\right| \leq \overline C_2 m \left(d_x^{-2}|v|\right)^{\min\{\alpha, \beta\}}.
\end{align*}
We complete this proof of the theorem.

\section*{Acknowledgments}
K.-C. Wu is supported by the National
Science and Technology Council under the grant NSTC 112-2636-M-006-001, 112-2628-M-006-006-MY4 and National Center
for Theoretical Sciences. K.-H. Wang is supported by the National Science and
Technology Council under the grant NSTC 113-2115-M-017-001.

\end{document}